\documentclass{article}

\usepackage[margin = 1in]{geometry}
\usepackage[sort,comma,authoryear,round]{natbib}
\usepackage{graphicx}
\usepackage{subfig}
\usepackage[colorlinks,urlcolor=blue,citecolor=blue,linkcolor=blue]{hyperref}
\usepackage{algorithm}
\usepackage{algpseudocode}
\usepackage{times}

\usepackage{amsmath}
\usepackage{amssymb}
\usepackage{amsthm}

\usepackage[colorinlistoftodos]{todonotes}

\usepackage{thmtools}
\usepackage{thm-restate}

\usepackage{bm}
\usepackage{color}
\usepackage{enumitem}

\newcommand{\innp}[1]{\left\langle #1 \right\rangle}

\newcommand{\vx}{\mathbf{x}}

\newcommand{\tG}{\Tilde{G}}

\newcommand{\cx}{\mathcal{X}}
\newcommand{\cb}{\mathcal{B}}

\newcommand{\cf}{\mathcal{F}}
\newcommand{\cc}{\mathcal{C}}

\newcommand{\cs}{\mathcal{S}}

\newcommand{\vyh}{\mathbf{\hat{y}}}
\newcommand{\vyb}{\mathbf{\bar{y}}}

\newcommand{\vy}{\mathbf{y}}
\newcommand{\vz}{\mathbf{z}}

\newcommand{\vb}{\mathbf{b}}
\newcommand{\vc}{\mathbf{c}}

\newcommand{\vu}{\mathbf{u}}
\newcommand{\vd}{\mathbf{d}}
\newcommand{\vs}{\mathbf{s}}

\newcommand{\vlambda}{\bm{\lambda}}

\newcommand{\defeq}{\stackrel{\mathrm{\scriptscriptstyle def}}{=}}

\newcommand{\rr}{\mathbb{R}}

\newcommand{\relint}{\mathrm{rel.int}}
\newcommand{\aff}{\mathrm{aff}}

\newcommand{\AFW}{\mathrm{AFW}}
\newcommand{\ACC}{\mathrm{ACC}}
\newcommand{\out}{\mathrm{out}}
\newcommand{\card}[1]{\left| #1 \right|}

\newcommand*{\vertbar}{\rule[-1ex]{0.5pt}{2.5ex}}
\newcommand*{\horzbar}{\rule[.5ex]{2.5ex}{0.5pt}}

\newcommand{\norm}[1]{\left\lVert#1\right\rVert} 
\newcommand{\abs}[1]{\left\lvert#1\right\rvert}
 
\newcommand{\inner}[2]{\left\langle#1,#2\right\rangle}

\newcommand{\alglinelabel}{%
  \label
}

\makeatletter
\def\mathcolor#1#{\@mathcolor{#1}}
\def\@mathcolor#1#2#3{%
  \protect\leavevmode
  \begingroup
    \color#1{#2}#3%
  \endgroup
}
\makeatother

\newcommand*{\vsepfbox}[1]{%
  \begingroup
    \sbox0{\fbox{#1}}%
    \setlength{\fboxrule}{0pt}%
    \mbox{\kern-\fboxsep\fbox{\unhbox0}\kern-\fboxsep}%
  \endgroup
}

\theoremstyle{plain} \numberwithin{equation}{section}
\newtheorem{theorem}{Theorem}[section]
\numberwithin{theorem}{section}
\newtheorem{corollary}[theorem]{Corollary}

\newtheorem{proposition}[theorem]{Proposition}

\theoremstyle{definition}
\newtheorem{definition}[theorem]{Definition}

\newtheorem{remark}[theorem]{Remark}

\newtheorem{assumption}[theorem]{Asssumption}
\newtheorem{fact}[theorem]{Fact}
\newtheorem{oracle}[theorem]{Oracle}

\DeclareMathOperator*{\argmin}{argmin}
\DeclareMathOperator*{\argmax}{argmax}
\DeclareMathOperator{\vertex}{\mathrm{vert}}
\DeclareMathOperator{\co}{\mathrm{co}}
\DeclareMathOperator*{\trace}{trace}

\makeatletter
\newcommand{\subalign}[1]{%
  \vcenter{%
    \Let@ \restore@math@cr \default@tag
    \baselineskip\fontdimen10 \scriptfont\tw@
    \advance\baselineskip\fontdimen12 \scriptfont\tw@
    \lineskip\thr@@\fontdimen8 \scriptfont\thr@@
    \lineskiplimit\lineskip
    \ialign{\hfil$\m@th\scriptstyle##$&$\m@th\scriptstyle{}##$\hfil\crcr
      #1\crcr
    }%
  }%
}
\makeatother

\title{Parameter-free Locally Accelerated Conditional Gradients\thanks{Authors are ordered alphabetically.}}

\author{
Alejandro Carderera \thanks{Department of Industrial and Systems Engineering, Georgia Institute of Technology. E-mail: \href{mailto:alejandro.carderera@gatech.edu}{\texttt{alejandro.carderera@gatech.edu}}} \and Jelena Diakonikolas \thanks{Department of Computer Sciences, University of Wisconsin-Madison. E-mail:  \href{mailto:jelena@cs.wisc.edu}{\texttt{jelena@cs.wisc.edu}}, \href{mailto:cylin@cs.wisc.edu}{\texttt{cylin@cs.wisc.edu}}} \and Cheuk Yin Lin \footnotemark[3] \and Sebastian Pokutta \thanks{Institute of Mathematics, Zuse Institute Berlin and Technische Universität Berlin. E-mail: \href{mailto:pokutta@zib.de}{\texttt{pokutta@zib.de}}}
}

\date{}

\begin{document}

\maketitle

\begin{abstract}
    Projection-free conditional gradient (CG) methods are the algorithms of choice for constrained optimization setups in which projections are often computationally prohibitive but linear optimization over the constraint set remains computationally feasible. 
    Unlike in projection-based methods, globally accelerated convergence rates are in general unattainable for CG. However, a very recent work on \emph{Locally accelerated CG} (LaCG) has demonstrated that local acceleration for CG is possible for many settings of interest. The main downside of LaCG is that it requires knowledge of the smoothness and strong convexity parameters of the objective function. We remove this limitation by introducing a novel, Parameter-Free Locally accelerated CG (PF-LaCG) algorithm, for which we provide rigorous convergence guarantees. Our theoretical results are complemented by numerical experiments, which demonstrate local acceleration and showcase the practical improvements of PF-LaCG over non-accelerated algorithms, both in terms of iteration count and wall-clock time. 
\end{abstract}

\section{Introduction}

Conditional gradient (CG) (or Frank-Wolfe (FW))
methods~\citep{frank1956algorithm, levitin1966constrained} are a fundamental
class of projection-free optimization methods, most frequently used to minimize
smooth convex objective functions over constrained sets onto which projections
are computationally prohibitive; see \citet{combettes2021complexity} for an
overview. These methods have received significant recent attention in the
machine learning and optimization communities, due to the fact that they eschew
projections and produce solutions with sparse
representations~\citep{jaggi2013revisiting,garber2016faster,hazan2016variance,pok17lazy,pok18bcg,lei2019primal,tsiligkaridis2020frank,combettes20adasfw}. 

While CG methods have been applied to many different problem settings (see,
e.g.,~\citet{hazan2016variance,zhou2018limited,pedregosa2018step,lei2019primal,tsiligkaridis2020frank,dvurechensky2020self,zhang2020one,negiar2020stochastic,
carderera2020second,UniformConvexFW_2020}), in this paper, we are interested in
using CG-type methods to solve problems of the form:
\begin{equation}\label{eq:problem}\tag{P}
\min_{\vx \in \cx} f(\vx),
\end{equation} 
where $f$ is an $L$-smooth (gradient-Lipschitz) and $m$-strongly convex function and $\cx \subseteq \rr^n$ is a polytope. 

We assume that we are given access to the objective function $f$ and to the feasible set $\cx$ via the following two oracles:
\begin{oracle}[First Order Oracle (FOO)]\label{assmpt:oracle-access-f}
Given $\vx \in \cx$, the FOO returns $f(\vx)$ and $\nabla f(\vx).$ 
\end{oracle}
\begin{oracle}[Linear Minimization Oracle (LMO)]\label{assmpt:oracle-access-X}
Given $\vc \in \rr^n$, the LMO returns $\argmin\limits_{\vx \in \cx}\innp{\vc, \vx}$. 
\end{oracle}

While being of extreme practical importance, these LMO-based methods in general
do not achieve the \emph{globally} optimal rates for smooth (and possibly strongly convex) minimization that are attained by projection-based methods. In particular, LMO-based algorithms cannot converge globally faster than $O(1/k)$ for the class of smooth strongly convex functions where $k$ is the number of iterations (up to the dimension threshold $n$)~\citep{lan2013complexity,jaggi2013revisiting}. Moreover, the dependence of the convergence rate on the dimension is unavoidable in general. 

At the same time, it was shown by \citet{diakonikolas2019lacg} that optimal
rates can be obtained asymptotically, referred to as \emph{locally optimal
rates}. That is, after a finite number of iterations (independent of the target
accuracy $\epsilon$) with potentially sub-optimal convergence, optimal
convergence rates can be achieved.  While \citet{diakonikolas2019lacg} resolve
the question of acceleration for CG-type methods, it unfortunately depends on the 
knowledge of $m$ and $L$. Although the latter can be removed with the common line
search-based arguments, the knowledge of a good estimate of the former is
crucial in achieving acceleration. This makes the \emph{Locally-accelerated
Conditional Gradient (LaCG)} algorithm from \citet{diakonikolas2019lacg},
despite being of theoretical interest, of limited use in practice. Not only is
it usually hard to come by a good estimate of the parameter $m$, but working
with an estimated lower bound does not take advantage of the potentially better
local strong convexity behavior in the vicinity of an iterate. 

To remedy these
shortcomings, we devise a new \emph{Parameter-Free Locally-accelerated
Conditional Gradient algorithm (PF-LaCG)} that is based on a similar coupling
between a variant of the \emph{Away-step Frank-Wolfe (AFW)} method \citep{guelat1986some,lacoste2015global} and an accelerated
method as used in LaCG.  However, beyond this basic inspiration, not many things can be reused
from \citet{diakonikolas2019lacg}, as in order to achieve parameter-freeness, we
need to devise a completely new algorithm employing a gradient-mapping
technique that out-of-the-box is incompatible with the approach used in LaCG. 

\subsection{Contributions and Further Related Work}
Our main contributions can be summarized as follows (see Section~\ref{sec:main-tech-ideas} for a detailed overview of the main technical ideas).

\paragraph{Near-optimal and parameter-free acceleration with inexact projections.} To devise PF-LaCG, we introduce a parameter-free accelerated algorithm for smooth strongly convex optimization that utilizes \emph{inexact} projections onto low-dimensional simplices. While near-optimal (i.e., optimal up to poly-log factors) parameter-free projection-based algorithms were known in the literature prior to our work~\cite{nesterov2013methods,ito2019gradientmapping}, their reliance on exact projections which are computationally infeasible makes them unsuitable for our setting.

\paragraph{Parameter-free Locally-accelerated Conditional Gradient (PF-LaCG) algorithm.} We propose a novel, parameter-free and locally accelerated CG-type method. Up to poly-logarithmic factors, our algorithm PF-LaCG attains an optimal accelerated local rate of convergence. PF-LaCG leverages efficiently computable projections onto low dimensional simplices, but is otherwise projection-free (i.e., it does not assume access to a projection oracle for $\cx$). Local acceleration is achieved by coupling the parameter-free accelerated method with inexact projections described in the previous paragraph and AFW with a fractional exit condition \cite{kerdreux2019restartfw}. This coupling idea is inspired by the coupling between $\mu$AGD+~\citep{cohen2018acceleration} (where AGD stands for Accelerated Gradient Descent) and AFW~\citep{guelat1986some} used in~\citet{diakonikolas2019lacg}; however, most of the similarities between our work and \citet{diakonikolas2019lacg} stop there, as there are major technical challenges that have to be overcome to attain the results in the parameter-free setting.

\paragraph{Computational experiments.} We demonstrate the efficacy of PF-LaCG using numerical experiments, comparing the performance of the proposed algorithms to several relevant CG-type algorithms. The use of PF-LaCG brings demonstrably faster local convergence in primal gap with respect to both iteration count and wall-clock time.

\subsection{Outline}
In Section~\ref{sec:basics} we introduce the notation and preliminaries that are required for stating our main results. 
We then present our approach to parameter-free local acceleration in Section~\ref{sec:param-free} and derive our main results. Finally, we demonstrate the practicality of our approach with computational experiments in Section~\ref{sec:compResults}, and conclude with a discussion in Section~\ref{sec:Conclusion}.

\section{Notation and Preliminaries}
\label{sec:basics}

 We denote the unique minimizer of Problem~\eqref{eq:problem} by $\vx^*$. Let $\norm{\cdot}$ and $\innp{\cdot, \cdot}$ denote the Euclidean norm and the standard inner product, respectively. We denote the \emph{diameter} of the polytope $\cx$ by $D = \max_{\vx,\vy \in \cx} \norm{\vx - \vy}$, and its \emph{vertices} by $\vertex\left(\mathcal{X}\right)\subseteq \cx$. Given a non-empty set $\mathcal{S}\subset \rr^n$,  we denote its \emph{convex hull} by $\co\left( \mathcal{S} \right)$. For any $\vx\in \cx$ we denote by $\mathcal{F}\left( \vx\right)$ the \emph{minimal face} of $\cx$ that contains $\vx$. We call a subset of vertices $\cs \subseteq \vertex(\cx)$ a \emph{support} of $\vx \in \cx$ if $\vx$ can be expressed as a convex combination of the elements of $\cs$. A support $\cs$ of $\vx$ is a \emph{proper support} of $\vx$ if the weights associated with the convex decomposition are positive. Let $\cb(\vx,r)$ denote the ball around $\vx$ with radius $r$ with respect to $\norm{\cdot}$. We say that $\vx$ is \emph{$r$-deep} in a convex set $\cc \subseteq \rr^n$ if $\cb(\vx,r) \cap \aff(\cc) \subseteq \cc$, where $\aff(\cc)$ denotes the smallest affine space that contains $\cc$. The point $\vx$ is contained in the \emph{relative interior of $\cc$}, $\vx \in \relint(\cc)$, if there exists an $r>0$ such that $\vx$ is $r$-deep in $\cc$. 
 
 \paragraph{Measures of optimality.} The two key measures of optimality that we will use in this paper are the \emph{Strong Wolfe Gap} and the \emph{Gradient Mapping}. We define the former as:
\begin{definition}[Strong Wolfe Gap]
  \label{defn:strong-wolfe-gap}
  Given $\vx \in \cx$, the \emph{strong Wolfe gap} $w (\vx)$ of $f$ over $\cx$ is defined as
  \begin{equation*}
    w (\vx) := \min_{\cs \in \cs_{\vx}} w (\vx, \cs),
  \end{equation*}
  where $\cs_{\vx}$ denotes the set of all proper supports of $\vx$ and  $w (\vx, \cs) := \max_{\vy \in \cs, \vz \in \cx} \inner{\nabla f (\vx)}{\vy - \vz}$.
\end{definition}
 Note that any polytope $\cx$ satisfies what is known as a $\delta (\cx)$-scaling inequality, where $\delta (\cx)$ is the pyramidal width of the polytope (see, e.g., \cite{lacoste2015global, beck2017linearly, pena2019polytope, gutman2018condition}) and where the inequality is defined as:
 \begin{definition}[$\delta$-scaling inequality]\label{defn:delta-scaling}
There exists $\delta > 0$ such that for all $\vx \in \cx \setminus \vx^*$ we have that Problem~\eqref{eq:problem} satisfies $w (\vx) \ge \delta (\cx) \inner{\nabla f (\vx)}{\vx - \vx^*}/\norm{\vx - \vx^*}$.
\end{definition}

To implement a parameter-free variant of a projection-based accelerated algorithm, we will rely on a second measure of optimality, the \emph{gradient mapping}. Recall that we do \emph{not} assume the availability of projections onto the polytope $\cx$; instead, our algorithm will only rely on low-complexity projections onto simplices spanned by a small number of vertices of $\cx$ (see~\citet{diakonikolas2019lacg} for a more detailed discussion). In the following, given a convex set $\cc,$ we denote the projection of $\vx \in \rr^n$ onto $\cc$ by $P_{\cc}(\vx).$ 
\begin{definition} [Gradient mapping]
  \label{defn:gradient-mapping}
  Given a convex set $\cc \subseteq \rr^n,$ a differentiable function $f:\cc \to \rr,$  and a scalar $\eta > 0,$ the gradient mapping of $f$ w.r.t.~$\eta, \cc$ is defined by:
  $$
    G_{\eta} (\vx) = \eta \Big(\vx - P_{\cc} \Big(\vx - \frac{1}{\eta} \nabla f(\vx)\Big)\Big).
  $$
\end{definition}

The gradient mapping is a generalization of the gradient to constrained sets:
when $\cc \equiv \rr^n,$ $G_{\eta}(\vx) = \nabla f(\vx).$ The norm of the
gradient mapping can also be used as a measure of optimality: the gradient mapping
at a point $\vx$ is zero if and only if $\vx$ minimizes $f$ over $\cc;$ more
generally, a small gradient mapping norm for a smooth function implies a small
optimality gap. See~\citet[Chapter~10]{beck2017optimization} and
Appendix~\ref{appx:section:acceralated-analysis} for more useful properties.

\subsection{Assumptions}

 We make two key assumptions. The first one, the \emph{strict complementarity assumption} (Assumption~\ref{assumption:strictComplementarity}) is key to proving the convergence of the iterates to $\mathcal{F}\left( \vx^*\right)$, and is a common assumption in the Frank-Wolfe literature~\cite{guelat1986some, garber2020revisiting}, which is related to the stability of the solution with respect to noise, and rules out degenerate instances.

\begin{assumption}[Strict complementarity]\label{assumption:strictComplementarity}
We have that $\innp{\nabla f\left( \vx^*\right), \vx - \vx^*} = 0$ if and only if $\vx \in \mathcal{F} \left( \vx^*\right)$. Or stated equivalently, there exists a $\tau >0$ such that $\innp{\nabla f\left( \vx^*\right), \vx - \vx^*} \geq \tau$ for all $\vx \in \mathrm{vert} \left( \cx \right)\setminus \mathcal{F} \left( \vx^*\right)$.
\end{assumption}

Lastly, to achieve local acceleration, similar to~\citet{diakonikolas2019lacg}, we require that the optimal solution $\vx^*$ is \emph{sufficiently deep} in the relative interior of a face of $\cx$. We make the following assumption about the problem, which covers all cases of interest.

\begin{assumption}[Location of $\vx^*$]\label{assmpt:x^*-sufficiently-deep-in-X}
The optimum satisfies $\vx^* \notin \vertex \left( \cx\right)$, or conversely, there exists an $r >0$ such that $\vx^*$ is $r$-deep in a face $\mathcal{F}$ of $\cx.$ 
\end{assumption}

Note that the entire polytope $\cx$ is an $n$-dimensional face of itself, and thus Assumption~\ref{assmpt:x^*-sufficiently-deep-in-X} allows $\vx^*\in \relint(\cx)$. Note that if $\vx^*\in \vertex \left( \cx\right)$ and the strict complementarity assumption is satisfied (Assumption~\ref{assumption:strictComplementarity}), then the projection-free algorithm that will be presented in later sections will reach $\vx^*$ in a finite number of iterations (see \citet{garber2020revisiting}), and so there is no need for acceleration. For computational feasibility, we assume that $r$ is bounded away from zero and much larger than the desired accuracy $\epsilon >0$.

\section{Parameter-Free Local Acceleration}
\label{sec:param-free}
This section provides our main result: a parameter-free locally-accelerated CG method (PF-LaCG). Before delving into the technical details, we first describe the core ideas driving our algorithm and its analysis. Due to space constraints, most of the proofs are omitted and are instead provided in the supplementary material.

\subsection{Overview of Main Technical Ideas}\label{sec:main-tech-ideas}

A standard idea for achieving acceleration in smooth and strongly convex setups where $m$ is unknown is to use a restart-based strategy, which can be described as follows: run an accelerated method for smooth (non-strongly) convex minimization and restart it every time some measure of optimality is reduced by a constant factor. The measures of optimality used in these strategies are either $f(\vx) - f(\vx^*)$~\cite{roulet2020sharpness} or $\|G_{\eta}(\vx)\|$~\cite{ito2019gradientmapping,nesterov2013methods}. 

Neither of these two optimality measures can be applied directly to our setting, as neither can be evaluated: (i) it is rarely the case that $f(\vx^*)$ is known, which is needed for evaluating $f(\vx) - f(\vx^*)$, and we make no such assumption here; and (ii) the gradient mapping norm $\|G_{\eta}(\vx)\|$ is a valid optimality measure only for the entire feasible set and requires computing projections onto it; and we do not assume availability of a projection operator onto $\cx$. 

Even though our algorithm utilizes a restarting-based strategy, the restarts are not used for local acceleration, but for the coupling of a CG method and a projection-based accelerated method. This idea is similar to~\citet{diakonikolas2019lacg}; however, there are important technical differences. First, the restarts in~\citet{diakonikolas2019lacg} are scheduled and parameter-based; as such, they cannot be utilized in our parameter-free setting. Our idea is to instead use $w(\vx, \cs)$ as a measure of optimality over the polytope, which is observable naturally in many CG algorithms (see Definition~\ref{defn:strong-wolfe-gap}), where $\cs$, dubbed the \emph{active set} of the CG algorithm, is a proper support of $\vx$. We perform restarts each time $w(\vx, \cs)$ is halved. The idea of using $w(\vx, \cs)$ as a measure of optimality comes from~\citet{kerdreux2019restartfw}; however in the aforementioned paper $w(\vx, \cs)$ was not used to obtain a locally accelerated algorithm.

The aforementioned idea of coupling an active-set-based CG method and an accelerated projection-based method can be summarized as follows. Because the objective function is strongly convex, any convergent algorithm (under any global optimality measure) will be reducing the distance between its iterates and the optimal solution $\vx^*$. The role of the CG method is to ensure such convergence without requiring projections onto the feasible set $\cx$. When $\vx^*$ is contained in the interior of a face of $\cx$, it can be argued that after a finite burn-in phase whose length is independent of the target accuracy $\epsilon$, every active set of the utilized CG method will contain $\vx^*$ in its convex hull. As it is possible to keep the active sets reasonably small (and their size can never be larger than the current iteration count), projections onto the convex hull of an active set can be performed efficiently, using low-complexity projections onto a probability simplex. Thus, after the burn-in phase, we could switch to a projection-based accelerated algorithm that uses the convex hull of the active set as its feasible set and attains an accelerated convergence rate from then onwards.

There are, of course, several technical challenges related to implementing such a coupling strategy. To begin with, there is no computationally feasible approach we are aware of that could allow detecting the end of the burn-in phase. Thus any reasonable locally accelerated algorithm needs to work without this information, i.e., we do not know when $\vx^*$ is contained in the convex hull of the active set. In~\citet{diakonikolas2019lacg}, this is achieved by using a parameter-based accelerated algorithm that monotonically decreases the objective value, running this accelerated algorithm and a CG method (AFW or the \emph{Pairwise-step Frank-Wolfe (PFW)}) in parallel, and, on restarts, updating the iterate of the accelerated method and the active set to the ones from the coupled CG method whenever the iterate of CG provides a lower function value. Thus, after the burn-in phase, if the feasible set of the accelerated method does not contain $\vx^*$ (or its close approximation), CG eventually constructs a point with the lower function value, after which the accelerated algorithm takes over, leading to local acceleration.

We can neither rely on the scheduled restarts nor the accelerated algorithm used in~\citet{diakonikolas2019lacg}, as both are parameter-based. Instead, our monotonic progress is w.r.t.~$w(\vx, \cs)$ (i.e., upon restarts we pick a point and the active set with the lower value of $w(\vx, \cs)$) and we rely on a parameter-free accelerated method. As mentioned before, even using a parameter-free projection-based acceleration requires new results, as we need to rely on inexact projections (and inexact gradient mappings). Further, for our argument to work, it is required that after a burn-in phase the accelerated method contracts $w(\vx, \cs)$ at an accelerated rate. Although this may seem like a minor point, we note that it is not true in general, as $w(\vx, \cs)$ can be related to other notions of optimality only when the algorithm iterates are contained in $\mathcal{F}(\vx^*)$ and the primal gap is sufficiently small. This is where the strict complementarity assumption (Assumption~\ref{assumption:strictComplementarity}) comes into play, as it allows us to show that after a burn-in phase (independent of $\epsilon$), we can upper bound $w(\vx, \cs)$ using $f(\vx) - f(\vx^*)$ (Theorem~\ref{theorem:no-more-drop-steps}).

\subsection{Burn-in Phase}

The variant of CG used in our work is the \emph{Away Frank-Wolfe} (AFW) method \cite{guelat1986some, lacoste2015global} with a stopping criterion based on halving the Frank-Wolfe gap \cite{kerdreux2019restartfw}, shown in Algorithm~\ref{alg:fafw:main}. For completeness, the useful technical results from~\citet{kerdreux2019restartfw} utilized in our analysis are provided in Appendix~\ref{appx:FAFW}.

\begin{algorithm}[htb!]
  \caption{Away-Step Frank-Wolfe Algorithm: $\mathrm{AFW} (\vx_0, \cs_0, \epsilon^w)$}
  \begin{algorithmic}[1]
    \State $k := 0$
    \While {$w (\vx_k, \cs_k) > w (\vx_0, \cs_0) / 2$} \label{alg:exit_criterion:main}
      \State $\mathbf{v}_k := \argmin_{\vu \in \cx} \inner{\nabla f (\vx_k)}{\vu}$, $\vd_k^{\mathrm{FW}} := \mathbf{v}_k - \vx_k$ \label{alg:FAFW:FWvertex:main}
      \State $\vs_k := \argmax_{\vu \in \cs_k} \inner{ \nabla f (\vx_k)}{\vu}$, $\vd_k^{\mathrm{Away}} := \vx_k - \vs_k$ \label{alg:FAFW:awayvertex:main}
      \If {$- \inner{\nabla f (\vx_k)}{\vd_k^{\mathrm{FW}}} \geq -\inner{\nabla f (\vx_k)}{\vd_k^{\mathrm{Away}}}$} \label{alg:FAFW:step_type_selection:main}
        \State $\vd_k := \vd_k^{\mathrm{FW}}$ with $\lambda_{\mathrm{max}} := 1$  \label{op:fw-step:main}
      \Else
        \State $\vd_k := \vd_k^{\mathrm{Away}}$ with $\lambda_{\mathrm{max}} := \frac{\alpha_k^{\vs_k}}{1 - \alpha_k^{\vs_k}}$  \label{op:away-step:main}
      \EndIf
      \State $\vx_{k + 1} := \vx_k + \lambda_k \vd_k$ with $\lambda_k \in [0, \lambda_{\mathrm{max}}]$ via line-search \label{alg:FAFW:stepsize}
      \State Update active set $\cs_{k + 1}$ and $\{ \alpha_{k + 1}^{\mathbf{v}} \}_{\mathbf{v} \in \cs_{k + 1}}$
      \State $k := k + 1$
    \EndWhile
    \State \Return ($\vx_k, \cs_k, w (\vx_k, \cs_k)$)
  \end{algorithmic} \label{alg:fafw:main}
\end{algorithm}

We now
briefly outline how after a finite number of iterations $T$ we can guarantee that $\vx_k \in \mathcal{F}(\vx^*)$ and $\vx^*\in \co \left( \cs_k \right)$ for $k\geq K_0$. 

In particular, using Assumption~\ref{assumption:strictComplementarity}, if the primal gap is made sufficiently small and the iterate $\vx_k$ is not contained in $\mathcal{F}\left( \vx^*\right)$, then the AFW algorithm will continuously drop those vertices in $\cs_k$ that are not in $\mathcal{F}\left( \vx^*\right)$, until the iterates reach $\mathcal{F}\left( \vx^*\right)$ (see  \citet[Theorem 5]{garber2020revisiting}, or Theorem~\ref{appx:theorem:drop-step-if-not-in-optimal-face} in Appendix~\ref{appx:section:auxiliary-results}, included for completeness).

\begin{theorem}  \label{theorem:drop-step-if-not-in-optimal-face}
  If the strict complementarity assumption is satisfied (Assumption~\ref{assumption:strictComplementarity}) and the primal gap satisfies $ f(\vx_k) -f(\vx^*) < 1/2 \min\left\{ ( \tau /(2D(\sqrt{L/m} + 1)))^2/L, \tau, LD^2 \right\} $ then the following holds for the AFW algorithm (Algorithm~\ref{alg:fafw}):
  \begin{enumerate}
      \item If $\vx_k \notin \mathcal{F}\left(\vx^* \right)$, AFW will perform an away step that drops a vertex $\vs_k \in \vertex \left( \cx \right) \setminus \mathcal{F}\left(\vx^* \right)$.
      \item If $\vx_k \in \mathcal{F}\left(\vx^* \right)$, AFW will either perform a Frank-Wolfe step with a vertex $\mathbf{v}_k \in \vertex \left( \mathcal{F}\left(\vx^* \right) \right)$ or an away-step with a vertex $\vs_k \in \vertex \left( \mathcal{F}\left(\vx^* \right) \right)$. Regardless of which step is chosen, the iterate will satisfy:
      \begin{align*}
          w (\vx_k, \cs_k) \leq LD\sqrt{2(f\left(\vx_k\right) - f\left(\vx^*\right))/m}.
      \end{align*}
  \end{enumerate}
\end{theorem}

Assuming that $\vx_0 \in \vertex \left( \cx\right)$ in the AFW algorithm, and using the primal gap convergence gap guarantee in \citet[Theorem 1]{lacoste2015global}, we can bound the number of iterations until $ f(\vx_k) -f(\vx^*)$ satisfies the requirement in Theorem~\ref{theorem:drop-step-if-not-in-optimal-face}. Using this bound, and the fact that the AFW algorithm can pick up at most one vertex per iteration, we can bound the number of iterations until $\vx_k \in \mathcal{F}(\vx^*)$. Note that by the second claim in Theorem~\ref{theorem:drop-step-if-not-in-optimal-face}, this means that when $\vx_k \in \mathcal{F}(\vx^*)$, then the iterates will not leave $\mathcal{F}(\vx^*)$. Furthermore, once the iterates are inside the optimal face, there are two options: if $\vx^* = \mathcal{F}(\vx^*)$, then the AFW algorithm will exit once $\vx_k \in \mathcal{F}(\vx^*)$, as $w(\vx_k, \cs_k) = 0$, otherwise if $\vx^* \notin \vertex \left( \cx\right)$ (the case of interest in our setting, by Assumption~\ref{assmpt:x^*-sufficiently-deep-in-X}), then we need to prove that after a given number of iterations the active set will satisfy $\vx^* \in \co \left( \cs_k\right)$. We prove the former using Fact~\ref{fact:asConvergence} (a variation of \citet[Fact B.3]{diakonikolas2019lacg}).

\begin{fact}[Critical strong Wolfe gap]
  \label{fact:asConvergence}
  There exists a $w_c > 0$ such that for any subset $\cs \subseteq \vertex(\mathcal{F}(\vx^*))$ and point $\vx \in \mathcal{F}(\vx^*)$ with $\vx \in \co(\cs)$ and $ w(\vx, \cs) \leq w_c$ it follows that $\vx^* \in \co(\cs)$.
\end{fact}

\begin{remark} 
The critical strong Wolfe gap in Fact~\ref{fact:asConvergence}, is a crucial parameter in the coming proofs. However, like the strict complementarity parameter $\tau$ \cite{guelat1986some,garber2020revisiting} and the critical radius $r_c$ defined in \citet{diakonikolas2019lacg}, the critical strong Wolfe gap can be arbitrarily small for some problems. Fortunately, as we will show in the proofs to come, it only affects the length of the burn-in phase of the accelerated algorithm, and moreover this dependence is logarithmic. In Remark~\ref{appx:remark:criticalWolfe} in the Appendix we show a simple example for which one can compute $w_c$ exactly, and we give a lower bound and an upper bound for $w_c$ for Problem~\eqref{eq:problem}.
\end{remark}

With these tools at hand, we have the bound shown in Theorem~\ref{theorem:no-more-drop-steps} (see Appendix~\ref{appx:FAFW} for the proof).

\begin{theorem} \label{theorem:no-more-drop-steps}
  Assume that the AFW algorithm (Algorithm~\ref{alg:fafw}) is run starting with $\vx_0 \in \vertex (\cx)$. If the strict complementarity assumption (Assumption~\ref{assumption:strictComplementarity}) is satisfied and $\vx^* \notin \vertex \left( \cx\right)$, then for $k \geq K_0$ with 
  $$
  K_0 = \frac{32L}{m \ln 2} \left(\frac{D}{\delta(\cx)}\right)^2 \log \left( \frac{2 w(\vx_0, \cs_0)}{\min\{ ( \tau /(2D(\sqrt{L/m} + 1)))^2/L, \tau, LD^2, 2 w_c \}}\right),
  $$
  where $\delta\left( \cx\right)$ is the pyramidal width from Definition~\ref{defn:delta-scaling} and $w_c>0$ is the critical strong Wolfe gap from Fact~\ref{fact:asConvergence},
we have that $\vx_k \in \mathcal{F}(\vx^*)$, $\vx^* \in \co \left( \cs_k \right)$. Moreover:
  \begin{align*}
   w (\vx_k, \cs_k) \leq LD\sqrt{2(f\left(\vx_k\right) - f\left(\vx^*\right))/m}.
  \end{align*}
\end{theorem}

\subsection{Parameter-free Projection-based Acceleration}

The main idea for obtaining parameter-free projection-based acceleration is to use the regularization trick of Nesterov~\cite{nesterov2012make} to obtain a near-optimal method for minimizing $\|G_{\eta}(\vx)\|$ for a smooth convex function. Then a near-optimal method for smooth strongly convex minimization is obtained by restarting this method every time $\|G_{\eta}(\vx)\|$ is halved. 

The restarting approach is important here, as it removes the requirement of knowing the parameter $m$. 
However, there are a few technical challenges in implementing the near-optimal method for minimizing $\|G_{\eta}(\vx)\|$ without the knowledge of the parameter $L$ or the distance to $\vx^*$, which is needed for setting the value of the regularization parameter. Some of these challenges were addressed in~\citet{ito2019gradientmapping,nesterov2013methods}. However, as discussed before, the approaches from~\citet{ito2019gradientmapping} and \citet{nesterov2013methods} are insufficient for our purposes, as they assume access to exact projections onto the feasible set (and thus, exact evaluations of the gradient mapping). 

In the following, we first present a near-optimal method for minimizing $\|G_{\eta}(\vx)\|$ for a smooth convex function that does not require knowledge of $L$ and works with inexact projections. We then show how to couple this method with adaptive tuning of the regularization parameter and restarts to obtain an overall near-optimal and parameter-free method for smooth strongly convex minimization. This subsection can be read independently from the rest of the paper.

\subsubsection{Small Gradient Mapping of Smooth Functions}

Let $\cc\subseteq \rr^n$ be a closed, convex, nonempty set and assume that $f: \rr^n \to \rr$ is $L$-smooth on $\cc.$ 
The rough idea of the regularization trick is the following: instead of working directly with $f,$ use a method for smooth strongly convex minimization to minimize ${f}_{\sigma}(\vx) = f(\vx) + \frac{\sigma}{2}\|\vx - \vx_0\|^2$ for some sufficiently small $\sigma > 0$ (for accuracy $\epsilon >0$, $\sigma = \Theta\left(\frac{\epsilon}{\|\vx^*_{\cc} - \vx_0\|}\right)$ suffices, where $\vx^*_{\cc} \in \argmin_{\vx \in \cc} f(\vx)$). As we select $\sigma$ ourselves, the method can assume knowledge of the strong convexity parameter $\sigma$ of ${f}_{\sigma}(\vx)$.  

The method that we employ here is a variant of $\mu$AGD+ from~\citet{cohen2018acceleration}, which is similar to~\citet{diakonikolas2019lacg}. However, unlike~\citet{cohen2018acceleration,diakonikolas2019lacg}, this method is adapted to work with an unknown smoothness parameter and to provide convergence guarantees on $\|G_{\eta}(\vx)\|$. One iteration of the algorithm is provided in AGD-Iter (Algorithm~\ref{algo:ACC-iter}). In the algorithm statement, we use $\stackrel{\epsilon}{\sim}\argmin$ to indicate that the function that follows the $\argmin$ is minimized to additive error $\epsilon.$ The main result is summarized in the following lemma, while the complete analysis is deferred to Appendix~\ref{appx:section:acceralated-analysis}.

\begin{algorithm}
  \caption{AGD-Iter($\vy_{k-1}, \mathbf{v}_{k-1}, \vz_{k-1}, A_{k-1}, \eta_k, \sigma, \epsilon_0, \eta_0$)}\label{algo:ACC-iter}
  \begin{algorithmic}[1]
  \State $\eta_k = \eta_k/2$
  \Repeat
  \State $\eta_k = 2\eta_k$
  \State $\theta_k = \sqrt{\frac{\sigma}{2(\eta_k + \sigma)}},$ $a_k = \frac{\theta_k}{1-\theta_k}A_{k-1}$ 
  \State $\epsilon_k^{\ell} = \theta_k \epsilon_0/4,$ $\epsilon_k^M = a_k \epsilon_0/4$ 
  \State $\vx_k = \frac{1}{1+\theta_k}\vy_{k-1} + \frac{\theta_k}{1+\theta_k}\mathbf{v}_{k-1}$
  \State $\vz_k = \vz_{k-1} -a_k \nabla f_{\sigma}(\vx_k) + \sigma a_k \vx_k$
  \State $\mathbf{v}_k \stackrel{\epsilon_k^M}{\sim}\argmin_{\vu \in \cc} M_k(\vu)$, where $M_k(\vu) = -\innp{\vz_k, \vu} + \frac{\sigma A_k + \eta_0}{2}\|\vu\|^2$ \label{alg:AGD-iter:Minimize_M}
  \State $\vyh_k = (1-\theta_k)\vy_{k-1} + \theta_k \mathbf{v}_k$
  \State $\vy_k \stackrel{\epsilon_k^{\ell}}{\sim}\argmin_{\vu \in \cc}\ell_k(\vu)$, where $\ell_k(\vu) = \innp{\nabla f_{\sigma}(\vyh_k), \vu - \vyh_k} + \frac{\eta_k + \sigma}{2}\|\vu - \vyh_k\|^2$ \label{alg:AGD-iter:Minimize_l}
  \Until{$f(\vyh_k) \leq f(\vx_k)+ \innp{\nabla f(\vx_k), \vyh_k - \vx_k} + \frac{\eta_k}{2}\|\vyh_k - \vx_k\|^2$ and $f(\vy_k) \leq f(\vyh_k)+ \innp{\nabla f(\vyh_k), \vy_k - \vyh_k} + \frac{\eta_k}{2}\|\vy_k - \vyh_k\|^2$}
  \State \Return $\eta_k, A_{k} = A_{k-1} + a_k, \vz_k, \mathbf{v}_k, \vyh_k, \vy_k, \tG_{\eta_k + \sigma}^{\sigma} (\vyh_k) = (\eta_k + \sigma)(\vyh_k - \vy_k)$
  \end{algorithmic}
\end{algorithm}

\begin{restatable}{lemma}{ACClemma}\label{lemma:ACC-inner}
Let $\cc \subseteq \rr^n$ be a closed convex set and let $f:\rr^n \to \rr$ be an $L$-smooth function on $\cc.$ Let $\vx_0 \in \cc$ be an arbitrary initial point, and, given $\sigma > 0,$ define $f_{\sigma}(\vx) = f(\vx) + \frac{\sigma}{2}\|\vx - \vx_0\|^2,$  $\vx^*_{\sigma} = \argmin_{\vx \in \cc}f_{\sigma}(\vx).$ Let $\vz_0 = (\eta_0 + \sigma)\vx_0 - \nabla f(\vx_0),$ $\vy_0 = \mathbf{v}_0 = \vyh_0 \stackrel{\epsilon_0^M}{\sim}\argmin_{\vu \in \cc} M_0(\vu)$, where $\epsilon_0^M > 0$, $M_0(\vu)$ is defined as in Algorithm~\ref{algo:ACC-iter}, and the estimate $\eta_0$ is doubled until $f(\vy_0) \leq f(\vx_0) + \innp{\nabla f(\vx_0), \vy_0 -\vx_0} + \frac{\eta_0}{2}\|\vy_0 - \vx_0\|^2,$ same as in Algorithm~\ref{algo:ACC-iter}. Given $\eta_0 > 0,$ sequence $\{a_k\}_{k \geq 0}$, $A_k = \sum_{i=0}^k a_i,$ $\theta_k = \frac{a_k}{A_k}$, and the sequences of errors $\{\epsilon_k^M\}_{k\geq 0}$, $\{\epsilon_k^{\ell}\}_{k\geq 0},$ let the sequences of points $\{\vyh_k, \vy_k\}_{k \geq 0}$ evolve according to Algorithm~\ref{algo:ACC-iter} for $k \geq 1$. 

If $a_0 = A_0 = 1$ and $\theta_k = \frac{a_k}{A_k} \leq \sqrt{\frac{\sigma}{2(\eta_k + \sigma)}}$ for $k \geq 1,$ then for all $k \geq 1$

$$
    \frac{1}{\eta_k + \sigma}\|G_{\eta_k + \sigma}^{\sigma}(\vyh_k)\|^2 \leq \frac{1}{A_k} \Big({\eta_0\|\vx^*_{\sigma} - \vx_0\|^2} + 2\sum_{i=0}^k (2\epsilon_i^M + \epsilon_i^{\ell})\Big),
$$

where 

$G_{\eta_k + \sigma}^{\sigma}(\vyh_k) = (\eta_k + \sigma)(\vyh_k - P_{\cc}(\vyh_k - \frac{1}{\eta_k + \sigma}\nabla f_{\sigma}(\vyh_k)))$ 

is the gradient mapping w.r.t.~$f_{\sigma}$, at $\vyh_k.$ 

In particular, if $a_0 = A_0 = 1$, $\theta_k = \frac{a_k}{A_k} = \sqrt{\frac{\sigma}{2(\eta_k + \sigma)}}$ for $k \geq 1,$ $\epsilon_k^M = \frac{a_k\epsilon^2}{8} ,$ and $\epsilon_k^{\ell} \leq \frac{a_k}{A_k}\frac{\epsilon^2}{8},$ then $\frac{1}{\eta_k + \sigma}\|G_{\eta_k}^{\sigma}(\vyh_k)\|^2 \leq \epsilon^2$ 
after at most
$$
    k = O\left(\sqrt{\frac{L}{\sigma}}\log\left(\frac{L\|\vx^*_{\sigma} - \vx_0\|}{\epsilon}\right)\right)
$$
iterations. Further, the total number of first-order queries to $f$ and oracle queries to inexact projections is at most
$$
    k' = k + 2\log\left(\frac{2L}{\eta_0}\right). 
$$
\end{restatable}

\subsubsection{Adaptive Regularization and Restarts}

We now provide the full details for the updates of the accelerated sequence. There are two main questions that need to be addressed: (1) how to adaptively adjust the regularization parameter $\sigma$ and (2) how to perform adaptive restarts based on inexact evaluations of the gradient mapping. For the former, note that, to obtain a near-optimal algorithm, $\sigma$ should not be allowed to decrease too much (and in fact, needs to satisfy $\sigma = \Omega(m)$; see the proof of Theorem~\ref{thm:ACC-full}). 

For the latter, we can rely on the strong convexity of the function $\ell_k(\vu) = \innp{\nabla f_{\sigma}(\vyh_k), \vu - \vy_k} + \frac{\eta_k + \sigma}{2}\|\vu - \vyh_k\|^2$ to relate the exact and the inexact gradient mapping at  $\vyh_k.$ 

We state the main convergence bound of the ACC algorithm (Algorithm~\ref{algo:ACC:main}) here, while the detailed description of the algorithm and its analysis are deferred to Appendix~\ref{appx:section:acceralated-analysis}, for space considerations. Note that the algorithm name stems from \emph{ACCeleration}.

\begin{algorithm}
  \caption{ACC($\vx_0, \eta_0, \sigma$)}\label{algo:ACC:main}
  \begin{algorithmic}[1]
    \State $\sigma = 2 \sigma$
    \Repeat 
    \State $\sigma = \sigma / 2$
    \State Run a minimization procedure for $\ell_0(\vu) = \innp{\nabla f(\vx_0), \vu - \vx_0} + \frac{\eta_0 + \sigma}{2}\|\vu - \vx_0\|^2.$ Halt when the current iterate $\vy$ of the procedure satisfies $\ell_0(\vy) - \min_{\vu \in \cc}\ell_0(\vu) \leq \epsilon_0,$ where $\epsilon_0 = \frac{\eta_0 + \sigma}{32}\|\vy_0 - \vx_0\|^2 = \frac{1}{32(\eta_0 + \sigma)}\|\tG_{\eta_0 + \sigma}(\vx_0)\|^2.$ 
    \State Set $\vyh_0 = \mathbf{v}_0 = \vy_0$; $\vz_0 = (\eta_0 + \sigma)\vx_0 - \nabla f(\vx_0)$
    \State $a_0 = A_0 = 1$
        \Repeat
            \State $k = k + 1$
            \State $\eta_k, A_k, \vz_k, \mathbf{v}_k, \vyh_k, \vy_k, \tG_{\eta_k + \sigma}^{\sigma}(\vyh_k) = \mathrm{AGD-Iter}(\vy_{k-1}, \mathbf{v}_{k-1}, \vz_{k-1}, A_{k-1}, \eta_{k-1}, \sigma, \epsilon_0, \eta_0)$
        \Until{$\frac{1}{\eta_k + \sigma}\|\tG_{\eta_k + \sigma}^{\sigma}(\vyh_k)\|^2 \leq \frac{9\epsilon_0}{4}$}
    \Until{$\frac{\sigma}{\sqrt{\eta_k + \sigma}}\|\vyh_k - \vx_0\| \leq \sqrt{\epsilon_0}$}
    \State\Return $\vyh_k, \eta_k \sigma$
  \end{algorithmic}
\end{algorithm}

\begin{restatable}{theorem}{thmACC}\label{thm:ACC-full}
Let $\cc \subseteq \rr^n$ be a closed convex set, and let $f:\rr^n \to \rr$ be a function that is $L$-smooth and $m$-strongly convex on $\cc.$ Let $\vx_0 \in \cc$ be an arbitrary initial point, and let $\sigma_0, \eta_0$ be such that $m \leq \sigma_0 \leq \eta_0 \leq L$. Let $\vx_0^{\mathrm{out}} = \vx_0,$ and for $k \geq 0,$ consider the following updates:
$$
    \vx_{k+1}^{\mathrm{out}}, \eta_{k+1}, \sigma_{k+1} = \mathrm{ACC}(\vx_{k}^{\mathrm{out}}, \eta_{k}, \sigma_{k}),
$$
where ACC is provided in Algorithm~\ref{algo:ACC:main}. Let $G_{\eta + \sigma}$ denote the gradient mapping of $f$ on $\cc$ with parameter $\eta + \sigma.$ Then, for any $\epsilon >0,$ $\|G_{\eta_k + \sigma_k}(\vx_k^{\mathrm{out}})\| \leq \epsilon$ for $k = O\left(\log\left(\frac{L\|G_{\eta_0 + \sigma_0}(\vx_0)\|}{m\epsilon}\right)\right).$ The algorithm for computing $\vx_k^{\mathrm{out}}$ utilizes a total number of 
$$
    K = \left(\sqrt{\frac{L}{m}}\log\left(\frac{L}{m}\right)\log\left(\frac{L}{m}\frac{\|G_{\eta_0 + \sigma_0}(\vx_0)\|}{\epsilon}\right)\right)
$$
queries to the FOO for $f$ and an inexact and efficiently computable projection oracle for $\cc,$ without knowledge of any of the problem parameters.
\end{restatable}
Observe that, up to the $\log(L/m)$ factor, the bound in Theorem~\ref{thm:ACC-full} is optimal for the class of smooth strongly convex functions, under the local oracle model (which includes the FOO model as a special case).

\subsection{Coupling of Methods and Local Acceleration}

We now show how to couple the AFW algorithm (Algorithm~\ref{alg:fafw} in Appendix~\ref{appx:FAFW}) with the algorithm described in the previous subsection to achieve local acceleration. The key observation is that after a burn-in phase whose length is independent of $\epsilon$, every active set that AFW constructs contains $\vx^* = \argmin_{\vx \in \cc} f(\vx).$ Thus, minimizing $f(\vx)$ over $\cx$ becomes equivalent to minimizing $f(\vx)$ over the convex hull of the active set of AFW. Of course, the algorithm needs to adapt to this setting without knowledge of when the burn-in phase has ended. The pseudocode for the resulting algorithm Parameter-Free Locally Accelerated Conditional Gradients (PF-LaCG) is provided in Algorithm~\ref{algo:PF_LaCG}. The coupling between AFW and ACC in PF-LaCG is illustrated on a simple example in Fig.~\ref{main:fig:pflacg}.

\begin{figure}[h]
\centering
\includegraphics[width=0.6\textwidth]{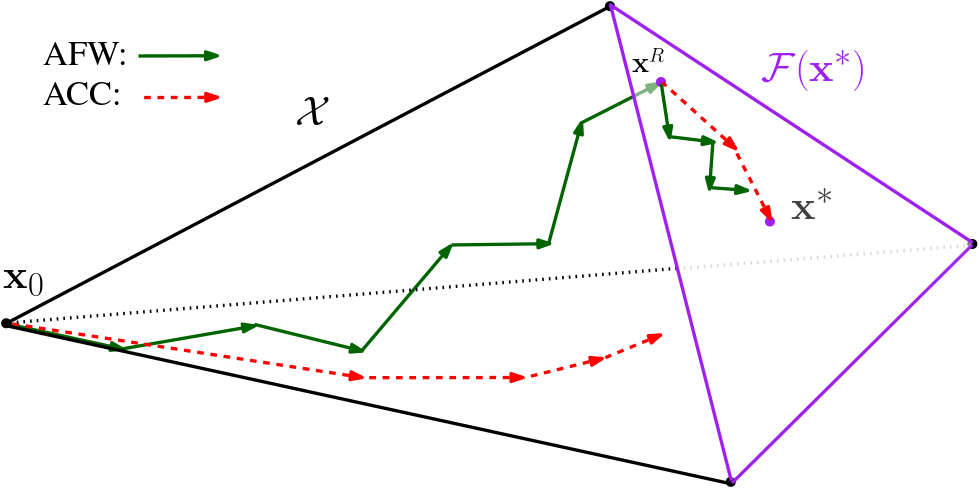} 
\caption{An example of coupling between AFW and ACC in PF-LaCG on a tetrahedron as the feasible set, 
starting from initial point $\vx_0$ with the base of the tetrahedron as its support $\cs_0$. The two algorithms are run in parallel from $\vx_0$: AFW optimizes over the entire tetrahedron, allowing it to add and remove vertices, while ACC optimizes over the base of the tetrahedron only and it cannot converge to the optimal point $\vx^*$, as $\vx^*\notin \mathrm{co} (\cs_0)$. After several iterations, once the restart criterion for AFW is triggered, PF-LaCG chooses the output point of AFW over that of ACC, as $w^{\AFW} \leq \min\{ w^{\ACC}, w^{\ACC}_{\mathrm{prev}}/2\}$, hence a PF-LaCG restart occurs at $\vx^R$. For ease of exposition we assume that the point outputted by AFW is contained in $\cf (\vx^*)$ after a single halving of $w(\vx, \cs)$, although in practice several restarts may be needed for AFW to reach $\cf (\vx^*)$. Since $\vx^R$ is on the optimal face $\cf (\vx^*)$, PF-LaCG has completed the burn-in phrase. The two algorithms again run in parallel from $\vx^R$ after the restart. However, ACC converges to the optimal $\vx^*$ at an accelerated rate, much faster than AFW. Hence, local acceleration is achieved by PF-LaCG while being at least as fast as vanilla AFW.} \label{main:fig:pflacg}
\end{figure}

\begin{algorithm}
  \caption{PF-LaCG($\vx_0 \in \vertex(\cx), \epsilon >0$)}\label{algo:PF_LaCG}
  \begin{algorithmic}[1]
    \State $\cs^{\out} = \{\vx_0\}$, $\vx^{\mathrm{out}} = \vx_0,$ $w^{\mathrm{out}} = w(\vx_0, \cs_0)$
    \State $\vx^{\AFW} = \vx^{\ACC} = \vx^{\out},$ $\cs^{\AFW} = \cs^{\ACC} = \cs^{\out}$, $w^{\AFW}_{\mathrm{prev}} = w^{\out}$
    \While{$w^{\mathrm{out}} > \epsilon$}
        \State Run AFW and restarted ACC (Theorem~\ref{thm:ACC-full}) in parallel, and let $(\vx^{\AFW}, \cs^{\AFW})$ and $(\vx^{{\ACC}}, \cs^{{\ACC}})$ denote their respective most recent output points and active sets; note that ACC is run on $\cc = \co(\cs^{{\ACC}})$ \alglinelabel{Ln:PF-LaCG-parallel}
        \If{$w^{\AFW} = w(\vx^{{\AFW}}, \cs^{\AFW}) \leq \frac{1}{2}w^{\AFW}_{\mathrm{prev}}$} \Comment{Restart criterion}\alglinelabel{Ln:PF-LaCG-restart}
            \State $w^{\AFW}_{\mathrm{prev}} = w^{\AFW},$ $w^{\ACC}_{\mathrm{prev}} = w^{\ACC},$
            \State Compute $w^{\ACC} = w(\vx^{\ACC}, \cs^{\ACC})$
            \If{$w^{\AFW} \leq \min\{ w^{\ACC}, w^{\ACC}_{\mathrm{prev}}/2\}$} \alglinelabel{Ln:PF-LaCG-ACC-update}
                \State $\vx^{\ACC} = \vx^{\AFW}$, $\cs^{\ACC} = \cs^{\AFW}$
                \State $\vx^{\out} = \vx^{\AFW},$ $\cs^{\out} = \cs^{\AFW},$ $w^{\out} = w^{\AFW}$
            \Else\If{$|\cs^{\ACC}| \leq |\cs^{\AFW}|$}
                 \State $\vx^{\AFW} = \vx^{\ACC}$,  $\cs^{\AFW} = \cs^{\ACC}$, $w^{\AFW} = w^{\ACC}$
                \EndIf
            \State $\vx^{\out} = \vx^{\ACC},$ $\cs^{\out} = \cs^{\ACC},$ $w^{\out} = w^{\ACC}$
            \EndIf
        \EndIf
    \EndWhile
    \State\Return $\vx^{\out}$
  \end{algorithmic}
\end{algorithm}

For simplicity, Algorithm~\ref{algo:PF_LaCG} is stated with the accelerated sequence started with $\vx_0\in \vertex \left( \cx \right)$ and with the active set $\cs_0 = \{\vx_0\}$. As $\cs_0$ contains only one vertex, there is no need to run the accelerated algorithm until a later restart (triggered by the condition in Line~\ref{Ln:PF-LaCG-restart}) where the active set contains more than one vertex occurs. Additionally, for the bound on the number of oracle queries in Theorem~\ref{thm:main} to hold, we need the iterations of ACC (Algorithm~\ref{algo:ACC}) and AFW (Algorithm~\ref{alg:fafw}) to be aligned, in the sense that one iteration of ACC (Algorithm~\ref{algo:ACC-iter}) occurs in parallel to one iteration of AFW. This is only needed for the theoretical bound on oracle complexity. In practice, the two algorithms can be run in parallel without aligning the iterations, so that the overall execution time (modulo coordination at restarts) is never higher than for running AFW alone. 

We are now ready to state our main result. The complete proof is deferred to Appendix~\ref{appx:section:PFLaCG}.

\begin{theorem}\label{thm:main}
  Let $\cx \subseteq \rr^n$ be a closed convex polytope of diameter $D$, and let $f:\rr^n \to \rr$ be a function that is $L$-smooth and $m$-strongly convex on $\cx.$ Let Assumptions~\ref{defn:delta-scaling} and~\ref{assumption:strictComplementarity} be satisfied for $f, \cx$. Denote $\vx^{*} = \argmin_{\vx \in \cx}f(\vx).$ Given $\epsilon > 0,$ let $\vx^{\out}$ be the output point of PF-LaCG (Algorithm~\ref{algo:PF_LaCG}), initialized at an arbitrary vertex $\vx_0$ of $\cx$. Then $w(\vx^{\out}, \cs^{\out}) \leq \epsilon$ and PF-LaCG uses a total of at most
  \begin{align*}
        K = O\bigg(\min\bigg\{& \log \left( \frac{w (\vx_0, \cs_0)}{LD^2} \right) + \frac{LD^2}{m \delta^2} \log \left( \frac{w (\vx_0, \cs_0)}{\epsilon} \right),\\
       & K_0 + K_1 + \sqrt{\frac{L}{m}}\log\left(\frac{L}{m}\right)\log\Big(\frac{LD}{m\delta}\Big)\log\left(\frac{LD}{\epsilon}\right)\bigg\}\bigg)
  \end{align*}
  queries to the FOO for $f$ and the LMO for $\cx$, where 
    \begin{align*}
  K_0 = \;&\frac{32L}{m \ln 2} \left(\frac{D}{\delta(\cx)}\right)^2\\
  &\cdot\log \Big( \frac{2 w(\vx_0, \cs_0)}{\min\{  \frac{1}{L}\Big(\frac{\tau}{(2D(\sqrt{L/m} + 1))}\Big)^2 , \tau, LD^2, 2 w_c \}}\Big),
  \end{align*}
  and $K_1 = \frac{128 LD^2}{m\delta^2}$. 
\end{theorem}

\paragraph{Strong Wolfe Gap as an Upper Bound For Primal Gap.} Note that while in Theorem~\ref{thm:main} we show a convergence rate for $w (\vx, \cs)$, this translates to the same convergence rate in primal gap, as the aforementioned quantity is an upper bound on the strong Wolfe gap and hence an upper bound on the primal gap \cite{kerdreux2019restartfw}. Furthermore, we remark that $w (\vx, \cs) = 0$ if and only if $\vx = \vx^*$ when $\cs \subseteq \vertex(\cx)$ is a proper support of $\vx$ with respect to the polytope $\cx$. In the next section we illustrate our computational results with respect to the primal gap $f(\vx) - f(\vx^*)$, as this quantity is more common in the optimization literature.

\section{Computational Experiments}\label{sec:compResults}

We numerically demonstrate that PF-LaCG (Algorithm~\ref{algo:PF_LaCG}) when implemented in Python 3 outperforms other parameter-free CG methods in both iteration count and in wall-clock time. Most notably, we compare against the AFW, PFW, \emph{Decomposition-Invariant CG} (DICG) (in the case of the probability simplex, as this algorithm only applies to $0-1$ polytopes \cite{garber2016linear}) and the \emph{Lazy AFW} (AFW (Lazy)) \cite{pok17lazy} algorithms. As predicted by our theoretical results, the improvement is observed locally, once the iterates of the algorithm reach the optimal face. Further, we empirically observe on the considered examples that the active sets do not become too large, which is important for the accelerated algorithm to have an edge over standard CG updates in the overall wall-clock time. Our code can be found at~\url{https://github.com/ericlincc/Parameter-free-LaCG}.

Similar to~\citet{diakonikolas2019lacg}, we solve the minimization subproblems (projections onto the convex hull of the active sets) within ACC (Algorithm~\ref{algo:ACC}) using Nesterov's accelerated gradient descent \citep{nesterov2018introductory} with $\mathcal{O}\left( n\log n \right)$ projections onto the simplex described in \citet[Algorithm 1]{duchi2008efficient}. Even though in our analysis we consider coupling accelerated sequence with AFW, we note that PF-LaCG can be coupled with any CG variant that maintains an active set, such as standard AFW or PFW.

Since the execution of local acceleration within PF-LaCG is completely independent of the execution of the coupled CG variant, it is possible to run the locally accelerated algorithm in parallel on a separate core within one machine or even on a secondary machine, allowing us to utilize more computational power with the goal of solving large-scale problems to high accuracy in less computing time. In our experiments, we implemented the former approach where we run each process using one CPU core, and we run the accelerated algorithm of PF-LaCG on a separate process. This approach enables us to guarantee that PF-LaCG is not slower than the CG variant such as AFW and PFW barring negligible process creation overhead and inter-process communication while achieving much faster convergence once we have reached the optimal face.

\begin{figure*}[th!]
    \captionsetup[subfigure]{labelformat=empty}
    \centering
    \vspace{-10pt}
    \subfloat{{\includegraphics[width=4.85cm]{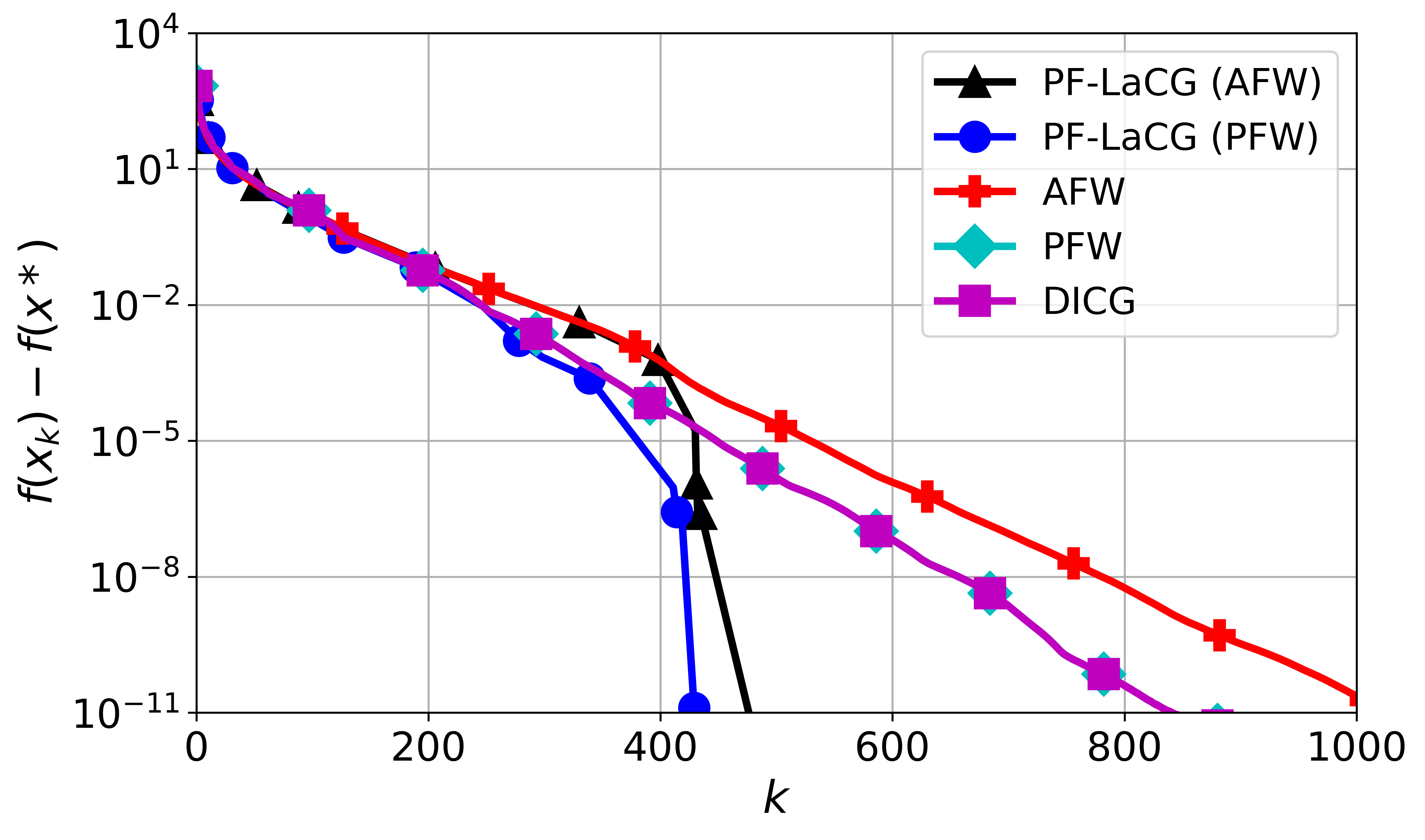}}\label{fig:simplex-PG-it}}%
    \hspace{\fill}
    \subfloat{{\includegraphics[width=4.85cm]{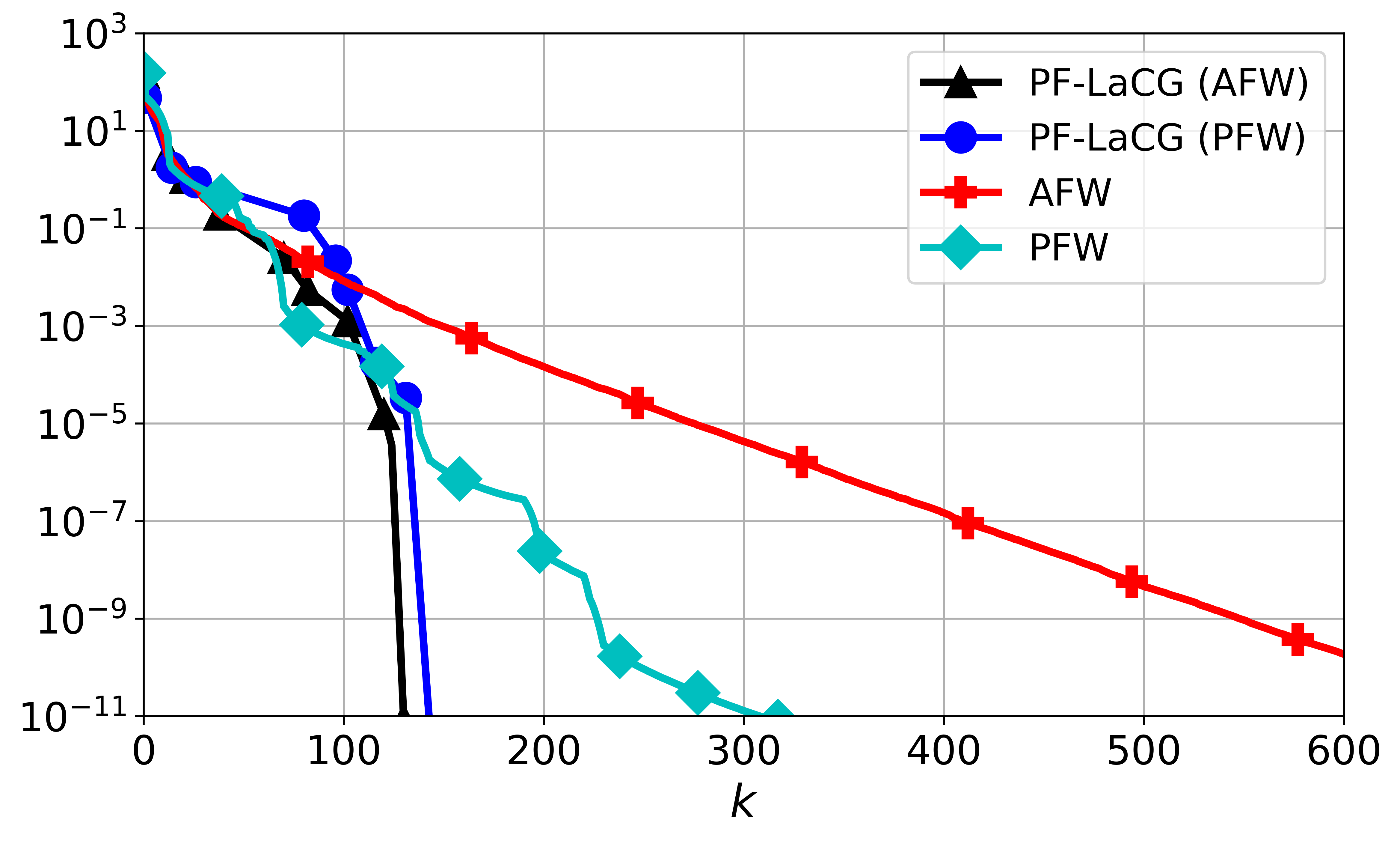} }\label{fig:lasso-PG-it}}%
    \hspace*{\fill}
    \subfloat{{\includegraphics[width=4.85cm]{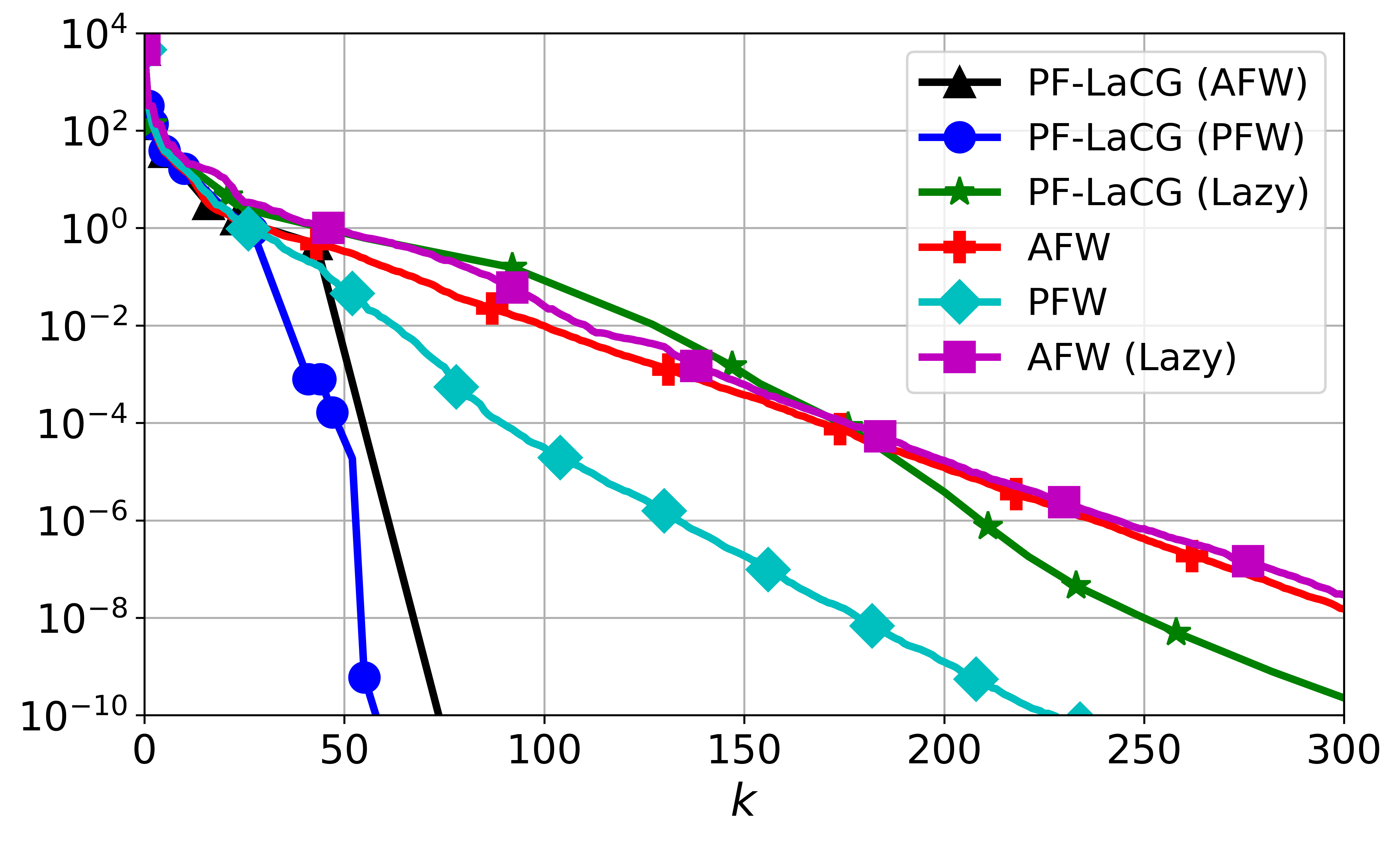} }\label{fig:matching-PG-it}}%
    \hspace*{\fill}

   \vspace{-10pt}
    \subfloat[Probability Simplex]{{\includegraphics[width=4.85cm]{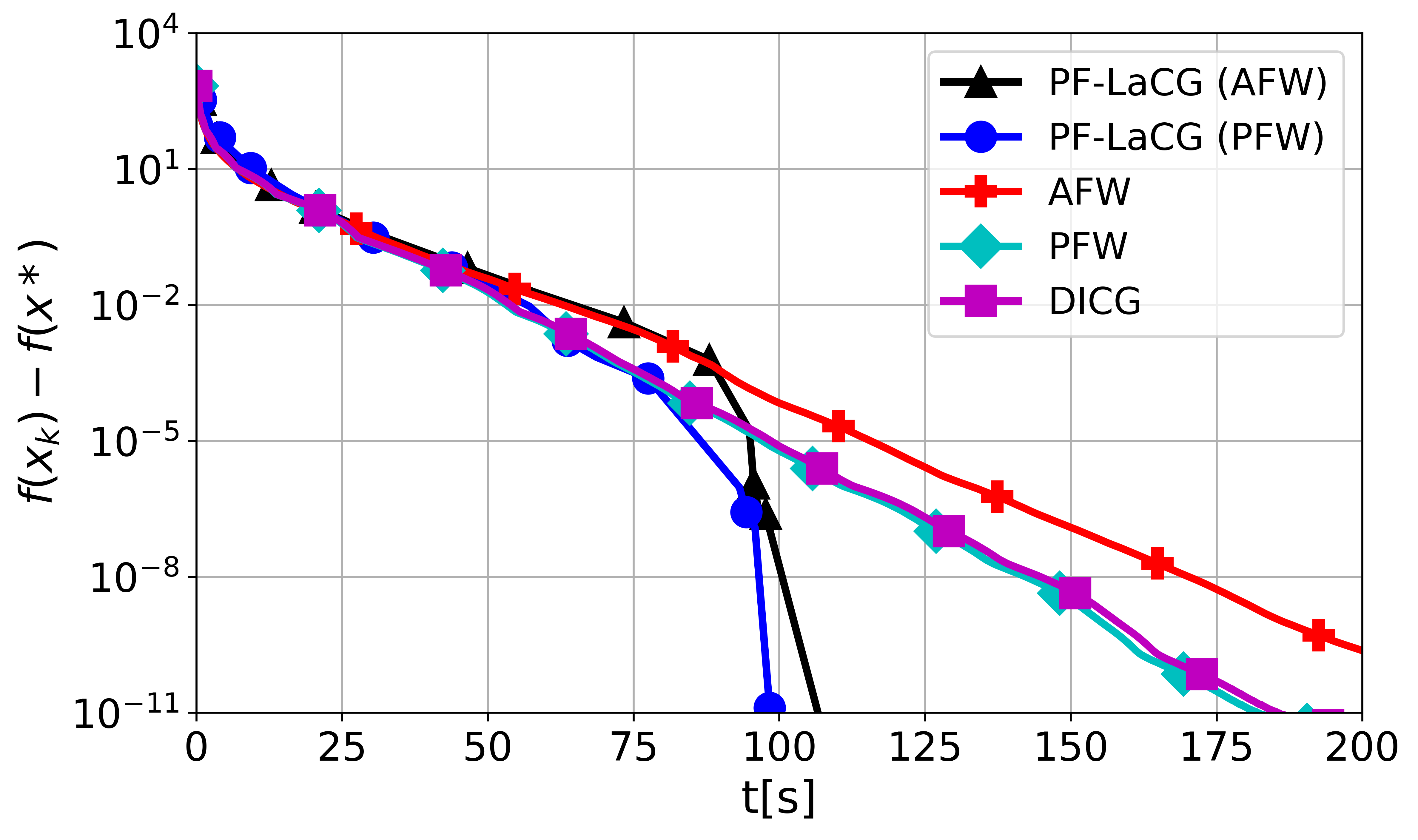} }\label{fig:simplex-PG-time}}%
    \hspace{\fill}
    \subfloat[Structured LASSO]{{\includegraphics[width=4.85cm]{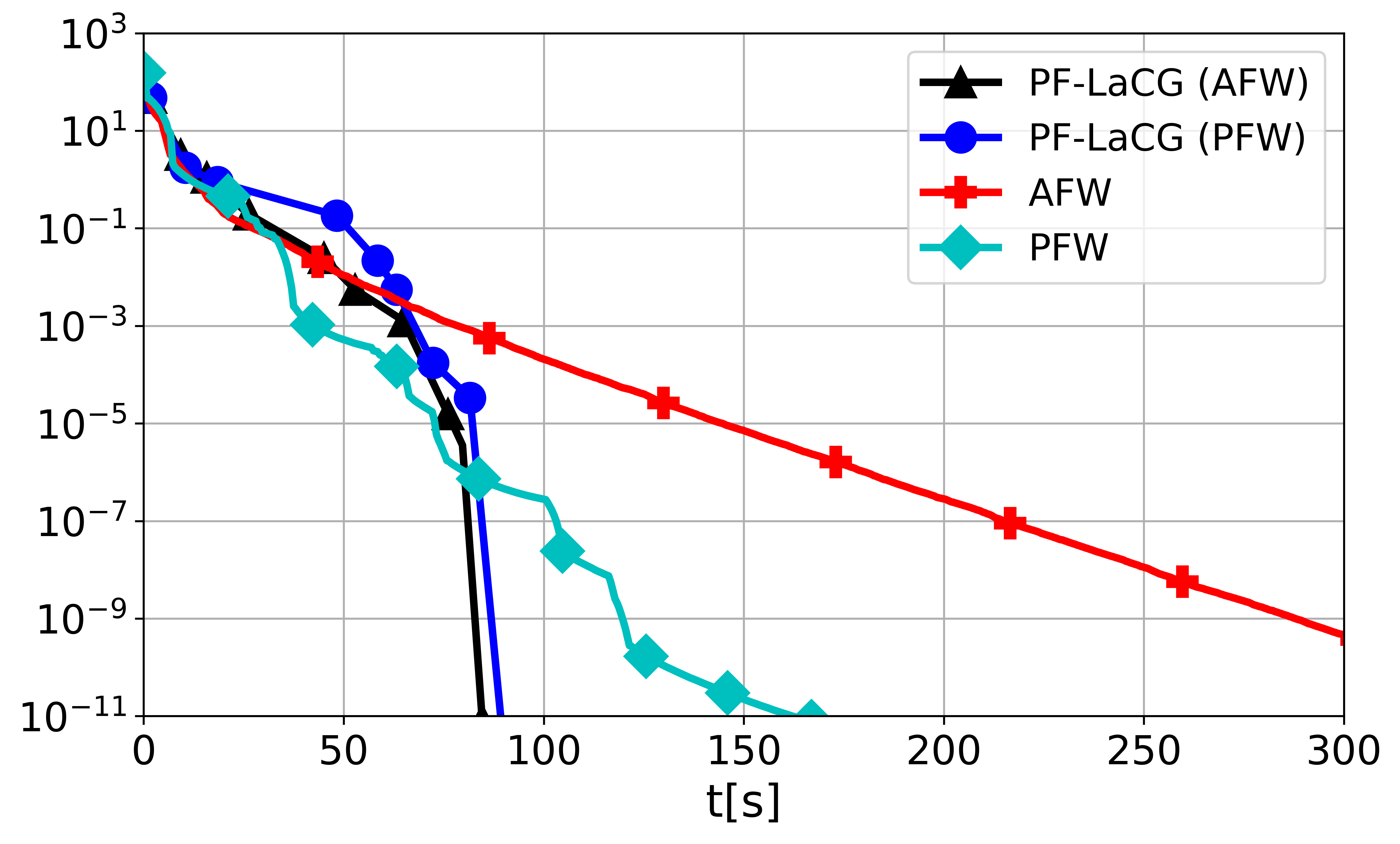} }\label{fig:lasso-PG-time}}%
    \hspace*{\fill}
    \subfloat[Constrained Birkhoff]{{\includegraphics[width=4.85cm]{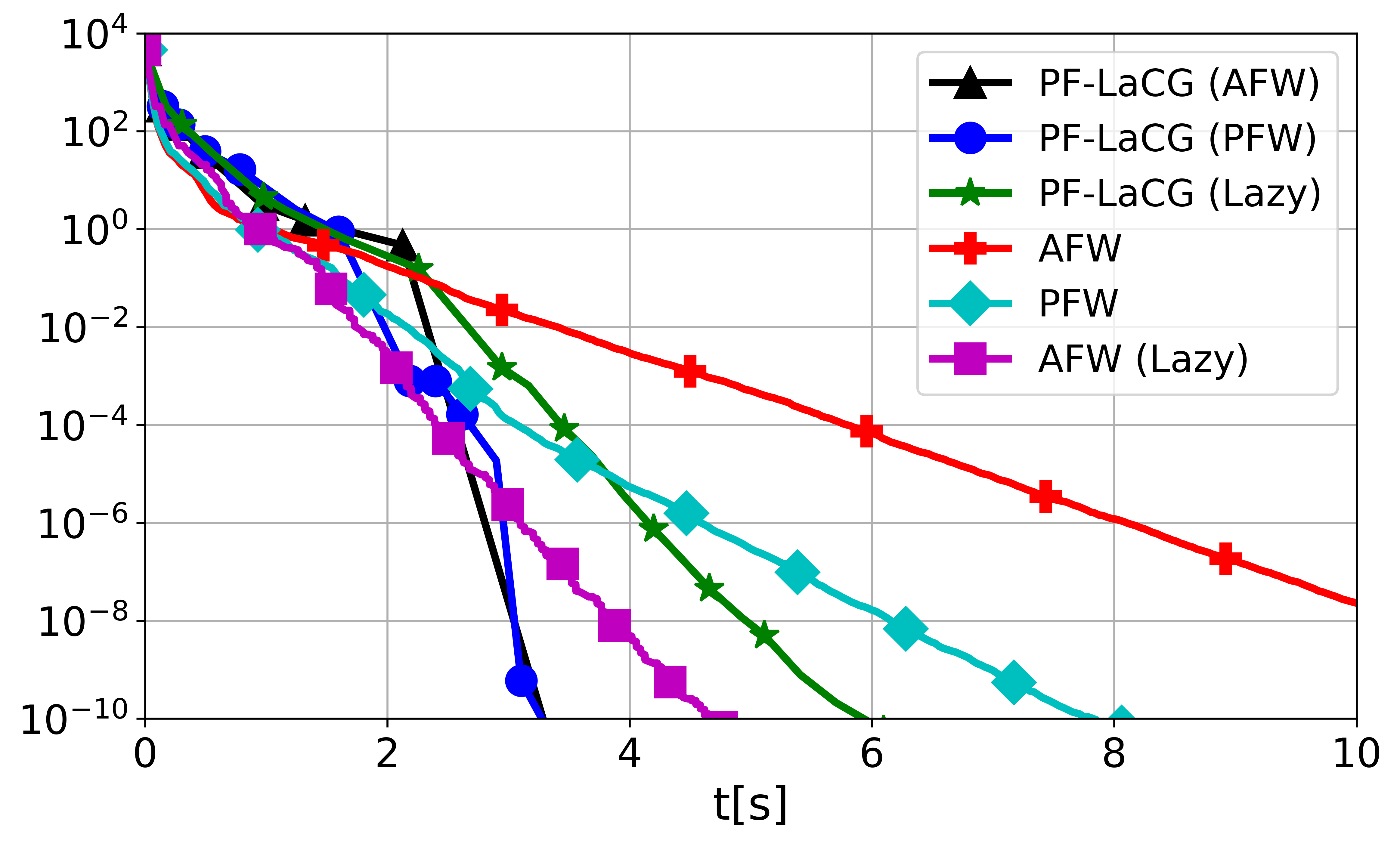} }\label{fig:matching-PG-time}}%
    \hspace*{\fill}
    \caption{\textbf{Numerical performance of PF-LaCG}: Top row depicts primal gap convergence in terms of iteration count, while bottom row depicts primal gap convergence in terms of time. Left-most column shows the results over the probability simplex, center column shows the results for the structured Lasso problem, and right-most column shows the results for the constrained Birkhoff polytope.}%
    \label{fig:PG-time}%
\end{figure*}

\paragraph{Probability Simplex.}
The unit probability simplex, although a toy example, can give us insight into the behaviour of PF-LaCG. We know that after a restart in which the AFW active set satisfies $\vx^* \in \co \left(\cs_k\right)$, we should expect to see accelerated convergence from the iterates computed by the ACC algorithm. Given the structure of the probability simplex, this is easy to check in the experiments once we have a high-accuracy solution to the minimization problem. The function being minimized in this example is $f(\vx) =  \vx^T \left( M^T M + \alpha \mathbf{1}_n \right)\vx/2 + \vb^T\vx$, where $M\in \rr^{n\times n}$ and $\vb\in \rr^{n}$ have entries sampled uniformly at random between $0$ and $1$ and $n = 10000$. The parameter $\alpha = 500$ is set so that the objective function satisfies $m \approx 500$. The resulting condition number is $L/m = 50000$, and the number of nonzero elements in $\vx^*$ is around $320$. The AFW algorithm in PF-LaCG (AFW) satisfies that $\vx^* \in \co \left( \cs_k \right)$ around iteration $400$, consequently we achieve the accelerated convergence rate from then onwards. The same can be said regarding PF-LaCG (PFW) around iteration $350$.

\paragraph{Structured LASSO Regression.}
The LASSO is an extremely popular sparse regression analysis method in statistics, and has become an important tool in ML. In many applications, we can impose additional structure on the solution being learnt through the addition of linear constraints; this is useful, for example, when learning sparse physics dynamics from data (see \citet{carderera2021cindy}), in which we can impose additional linear equality constraints on the problem to reflect the symmetries present in the physical problem. This allows us to learn dynamics that are consistent with the underlying physics, and which potentially generalize better when predicting on unseen data. The objective function is the same as in the last section, except we now have $n = 1000$, $\alpha = 100$, and we choose the elements of $\vb$ uniformly from $0$ to $100$. This results in a condition number of $L/m = 250000$. The feasible region is the intersection of the $\ell_1$ unit ball and a series of equality constraints. To generate the additional equality constraints, we sample $125$ pairs of distinct integers $(i,j)$ from $1\leq i,j \leq n$ without replacement, and we set $x_i = x_j$ for each pair, adding $125$ linear constraints. 

\paragraph{Constrained Birkhoff Polytope.}
We also solve a matching with the same quadratic function as in the previous section with $\alpha = 1$, and where we have scaled the matrix $M^T M$ to have a maximum eigenvalue of $100000$. This results in an objective function that has a condition number of $L/m = 100000$. The matching problem is solved over the Birkhoff polytope with $n = 400$ where we have imposed additional linear constraints. We sample $80$ integers $i$ from $1\leq i \leq n$ without replacement, and we set $x_i = 0$ for the first $40$ integers (to represent that certain matchings are not possible), and $x_i \leq 0.5$ for the remaining $40$ integers to represent a maximum fractional matching. As in the LaCG algorithm \cite{diakonikolas2019lacg}, our approach is compatible with the lazification technique for AFW \citep{pok17lazy}. In this last example we couple our algorithm with the AFW (Lazy) algorithm, resulting in the PF-LaCG (Lazy) algorithm. Moreover, we also benchmark against the AFW (Lazy) algorithm for reference.

\section{Discussion}\label{sec:Conclusion}

We have introduced a 
novel projection-free PF-LaCG 
algorithm for minimizing smooth and strongly convex functions over polytopes. This algorithm is parameter-free and locally accelerated.  

In particular, we have shown that after a finite burn-in phase (independent of the target error $\epsilon$) PF-LaCG achieves a near-optimal accelerated 
convergence rate without knowledge of any of the problem parameters (such as the smoothness or strong convexity of the function). As mentioned in the introduction, global acceleration is generally not possible for CG-type methods. We have also demonstrated the improved locally accelerated convergence rate of PF-LaCG using numerical experiments. 
Some interesting questions for future research remain. For example, it is an interesting and practically relevant question whether local acceleration is possible for CG methods that are not active set-based, such as, e.g., DICG~\citep{garber2016linear}, which would possibly lead to even faster algorithms.

\section*{Acknowledgments}
\label{sec:acknowledgments}

This research was partially funded by NSF grant CCF-2007757, by the Office of the Vice Chancellor for Research and Graduate Education at the University of Wisconsin–Madison with funding from the Wisconsin Alumni Research Foundation, and by the Deutsche Forschungsgemeinschaft (DFG) through the DFG Cluster of Excellence MATH+ and the Research Campus Modal funded by the German Federal Ministry of Education and Research (fund numbers 05M14ZAM, 05M20ZBM).


\bibliographystyle{icml2021}

\bibliography{references}

\begin{thebibliography}{44}
\providecommand{\natexlab}[1]{#1}
\providecommand{\url}[1]{\texttt{#1}}
\expandafter\ifx\csname urlstyle\endcsname\relax
  \providecommand{\doi}[1]{doi: #1}\else
  \providecommand{\doi}{doi: \begingroup \urlstyle{rm}\Url}\fi

\bibitem[Bashiri \& Zhang(2017)Bashiri and Zhang]{bashiri2017decomposition}
Bashiri, M.~A. and Zhang, X.
\newblock Decomposition-invariant conditional gradient for general polytopes
  with line search.
\newblock In \emph{NIPS}, pp.\  2690--2700, 2017.

\bibitem[Beck(2017)]{beck2017optimization}
Beck, A.
\newblock \emph{First-order methods in optimization}.
\newblock MOS-SIAM, 2017.

\bibitem[Beck \& Shtern(2017)Beck and Shtern]{beck2017linearly}
Beck, A. and Shtern, S.
\newblock Linearly convergent away-step conditional gradient for non-strongly
  convex functions.
\newblock \emph{Mathematical Programming}, 164\penalty0 (1-2):\penalty0 1--27,
  2017.

\bibitem[Braun et~al.(2017)Braun, Pokutta, and Zink]{pok17lazy}
Braun, G., Pokutta, S., and Zink, D.
\newblock Lazifying conditional gradient algorithms.
\newblock In \emph{Proc.~ICML'2017}, 2017.

\bibitem[Braun et~al.(2019)Braun, Pokutta, Tu, and Wright]{pok18bcg}
Braun, G., Pokutta, S., Tu, D., and Wright, S.
\newblock Blended conditional gradients: {T}he unconditioning of conditional
  gradients.
\newblock In \emph{Proc.~ICML'19}, 2019.

\bibitem[Carderera \& Pokutta(2020)Carderera and Pokutta]{carderera2020second}
Carderera, A. and Pokutta, S.
\newblock Second-order conditional gradient sliding.
\newblock \emph{arXiv preprint arXiv:2002.08907}, 2020.

\bibitem[Carderera et~al.(2021)Carderera, Pokutta, Sch{\"u}tte, and
  Weiser]{carderera2021cindy}
Carderera, A., Pokutta, S., Sch{\"u}tte, C., and Weiser, M.
\newblock {CINDy}: {C}onditional gradient-based identification of non-linear
  dynamics--noise-robust recovery.
\newblock \emph{arXiv preprint arXiv:2101.02630}, 2021.

\bibitem[Chen et~al.(1998)Chen, Donoho, and Saunders]{chen1998atomic}
Chen, S.~S., Donoho, D.~L., and Saunders, M.~A.
\newblock Atomic decomposition by basis pursuit.
\newblock \emph{SIAM Journal on Scientific Computing}, 20\penalty0
  (1):\penalty0 33--61, 1998.

\bibitem[Cohen et~al.(2018)Cohen, Diakonikolas, and
  Orecchia]{cohen2018acceleration}
Cohen, M.~B., Diakonikolas, J., and Orecchia, L.
\newblock On acceleration with noise-corrupted gradients.
\newblock In \emph{Proc.~ICML'18}, 2018.

\bibitem[Combettes \& Pokutta(2021)Combettes and
  Pokutta]{combettes2021complexity}
Combettes, C.~W. and Pokutta, S.
\newblock Complexity of linear minimization and projection on some sets.
\newblock \emph{arXiv preprint arXiv:2101.10040}, 2021.

\bibitem[Combettes et~al.(2020)Combettes, Spiegel, and
  Pokutta]{combettes20adasfw}
Combettes, C.~W., Spiegel, C., and Pokutta, S.
\newblock Projection-free adaptive gradients for large-scale optimization.
\newblock \emph{arXiv preprint arXiv:2009.14114}, 2020.

\bibitem[Condat(2016)]{condat2016fast}
Condat, L.
\newblock Fast projection onto the simplex and the $\ell_1$ ball.
\newblock \emph{Mathematical Programming}, 158\penalty0 (1):\penalty0 575--585,
  2016.

\bibitem[Diakonikolas \& Orecchia(2019)Diakonikolas and
  Orecchia]{diakonikolas2019approximate}
Diakonikolas, J. and Orecchia, L.
\newblock The approximate duality gap technique: {A} unified theory of
  first-order methods.
\newblock \emph{SIAM Journal on Optimization}, 29\penalty0 (1):\penalty0
  660--689, 2019.

\bibitem[Diakonikolas et~al.(2020)Diakonikolas, Carderera, and
  Pokutta]{diakonikolas2019lacg}
Diakonikolas, J., Carderera, A., and Pokutta, S.
\newblock Locally accelerated conditional gradients.
\newblock In \emph{Proc.~AISTATS'20}, 2020.

\bibitem[Duchi et~al.(2008)Duchi, Shalev-Shwartz, Singer, and
  Chandra]{duchi2008efficient}
Duchi, J., Shalev-Shwartz, S., Singer, Y., and Chandra, T.
\newblock Efficient projections onto the $\ell_1$-ball for learning in high
  dimensions.
\newblock In \emph{Proc.~NIPS'08}, 2008.

\bibitem[Dvurechensky et~al.(2020)Dvurechensky, Ostroukhov, Safin, Shtern, and
  Staudigl]{dvurechensky2020self}
Dvurechensky, P., Ostroukhov, P., Safin, K., Shtern, S., and Staudigl, M.
\newblock Self-concordant analysis of {Frank-Wolfe} algorithms.
\newblock In \emph{Proc.~ICML'20}, 2020.

\bibitem[Frank \& Wolfe(1956)Frank and Wolfe]{frank1956algorithm}
Frank, M. and Wolfe, P.
\newblock An algorithm for quadratic programming.
\newblock \emph{Naval research logistics quarterly}, 3\penalty0 (1-2):\penalty0
  95--110, 1956.

\bibitem[Garber(2016)]{garber2016faster}
Garber, D.
\newblock Faster projection-free convex optimization over the spectrahedron.
\newblock In \emph{Proc.~NIPS'16}, 2016.

\bibitem[Garber(2020)]{garber2020revisiting}
Garber, D.
\newblock Revisiting {Frank-Wolfe} for polytopes: {S}trict complementarity and
  sparsity.
\newblock In \emph{Proc. NeurIPS'20}, 2020.

\bibitem[Garber \& Meshi(2016)Garber and Meshi]{garber2016linear}
Garber, D. and Meshi, O.
\newblock Linear-memory and decomposition-invariant linearly convergent
  conditional gradient algorithm for structured polytopes.
\newblock In \emph{Proc.~NIPS'16}, 2016.

\bibitem[Gu{\'e}lat \& Marcotte(1986)Gu{\'e}lat and Marcotte]{guelat1986some}
Gu{\'e}lat, J. and Marcotte, P.
\newblock Some comments on {W}olfe's `away step'.
\newblock \emph{Mathematical Programming}, 35\penalty0 (1):\penalty0 110--119,
  1986.

\bibitem[Gutman \& Pe\~na(2018)Gutman and Pe\~na]{gutman2018condition}
Gutman, D.~H. and Pe\~na, J.~F.
\newblock The condition of a function relative to a polytope.
\newblock \emph{arXiv preprint arXiv:1802.00271}, 2018.

\bibitem[Hazan \& Luo(2016)Hazan and Luo]{hazan2016variance}
Hazan, E. and Luo, H.
\newblock Variance-reduced and projection-free stochastic optimization.
\newblock In \emph{Proc.~ICML'16}, 2016.

\bibitem[Held et~al.(1974)Held, Wolfe, and Crowder]{held1974validation}
Held, M., Wolfe, P., and Crowder, H.~P.
\newblock Validation of subgradient optimization.
\newblock \emph{Mathematical programming}, 6\penalty0 (1):\penalty0 62--88,
  1974.

\bibitem[Ito \& Fukuda(2019)Ito and Fukuda]{ito2019gradientmapping}
Ito, M. and Fukuda, M.
\newblock Nearly optimal first-order methods for convex optimization under
  gradient norm measure: {A}n adaptive regularization approach.
\newblock \emph{arXiv preprint arXiv:1912.12004}, 2019.

\bibitem[Jaggi(2013)]{jaggi2013revisiting}
Jaggi, M.
\newblock Revisiting {Frank-Wolfe}: {P}rojection-free sparse convex
  optimization.
\newblock In \emph{Proc.~ICML'13}, 2013.

\bibitem[Kerdreux et~al.(2019)Kerdreux, d'Aspremont, and
  Pokutta]{kerdreux2019restartfw}
Kerdreux, T., d'Aspremont, A., and Pokutta, S.
\newblock Restarting {F}rank-{W}olfe: {F}aster rates under {H}{\"o}lderian
  error bounds.
\newblock In \emph{Proc.~AISTATS'19}, 2019.

\bibitem[Kerdreux et~al.(2021)Kerdreux, d'Aspremont, and
  Pokutta]{UniformConvexFW_2020}
Kerdreux, T., d'Aspremont, A., and Pokutta, S.
\newblock Projection-free optimization on uniformly convex sets.
\newblock In \emph{Proc.~AISTATS'21}, 2021.

\bibitem[Lacoste-Julien \& Jaggi(2015)Lacoste-Julien and
  Jaggi]{lacoste2015global}
Lacoste-Julien, S. and Jaggi, M.
\newblock On the global linear convergence of {Frank-Wolfe} optimization
  variants.
\newblock In \emph{Proc.~NIPS'15}, 2015.

\bibitem[Lam et~al.(2015)Lam, Pitrou, and Seibert]{lam2015numba}
Lam, S.~K., Pitrou, A., and Seibert, S.
\newblock Numba: A {LLVM}-based {P}ython {JIT} compiler.
\newblock In \emph{Proceedings of the Second Workshop on the LLVM Compiler
  Infrastructure in HPC}, pp.\  1--6, 2015.

\bibitem[Lan(2013)]{lan2013complexity}
Lan, G.
\newblock The complexity of large-scale convex programming under a linear
  optimization oracle.
\newblock \emph{arXiv preprint arXiv:1309.5550}, 2013.

\bibitem[Lei et~al.(2019)Lei, Zhuo, Caramanis, Dhillon, and
  Dimakis]{lei2019primal}
Lei, Q., Zhuo, J., Caramanis, C., Dhillon, I.~S., and Dimakis, A.~G.
\newblock Primal-dual block generalized {F}rank-{W}olfe.
\newblock \emph{Proc.~NeurIPS'19}, 2019.

\bibitem[Levitin \& Polyak(1966)Levitin and Polyak]{levitin1966constrained}
Levitin, E.~S. and Polyak, B.~T.
\newblock Constrained minimization methods.
\newblock \emph{USSR Computational mathematics and mathematical physics},
  6\penalty0 (5):\penalty0 1--50, 1966.

\bibitem[N{\'e}giar et~al.(2020)N{\'e}giar, Dresdner, Tsai, El~Ghaoui,
  Locatello, Freund, and Pedregosa]{negiar2020stochastic}
N{\'e}giar, G., Dresdner, G., Tsai, A., El~Ghaoui, L., Locatello, F., Freund,
  R., and Pedregosa, F.
\newblock Stochastic {F}rank-{W}olfe for constrained finite-sum minimization.
\newblock In \emph{Proc.~ICML'20}, 2020.

\bibitem[Nesterov(2012)]{nesterov2012make}
Nesterov, Y.
\newblock How to make the gradients small.
\newblock \emph{Optima. Mathematical Optimization Society Newsletter},
  \penalty0 (88):\penalty0 10--11, 2012.

\bibitem[Nesterov(2013)]{nesterov2013methods}
Nesterov, Y.
\newblock Gradient methods for minimizing composite functions.
\newblock \emph{Mathematical Programming}, \penalty0 (140(1)):\penalty0
  125--161, 2013.

\bibitem[Nesterov(2018)]{nesterov2018introductory}
Nesterov, Y.
\newblock \emph{Lectures on Convex Optimization}.
\newblock Springer, 2018.

\bibitem[Pe\~na \& Rodr\'{i}guez(2019)Pe\~na and
  Rodr\'{i}guez]{pena2019polytope}
Pe\~na, J. and Rodr\'{i}guez, D.
\newblock Polytope conditioning and linear convergence of the {F}rank--{W}olfe
  algorithm.
\newblock \emph{Mathematics of Operations Research}, 44\penalty0 (1):\penalty0
  1--18, 2019.

\bibitem[Pedregosa et~al.(2020)Pedregosa, Negiar, Askari, and
  Jaggi]{pedregosa2018step}
Pedregosa, F., Negiar, G., Askari, A., and Jaggi, M.
\newblock Linearly convergent {F}rank--{W}olfe with backtracking line-search.
\newblock In \emph{Proc.~AISTATS'20}, 2020.

\bibitem[Roulet \& d'Aspremont(2020)Roulet and
  d'Aspremont]{roulet2020sharpness}
Roulet, V. and d'Aspremont, A.
\newblock Sharpness, restart, and acceleration.
\newblock \emph{SIAM Journal on Optimization}, 30\penalty0 (1):\penalty0
  262--289, 2020.

\bibitem[Tibshirani(1996)]{tibshirani1996regression}
Tibshirani, R.
\newblock Regression shrinkage and selection via the lasso.
\newblock \emph{Journal of the Royal Statistical Society: Series B
  (Methodological)}, 58\penalty0 (1):\penalty0 267--288, 1996.

\bibitem[Tsiligkaridis \& Roberts(2020)Tsiligkaridis and
  Roberts]{tsiligkaridis2020frank}
Tsiligkaridis, T. and Roberts, J.
\newblock On {Frank-Wolfe} optimization for adversarial robustness and
  interpretability.
\newblock \emph{arXiv preprint arXiv:2012.12368}, 2020.

\bibitem[Zhang et~al.(2020)Zhang, Shen, Mokhtari, Hassani, and
  Karbasi]{zhang2020one}
Zhang, M., Shen, Z., Mokhtari, A., Hassani, H., and Karbasi, A.
\newblock One sample stochastic {F}rank-{W}olfe.
\newblock In \emph{Proc.~AISTATS'20}, 2020.

\bibitem[Zhou et~al.(2018)Zhou, Gupta, and Udell]{zhou2018limited}
Zhou, S., Gupta, S., and Udell, M.
\newblock Limited memory {K}elley's method converges for composite convex and
  submodular objectives.
\newblock \emph{arXiv preprint arXiv:1807.07531}, 2018.

\end{thebibliography}

\newpage
\onecolumn
\appendix

{\centering{\LARGE\bfseries Parameter-free Locally Accelerated Conditional Gradients}

  \vspace{1em}
  \centering{{\LARGE\bfseries Appendix}}

}

\vspace{2em}

\paragraph{Outline.} The appendix of the paper is organized as follows:
\begin{itemize}[leftmargin=*]
  \item Section~\ref{appx:FAFW} presents the proofs related to the AFW algorithm used in the main body of the paper.
  \item Section~\ref{appx:section:acceralated-analysis} presents the proofs related to the accelerated algorithm with inexact gradient mappings that is used in this work, as well as come useful properties of the gradient mapping.
   \item Section~\ref{appx:section:PFLaCG} shows how we can couple the AFW algorithm and the accelerated algorithm presented in Section~\ref{appx:section:acceralated-analysis} to achieve a parameter-free accelerated CG variant, dubbed \emph{Parameter-free Locally Accelerated Conditional Gradients} (PF-LaCG).
  \item Section~\ref{appx:section:comp-results} presents the full details of the computational experiments performed in the paper.
\end{itemize}

\section{Away-Step Frank-Wolfe (AFW)}\label{appx:FAFW}

We use a modified version of the \emph{Away-Step Frank-Wolfe} (AFW) algorithm \cite{guelat1986some, lacoste2015global} that is run until the Frank-Wolfe gap is halved, in a similar manner as in \cite{kerdreux2019restartfw}. The only difference between the AFW algorithm presented in Algorithm~\ref{alg:fafw} and the one from \citep{kerdreux2019restartfw} is that 
we substitute the condition that chooses between Frank-Wolfe steps (Line~\ref{op:fw-step}) and away-steps (Line~\ref{op:away-step}) for the condition used in the classical Away-Step Frank Wolfe (AFW) algorithm, shown in Line~\ref{alg:FAFW:step_type_selection}. In the following, we will say that a step is a \emph{full-progress step} if it is either a Frank-Wolfe Step (Line~\ref{op:fw-step}) or an away-step (Line~\ref{op:away-step}) that is not a drop step, i.e., when $\lambda_k < \alpha_k^{\vs_k} / (1 - \alpha_k^{\vs_k})$, where $\alpha_k^{\vs_k}$ is the barycentric coordinate of $\vs_k$ with respect to the current active set. Here we summarize the results that are utilized in our analysis. 

The following definition is introduced for completeness, to state the results from~\cite{kerdreux2019restartfw}.
\begin{definition}[Away curvature]
  \label{defn:away-curvature}
  Given a problem~\eqref{eq:problem}, the away curvature $C_f^A$ is defined as
  \begin{equation*}
    C_f^A := \sup_{\substack{\vx, \vs, \mathbf{v} \in \cx \\ \rho \in [0, 1] \\ \vy = \vx + \rho (\vs - \mathbf{v})}} \frac{2}{\rho^2} (f (\vy) - f (\vx) - \rho \inner{\nabla f (\vx)}{\vs - \mathbf{v}}).
  \end{equation*}
  As $f$ is $L$-smooth, we have $C_f^A \le L D^2$, where $D = \max_{\vx, \vy \in \cx}\|\vy - \vx\|$ denotes the diameter of $\cx$.  
\end{definition}

\begin{algorithm}[htb!]
  \caption{Away-Step Frank-Wolfe Algorithm: $\mathrm{AFW} (\vx_0, \cs_0, \epsilon^w)$}
  \begin{algorithmic}[1]
    \State $k := 0$
    \While {$w (\vx_k, \cs_k) > w (\vx_0, \cs_0) / 2$} \label{alg:exit_criterion}
      \State $\mathbf{v}_k := \argmin_{\vu \in \cx} \inner{\nabla f (\vx_k)}{\vu}$ and $\vd_k^{\mathrm{FW}} := \mathbf{v}_k - \vx_k$ \label{alg:FAFW:FWvertex}
      \State $\vs_k := \argmin_{\vu \in \cs_k} \inner{- \nabla f (\vx_k)}{\vu}$ with $\cs_k$ current active set and $\vd_k^{\mathrm{Away}} := \vx_k - \vs_k$ \label{alg:FAFW:awayvertex}
      \If {$- \inner{\nabla f (\vx_k)}{\vd_k^{\mathrm{FW}}} \geq -\inner{\nabla f (\vx_k)}{\vd_k^{\mathrm{Away}}}$} \label{alg:FAFW:step_type_selection}
        \State $\vd_k := \vd_k^{\mathrm{FW}}$ with $\lambda_{\mathrm{max}} := 1$  \label{op:fw-step}
      \Else
        \State $\vd_k := \vd_k^{\mathrm{Away}}$ with $\lambda_{\mathrm{max}} := \frac{\alpha_k^{\vs_k}}{1 - \alpha_k^{\vs_k}}$  \label{op:away-step}
      \EndIf
      \State $\vx_{k + 1} := \vx_k + \lambda_k \vd_k$ with $\lambda_k \in [0, \lambda_{\mathrm{max}}]$ via line-search \label{alg:FAFW:stepsize}
      \State Update active set $\cs_{k + 1}$ and coefficients $\{ \alpha_{k + 1}^{\mathbf{v}} \}_{\mathbf{v} \in \cs_{k + 1}}$
      \State $k := k + 1$
    \EndWhile
    \State \Return ($\vx_k, \cs_k, w (\vx_k, \cs_k)$) where $\vx_k \in \cx$ and $w (\vx_k, \cs_k) \le \epsilon^w$
  \end{algorithmic} \label{alg:fafw}
\end{algorithm}

\begin{fact} [{\citep[Lemma 3.6]{kerdreux2019restartfw}}]
  \label{lemma:optgap-le-strwolfegap}
  When $f$ is $m$-strongly convex over $\cx$ and $\cx$ is a polytope satisfying $\delta$-scaling inequality with $\delta > 0$, then we have for all $\vx \in \cx$:
  $$
    f(\vx) - \min_{\vy \in \cx}f(\vy) \le \frac{2 w(\vx)^2}{m \delta^2}. 
  $$
\end{fact}

\begin{proposition} [Iterations needed to halve $w (\vx, \cs)$ {\citep[Proposition 4.1]{kerdreux2019restartfw}}]
  \label{prop:fafw-halve-rate}
  Let $f$ be an $L$-smooth, $m$-strongly convex function with away curvature $C_f^A$ and let $\cx$ be a polytope satisfying the $\delta$-scaling inequality with $\delta > 0$. Assume that $\vx_0 \in \cx$ is such that $w (\vx_0) / 4 \le C_f^A$. Then AFW (Algorithm~\ref{alg:fafw}) outputs an iterate $\vx_K \in \cx$ such that
  $$
    w (\vx_K, \cs_K) \le w (\vx_0, \cs_0) / 2
  $$
  after at most 
  $$
    K \le \abs{\cs_0} - \abs{\cs_K} + \frac{128 C_f^A}{m \delta^2}
  $$
  iterations, where $\cs_0$ and $\cs_K$ are the initial active set and the active set at iteration $K$ respectively.
  \begin{proof}
  Note that by Fact~\ref{lemma:optgap-le-strwolfegap} we have that the $2$-strong Wolfe primal bound holds. We only explain the part of the proof that differs from the proof in \cite{kerdreux2019restartfw}, due to the modified condition in Line~\ref{alg:FAFW:step_type_selection}. This condition only affects how we bound $w (\vx_0, \cs_0)$, depending on whether we take a Frank-Wolfe step or an away-step, as follows:
    \begin{align*}
    w (\vx_0, \cs_0) / 2 < w (\vx_k, \cs_k) = - \inner{\nabla f (\vx_k)}{\vd_k^{\mathrm{FW}}} -\inner{\nabla f (\vx_k)}{\vd_k^{\mathrm{Away}}} \leq -2 \inner{\nabla f (\vx_k)}{\vd_k},
    \end{align*}
    where the first inequality simply states that we have not yet halved  $w (\vx, \cs)$ (see Line~\ref{alg:exit_criterion}), and the second inequality stems from the condition with which we choose the steps in Line~\ref{alg:FAFW:step_type_selection}. This allows us to claim that in case a Frank-Wolfe step is chosen we have that $ w (\vx_0, \cs_0) / 4 <  \inner{\nabla f (\vx_k)}{\vd_k}$. Using this fact, we can proceed with the same arguments as in the proof of~\citet[Proposition 4.1]{kerdreux2019restartfw}, which is omitted.
  \end{proof}
\end{proposition}

Notice that Proposition~\ref{prop:fafw-halve-rate} requires an assumption that $w (\vx_0) / 4 \le C_f^A$. This assumption is satisfied after a small number of iterations, as summarized in the following proposition. 

\begin{proposition} [{\citep[Proposition 4.2]{kerdreux2019restartfw}}]
  \label{prop:fafw-burn-in}
  Assume the AFW algorithm (Algorithm~\ref{alg:fafw}) is run repeatedly until it outputs a point $\vx \in \cx$ such that $w (\vx, \cs) / 4 \le C_f^A$ where $\cs$ is a proper support for $\vx$. This will happen after at most
  $$
    T_0 = \frac{16}{\log 2} \log \frac{w (\vx_0, \cs_0)}{2 C_f^A} + \abs{\cs_0}
  $$
  FOO and LMO calls.
\end{proposition}

Since $w (\vx) \le w (\vx, \cs)$ for all $\vx \in \cx$ and any support $\cs$ for $\vx$, Proposition~\ref{prop:fafw-burn-in} implies that Proposition~\ref{prop:fafw-halve-rate} applies after at most $T_0$ initial iterations. The proof is omitted and can instead be found in~\citet{kerdreux2019restartfw}. 

\subsection{Implications of Strict Complementarity} \label{appx:section:auxiliary-results}

We reproduce here for completeness the relevant results from~\citet[Theorem 1]{garber2020revisiting}, where it is shown that if the iterates are not in $\mathcal{F}\left( \vx^*\right)$ and the primal gap is below a given tolerance, then all subsequent steps will be drop steps that drop vertices in the current active set that are not in $\mathcal{F}\left( \vx^*\right)$.

\begin{theorem}  \label{appx:theorem:drop-step-if-not-in-optimal-face}
  If the strict complementarity assumption is satisfied (Assumption~\ref{assumption:strictComplementarity}) and the primal gap satisfies $ f(\vx_k) -f(\vx^*) < 1/2 \min\left\{  ( \tau /(2D(\sqrt{L/m} + 1)))^2/L, \tau, LD^2 \right\} $ then the following holds for the AFW algorithm (Algorithm~\ref{alg:fafw}):
  \begin{enumerate}
      \item If $\vx_k \notin \mathcal{F}\left(\vx^* \right)$, AFW will perform an away step that drops a vertex $\vs_k \in \vertex \left( \cx \right) \setminus \mathcal{F}\left(\vx^* \right)$.
      \item If $\vx_k \in \mathcal{F}\left(\vx^* \right)$, AFW will either perform a Frank-Wolfe step with a vertex $\mathbf{v}_k \in \vertex \left( \mathcal{F}\left(\vx^* \right) \right)$ or an away-step with a vertex $\vs_k \in \vertex \left( \mathcal{F}\left(\vx^* \right) \right)$. Regardless of which step is chosen, the iterate will satisfy:
      \begin{align*}
          w (\vx_k, \cs_k) \leq \frac{LD\sqrt{2}}{\sqrt{m}}\sqrt{f\left(\vx_k\right) - f\left(\vx^*\right)}.
      \end{align*}
  \end{enumerate}
  \begin{proof}
  We first prove the first claim. Consider two vertices $\mathbf{v} \in \vertex \left( \cx \right) \cap \mathcal{F}\left(\vx^* \right)$ and $\vs \in \vertex \left( \cx \right) \setminus \mathcal{F}\left(\vx^* \right)$, then we have that:
  \begin{align}
      \inner{\vs -\mathbf{v}}{\nabla f(\vx_k)} &= \inner{\vs -\vx^*}{\nabla f(\vx^*)} + \inner{\vx^* -\mathbf{v}}{\nabla f(\vx^*)} + \inner{\vs -\mathbf{v}}{\nabla f(\vx_k) - \nabla f(\vx^*)}\notag \\
       &\geq \tau - \norm{\vs -\mathbf{v}}\norm{\nabla f(\vx_k) - \nabla f(\vx^*)}\notag \\
       &\geq \tau - LD\norm{\vx_k - \vx^*}\notag \\
       &\geq \tau - LD\sqrt{2\left( f(\vx_k) -f(\vx^*) \right)/m}\notag \\
       &\geq \tau /2, \label{appx:eq:FW_vertex_choice}
  \end{align}
  where the first inequality comes from the strict complementarity assumption, as $\mathbf{v} \in \vertex \left( \cx \right) \cap \mathcal{F}\left(\vx^* \right)$ and $\vs \in \vertex \left( \cx \right) \setminus \mathcal{F}\left(\vx^* \right)$, and the Cauchy-Schwarz inequality, the second from $L$-smoothness and the fact that $\norm{\vs -\mathbf{v}} \leq D$, and the third inequality from $m$-strong convexity. Note that the last inequality comes from the fact that we assume that the primal gap satisfies $f(\vx_k) -f(\vx^*) < ( \tau /(2D (\sqrt{L/m} + 1)))^2/(2L) \leq \left( \tau /(2LD)\right)^2m/2$. This allows us to claim that $\inner{\vs}{\nabla f(\vx_k)} > \inner{\mathbf{v}}{\nabla f(\vx_k)}$ for any $\mathbf{v} \in \vertex \left( \cx \right) \cap \mathcal{F}\left(\vx^* \right)$ and $\vs \in \vertex \left( \cx \right) \setminus \mathcal{F}\left(\vx^* \right)$, which means that the Frank-Wolfe vertex in Line~\ref{alg:FAFW:FWvertex} satisfies that $\mathbf{v}_k \in \vertex \left( \cx \right) \cap \mathcal{F}\left(\vx^* \right)$. Alternatively, if we have that $\vx_k \notin \mathcal{F}\left( \vx^*\right)$ then $\cs_k \setminus \mathcal{F}\left(\vx^* \right)$ is nonempty, which also means that the away-vertex chosen in Line~\ref{alg:FAFW:awayvertex} will be such that $\vs_k \in \cs_k \setminus \mathcal{F}\left(\vx^* \right)$. The proof proceeds by showing that the AFW algorithm will chose to perform an away-step in Line~\ref{alg:FAFW:step_type_selection} of Algorithm~\ref{alg:fafw}, as opposed to a Frank-Wolfe step. Let $\vs \in \cs_k \setminus \mathcal{F}\left(\vx^* \right)$, using arguments that are similar to the ones in the previous chain of inequalities, then:
  \begin{align}
      \inner{\vs - \vx_k}{\nabla f(\vx_k)} & = \inner{\vs - \vx^*}{\nabla f(\vx^*)} + \inner{\vs - \vx^*}{\nabla f(\vx_k) - \nabla f(\vx^*)} +\inner{\vx^* - \vx_k}{\nabla f(\vx_k)} \notag \\
    & \geq \tau - \norm{\vs - \vx^*}\norm{\nabla f(\vx_k) - \nabla f(\vx^*)} - \max_{\mathbf{v} \in \cx}\inner{\vx_k - \mathbf{v}}{\nabla f(\vx_k)}\notag \\
    & \geq \tau -LD \norm{\vx_k - \vx^*} - \max_{\mathbf{v} \in \cx}\inner{\vx_k - \mathbf{v}}{\nabla f(\vx_k)} \notag\\
   &\geq \tau - LD\sqrt{2\left( f(\vx_k) -f(\vx^*) \right)/m} - D \sqrt{2L (f(\vx_k) - f(\vx^*))}\notag\\
   &= \tau - D\sqrt{2L\left( f(\vx_k) -f(\vx^*) \right)} \left(\sqrt{L/m} + 1 \right) \notag\\
   &\geq \tau/2 , \label{appx:bound:away_gap}
  \end{align}
  where the first inequality stems from the strict complementarity assumption, which applies to $\vs \in \vertex \left( \cx \right) \setminus \mathcal{F}\left(\vx^* \right)$, the Cauchy-Schwarz inequality, and the fact that $\inner{\vx_k - \vx^*}{\nabla f(\vx_k)} \leq \max_{\mathbf{v} \in \cx}\inner{\vx_k - \mathbf{v}}{\nabla f(\vx_k)}$. The second inequality uses the fact that $\norm{\vs - \vx^*} \leq D$ and $L$-smoothness. Note that from Theorem 2 in \citet{lacoste2015global}, we know that since the function is $L$-smooth then if we have that $f(\vx_k) - f(\vx^*)< LD^2/2$, then we also have that $\max_{\mathbf{v} \in \cx} \inner{\vx_k - \mathbf{v}}{\nabla f(\vx_k)} \leq D \sqrt{2L (f(\vx_k) - f(\vx^*))}$. The third inequality uses the aforementioned bound, along with $m$-strong convexity for the term that contains $\norm{\vx_k - \vx^*}$. The last inequality again uses the assumption that the primal gap satisfies $f(\vx_k) -f(\vx^*) < ( \tau /(D (\sqrt{L/m} + 1)))^2/(8L)$.  Moving on to the bound on the Frank-Wolfe gap, using the bound on Theorem 2 in \citet{lacoste2015global} and our assumption on the primal gap allows us to conclude that:
  \begin{align}
      \max_{\mathbf{v} \in \cx} \inner{\vx_k - \mathbf{v}}{\nabla f(\vx_k)} \leq D \sqrt{2L (f(\vx_k) - f(\vx^*))} \leq \frac{\tau}{2 (\sqrt{L/m} + 1)} \leq \frac{\tau}{2}, \label{appx:bound:FW_gap}
  \end{align}
  where the last inequality simply stems from $L/m \geq 1$. Putting together Eqs.~\eqref{appx:bound:away_gap} and \eqref{appx:bound:FW_gap} we can conclude that for any $\vs \in \cs_k \setminus \mathcal{F}\left(\vx^* \right)$ we have that $\inner{\vs - \vx_k}{\nabla f(\vx_k)} \geq \max_{\mathbf{v} \in \cx} \inner{\vx_k - \mathbf{v}}{\nabla f(\vx_k)}$. Consequently, we have that $\max_{\vs \in \cs_k \setminus \mathcal{F}\left(\vx^* \right)}\inner{\vs - \vx_k}{\nabla f(\vx_k)} \geq \max_{\mathbf{v} \in \cx} \inner{\vx_k - \mathbf{v}}{\nabla f(\vx_k)}$. This means that the AFW algorithm will perform an away step, moving away from a vertex $\vs_k \in \cs_k \setminus \mathcal{F}\left(\vx^* \right)$. The last step of the proof proceeds to show that the step size chosen in the away-step corresponds to the largest possible step size $\lambda_{\mathrm{max}}$. To do this, assume for contradiction that this is not the case, and assume that $\lambda_k = \argmin_{\lambda \in [0, \lambda_{\mathrm{max}}]}f(\vx + \lambda (\vx_k - \vs_k)) <\lambda_{\mathrm{max}}$, by the optimality of the line search we must have that $\inner{\vs_k - \vx_k}{\nabla f(\vx_{k+1})} = 0$. On the other hand, we can write:
  \begin{align}
      \inner{\vs_k - \vx_k}{\nabla f(\vx_{k+1})} & = \inner{\vs_k - \vx^*}{\nabla f(\vx^*)} +\inner{\vx^* - \vx_k}{\nabla f(\vx^*)}  + \inner{\vs_k - \vx_k}{\nabla f(\vx_{k+1})- \nabla f(\vx^*)}, \notag \\
      & \geq \tau -\left( f(\vx_k) - f(\vx^*)\right) + \inner{\vs_k - \vx_k}{\nabla f(\vx_{k+1})- \nabla f(\vx^*)} \label{appx:initial_equation}\\
      & > \tau/2 - \norm{\vs_k - \vx_k}\norm{\nabla f(\vx_{k+1})- \nabla f(\vx^*)} \label{appx:eq:bound-decomposition}\\
      & \geq \tau/2 - LD\norm{ \vx_{k+1}-  \vx^*}\notag \\
      & \geq \tau/2 - LD\sqrt{ 2(f(\vx_{k}) - f(\vx^*))/m}\notag \\
     & > 0,\notag
  \end{align}
  where the inequality in Eq.~\eqref{appx:initial_equation} follows from the strict complementarity assumption, which means that $ \inner{\vs_k - \vx^*}{\nabla f(\vx^*)} \geq \tau$, and from convexity, which implies that $\inner{\vx^* - \vx_k}{\nabla f(\vx^*)} \geq -\left( f(\vx_k) - f(\vx^*)\right)$. The following inequality, shown in Eq.~\eqref{appx:eq:bound-decomposition}, follows from the Cauchy-Schwarz inequality and from the bound on the primal gap $f(\vx_k) - f(\vx^*) < \tau/2$. The following inequalities follow from the application of the $L$-smoothness and $m$-strong convexity of the objective function and the bound on the primal gap $f(\vx_k) -f(\vx^*) < ( \tau /(2D (\sqrt{L/m} + 1)))^2/(2L) \leq \left( \tau /(2LD)\right)^2m/2$. This proves the desired contradiction, and so we must have that $\lambda_k = \lambda_{\mathrm{max}}$. This means that at iteration $k$ we have performed an away step that has dropped a vertex $\vs_k \in \cs_k \setminus \mathcal{F}\left( \vx^*\right)$. This proves the first claim. 
  
  The first part of the second claim follows by noting that if $\vx_k \in \mathcal{F}(\vx^*)$, then the away-vertex chosen in Line~\ref{alg:FAFW:awayvertex} will satisfy $\vs_k \in \mathcal{F}(\vx^*)$, as $\cs_k \subseteq \mathcal{F}(\vx^*)$. Moreover by Eq.~\eqref{appx:eq:FW_vertex_choice} we know that the Frank-Wolfe vertex chosen in Line~\ref{alg:FAFW:FWvertex} will satisfy $\mathbf{v}_k \in \mathcal{F}(\vx^*)$. This proves the first part of the second claim. Moving on to the second part of the second claim. If we denote the active set of $\vx_k$ at iteration $k$ by $\cs_k$ it follows that $\cs_k \subseteq \mathcal{F}\left( \vx^* \right)$ from the fact that $\vx_k \in \mathcal{F}\left( \vx^* \right)$ (as otherwise we would have $\vx_k \notin \mathcal{F}\left( \vx^* \right)$). Moreover from the first part of the second claim we have that $\mathbf{v}_k = \argmin_{\vu\in \cx} \inner{\nabla f (\vx_k)}{ \vu} \in \mathcal{F}\left( \vx^*\right)$ and $\vs_k = \argmin_{\vu\in \mathcal{S}_k} \inner{-\nabla f (\vx_k)}{ \vu} \in \mathcal{F}\left( \vx^*\right)$. From the definition of $w(\vx_k, \cs_k)$ we have that:
\begin{align*}
    w (\vx_k, \cs_k) & =  \inner{\vs_k -\mathbf{v}_k}{\nabla f(\vx_k)} \\
    & =  \inner{\vs_k -\vx^*}{\nabla f(\vx^*)} + \inner{\vx^* -\mathbf{v}_k}{\nabla f(\vx^*)} + \inner{\vs_k -\mathbf{v}_k}{\nabla f(\vx_k) - \nabla f(\vx^*)} \\
    & = \inner{\vs_k -\mathbf{v}_k}{\nabla f(\vx_k) - \nabla f(\vx^*)} \\
    & \leq \norm{\vs_k -\mathbf{v}_k}\norm{\nabla f(\vx_k) - \nabla f(\vx^*)} \\
    & \leq LD\norm{\vx_k - \vx^*} \\
    & \leq \frac{LD\sqrt{2}}{\sqrt{m}}\sqrt{f\left(\vx_k\right) - f\left(\vx^*\right)}.
\end{align*}
Where the third equality follows from the strict complementarity assumption (Assumption~\ref{assumption:strictComplementarity}), and the fact that as $\vs_k, \mathbf{v}_k \in \mathcal{F}\left( \vx^*\right)$ we have that $\inner{\vx^* -\mathbf{v}_k}{\nabla f(\vx^*)} = 0$ and $\inner{\vs_k -\vx^*}{\nabla f(\vx^*)} = 0$. The first inequality follows from the Cauchy-Schwartz inequality, the second from the $L$-smoothness of $f(x)$ and $\norm{\vs_k -\mathbf{v}_k} \leq D$, and the last one from the $m$-strong convexity of the $f(x)$. This completes the proof.
\end{proof}
\end{theorem}

\begin{corollary}  \label{appx:theorem:drop-step-if-not-in-optimal-face-PFW}
  If the strict complementarity assumption is satisfied (Assumption~\ref{assumption:strictComplementarity}) and the primal gap satisfies $ f(\vx_k) -f(\vx^*) < 1/2 \min\left\{ ( \tau /(2D(\sqrt{L/m} + 1)))^2/L, \tau, LD^2 \right\} $ then the claims in Theorem~\ref{appx:theorem:drop-step-if-not-in-optimal-face} hold for the \emph{Pairwise-Step Frank-Wolfe} (PFW) algorithm, the natural extension of the AFW algorithm to the use of \emph{pairswise-steps}.
  \begin{proof}
  The proof uses many of the same techniques and concepts as Theorem~\ref{appx:theorem:drop-step-if-not-in-optimal-face}, so we only give a brief proof-sketch here. Note that chain of inequalities in Eq.~\eqref{appx:eq:FW_vertex_choice} is independent of the algorithm being used. Therefore if the primal gap satisfies $f(\vx_k) -f(\vx^*) < ( \tau /(2D (\sqrt{L/m} + 1)))^2/(2L) \leq \left( \tau /(2LD)\right)^2m/2$, then we can claim that $\inner{\vs}{\nabla f(\vx_k)} > \inner{\mathbf{v}}{\nabla f(\vx_k)}$ for any $\mathbf{v} \in \vertex \left( \cx \right) \cap \mathcal{F}\left(\vx^* \right)$ and $\vs \in \cs_k \setminus \mathcal{F}\left(\vx^* \right)$, which means that the Frank-Wolfe vertex in the PFW algorithm satisfies that $\mathbf{v}_k \in \vertex \left( \cx \right) \cap \mathcal{F}\left(\vx^* \right)$. Alternatively, this also means that the away-vertex chosen in the PFW algorithm will be such that $\vs_k \in \cs_k \setminus \mathcal{F}\left(\vx^* \right)$. Note that the directions towards which the iterates move in the PFW algorithm are given by $\vs_k - \vx_k$, with a maximum step size of $\lambda_{\max} = \alpha_k^{\vs_k}$. If  the maximum step size is chosen this means that the vertex $\vs_k \in \cs_k \setminus \mathcal{F}\left(\vx^* \right)$ has been dropped, and so to proceed we will show that this maximum step size is chosen. To do this, assume for contradiction that this is not the case, and assume that $\lambda_k = \argmin_{\lambda \in [0, \lambda_{\mathrm{max}}]}f(\vx + \lambda (\vx_k - \vs_k)) <\lambda_{\mathrm{max}}$, by the optimality of the line search we must have that $\inner{\vs_k - \vx_k}{\nabla f(\vx_{k+1})} = 0$. On the other hand, we can write:
  \begin{align}
      \inner{\vs_k - \mathbf{v}_k}{\nabla f(\vx_{k+1})} & = \inner{\vs_k - \vx^*}{\nabla f(\vx^*)} +\inner{\vx^* - \mathbf{v}_k}{\nabla f(\vx^*)}  + \inner{\vs_k - \mathbf{v}_k}{\nabla f(\vx_{k+1})- \nabla f(\vx^*)}, \notag \\
      & > \tau - \norm{\vs_k - \mathbf{v}_k}\norm{\nabla f(\vx_{k+1})- \nabla f(\vx^*)} \label{appx:eq:bound-decomposition-pfw}\\
      & \geq \tau - LD\norm{ \vx_{k+1}-  \vx^*}\notag \\
      & \geq \tau - LD\sqrt{ 2(f(\vx_{k}) - f(\vx^*))/m}\notag \\
     & > 0,\notag
  \end{align}
  where the inequality in Eq.~\eqref{appx:eq:bound-decomposition-pfw} follows from the Cauchy-Schwarz inequality and the strict complementarity assumption for $\vs_k$ and $\mathbf{v}_k$, which means that $ \inner{\vs_k - \vx^*}{\nabla f(\vx^*)} \geq \tau$ and $ \inner{\mathbf{v}_k - \vx^*}{\nabla f(\vx^*)} =0 $. The following inequalities follow from the application of the $L$-smoothness and $m$-strong convexity of the objective function and the bound on the primal gap $f(\vx_k) -f(\vx^*) < ( \tau /(2D (\sqrt{L/m} + 1)))^2/(2L) \leq \left( \tau /(2LD)\right)^2m/2$. This proves the desired contradiction, and so we must have that $\lambda_k = \lambda_{\mathrm{max}}$. This means that at iteration $k$ we have dropped a vertex $\vs_k \in \cs_k \setminus \mathcal{F}\left( \vx^*\right)$ from the active set $\cs_k$. This proves the first claim. The proof of the second claim and the bound on $w(\vx_k , \cs_k)$ can be repeated word for word from Theorem~\ref{appx:theorem:drop-step-if-not-in-optimal-face}.
\end{proof}
\end{corollary}

Using Theorem~\ref{appx:theorem:drop-step-if-not-in-optimal-face} (or Corollary~\ref{appx:theorem:drop-step-if-not-in-optimal-face-PFW} for the Pairwise-Step Frank-Wolfe (PFW)), and assuming that $\vx_0 \in \vertex \left( \cx\right)$, and using the primal gap convergence gap guarantee in \citet[Theorem 1]{lacoste2015global}, we can bound the number of iterations until $ f(\vx_k) -f(\vx^*)$ satisfies the requirement in Theorem~\ref{appx:theorem:drop-step-if-not-in-optimal-face}. Using this bound, and the fact that the AFW algorithm can pick up at most one vertex per iteration, we can bound the number of iterations until $\vx_k \in \mathcal{F}(\vx^*)$ (there exist other bounds for the PFW algorithm). Note that by the second claim in Theorem~\ref{theorem:drop-step-if-not-in-optimal-face}, this means that when $\vx_k \in \mathcal{F}(\vx^*)$, then the iterates will not leave $\mathcal{F}(\vx^*)$. Furthermore, once the iterates are inside the optimal face, there are two options: if $\vx^* = \mathcal{F}(\vx^*)$, then the AFW algorithm will exit once $\vx_k \in \mathcal{F}(\vx^*)$, as $w(\vx_k, \cs_k) = 0$, otherwise if $\vx^* \notin \vertex \left( \cx\right)$ (the case of interest in our setting, by Assumption~\ref{assmpt:x^*-sufficiently-deep-in-X}), then we need to prove that after a given number of iterations the active set will satisfy $\vx^* \in \co \left( \cs_k\right)$. We prove the former using Fact~\ref{fact:appx:asConvergence} (a variation of \citet[Fact B.3]{diakonikolas2019lacg}).

\begin{fact}[Critical strong Wolfe gap]
  \label{fact:appx:asConvergence}
  There exists an $w_c > 0$ such that for any subset $\cs \subseteq \vertex(\mathcal{F}(\vx^*))$ and point $\vx \in \mathcal{F}(\vx^*)$ with $\vx \in \co(\cs)$ and $ w(\vx, \cs) \leq w_c$ it follows that $\vx^* \in \co(\cs)$.
\end{fact}

\begin{remark} \label{appx:remark:criticalWolfe}
The critical strong Wolfe gap in Fact~\ref{fact:appx:asConvergence}, is a crucial parameter in the coming proofs. However, like the strict complementarity parameter $\tau$ \cite{guelat1986some,garber2020revisiting} and the critical radius defined in \cite{diakonikolas2019lacg}, the critical strong Wolfe gap can be arbitrarily small. Fortunately, as we will show in the proofs to come, it only affects the length of the burn-in phase of the accelerated algorithm, and moreover this dependence is logarithmically. Below, we sketch a simple example for which one can compute an exact expression for $w_c$. We then give an upper and a lower bound on $w_c$ for any general problem of the form shown in Problem~\eqref{eq:problem}.

Consider minimizing $f(\vx) = \frac{1}{2}\norm{\vx}^2$ over the unit probability simplex, which we denote by $\Delta^n$. The function is smooth and strongly convex, and the minimizer of this function over the feasible region is given by $\vx^* = \mathbf{1} / n$, where $n$ is the dimensionality of the problem and $\mathcal{F}(\vx^*) = \Delta^n$. We can compute the exact expression for $w_c$ as:
\begin{align}
    w_c & = \inf\limits_{\substack{\cs \subset \vertex\left(\Delta^n\right) \\  \vx \in \co(\cs) \\ \vx^* \notin \co(\cs)}} \max\limits_{\substack{\vu \in \cs \\ \mathbf{v}\in \Delta^n}} \innp{\nabla f(\vx), \vu - \mathbf{v}} \\
    & = \inf\limits_{\substack{\cs \subset \vertex\left(\Delta^n\right) \\  \vx \in \co(\cs) \\ \vx^* \notin \co(\cs)}} \left( \max\limits_{\substack{\vu \in \cs}} \innp{\vx, \vu }  - \min\limits_{\substack{\mathbf{v}\in \Delta^n}} \innp{\vx, \mathbf{v}}\right) \label{eq:appx:first_ineq} \\
    & = \inf\limits_{\substack{\cs \subset \vertex\left(\Delta^n\right) \\  \vx \in \co(\cs) \\ \vx^* \notin \co(\cs)}}  \max\limits_{\substack{\vu \in \cs}} \innp{\vx, \vu }  \label{eq:appx:second_ineq}  \\
    & = \inf\limits_{\substack{\cs \subset \vertex\left(\Delta^n\right) \\  \vx \in \co(\cs) \\ \vx^* \notin \co(\cs)}}  \norm{\vx}_{\infty} \label{eq:appx:third_ineq}  \\
    & = \frac{1}{n - 1}.
\end{align}
Where the equality in Equation~\eqref{eq:appx:first_ineq} follows from $\nabla f(\vx) = \vx$ and Equation~\eqref{eq:appx:second_ineq} from the fact that as $\cs \subset \vertex\left(\Delta^n\right)$, $\vx^* \notin \co (\cs)$ and $\vx \in \co (\cs)$, then one or more of the components of $\vx$ are zero, and as the rest are positive, it follows that $\min_{\substack{\mathbf{v}\in \Delta^n}} \innp{\vx, \mathbf{v}} = 0$. The equality in Equation~\eqref{eq:appx:third_ineq} follows from the fact that the components of $\vx$ are non-negative and the elements of $\cs$ are the standard orthogonal basis vectors, and so $\max_{\substack{\vu \in \cs}} \innp{\vx, \vu }  = \norm{\vx}_{\infty}$. The last equality simply follows from the fact that the infinum is achieved for the set $\cs$ of largest cardinality that does not contain $\vx^*$, and for a point $\vx^* \in \co (\cs)$ that minimizes the infinity norm. We can now prove a lower bound on $w_c$ for the general case as follows:
\begin{align*}
    w_c & = \inf\limits_{\substack{\cs \subset \vertex\left(\mathcal{F}(\vx^*)\right) \\  \vx \in \co(\cs) \\ \vx^* \notin \co(\cs)}} w(\vx, \cs)  \geq  \inf\limits_{\substack{\cs \subset \vertex\left(\mathcal{F}(\vx^*)\right) \\  \vx \in \co(\cs) \\ \vx^* \notin \co(\cs)}} \frac{m}{2} \norm{\vx^* -\vx}^2
    \geq  \inf\limits_{\substack{\cs \subset \vertex\left(\cx\right) \\  \vx \in \co(\cs) \\ \vx^* \notin \co(\cs)}} \frac{m}{2} \norm{\vx^* -\vx}^2
    =  \frac{\mu}{2} r_c^2.
\end{align*}
The first inequality stems from the fact as the function is $m$-strongly convex then $w(\vx, \cs) \geq f(\vx) - f(\vx^*) \geq \frac{m}{2} \norm{\vx - \vx^*}$. The second inequality stems from the fact that $\vertex\left(\mathcal{F}(\vx^*)\right) \subset \vertex (\cx)$. The last equality comes from the definition of $r_c$ in \cite{diakonikolas2019lacg}. Regarding the upper bound, for convenience we denote away-vertex and the Frank-Wolfe vertex as $\vu = \argmax_{\vy \in \cs } \innp{\nabla f(\vx), \vy}$ and $\mathbf{v} = \argmin_{\vy \in \cx } \innp{\nabla f(\vx), \vy}$, respectively. This allows us to write:
\begin{align*}
    w_c & = \inf\limits_{\substack{\cs \subset \vertex\left(\mathcal{F}(\vx^*)\right) \\  \vx \in \co(\cs) \\ \vx^* \notin \co(\cs)}} \innp{\nabla f(\vx^*), \vu - \vx^*} + \innp{\nabla f(\vx^*), \vx^* - \mathbf{v}} + \innp{\nabla f(\vx) - \nabla f(\vx^*), \vu - \mathbf{v}} \\
    & = \inf\limits_{\substack{\cs \subset \vertex\left(\mathcal{F}(\vx^*)\right) \\  \vx \in \co(\cs) \\ \vx^* \notin \co(\cs)}}   \innp{\nabla f(\vx) - \nabla f(\vx^*), \vu - \mathbf{v}} \\
    & \leq \inf\limits_{\substack{\cs \subset \vertex\left(\mathcal{F}(\vx^*)\right) \\  \vx \in \co(\cs) \\ \vx^* \notin \co(\cs)}}  \norm{\nabla f(\vx) - \nabla f(\vx^*)} \norm{ \vu - \mathbf{v}} \\
    & \leq LD \inf\limits_{\substack{\cs \subset \vertex\left(\mathcal{F}(\vx^*)\right) \\  \vx \in \co(\cs) \\ \vx^* \notin \co(\cs)}}  \norm{\vx - \vx^*}.
\end{align*}
Where the first equality is simply obtained from rewriting $\innp{\nabla f(\vx), \vu - \mathbf{v}} $, and the second inequality comes from the fact that $\innp{\nabla f(\vx^*), \vu - \vx^*} = 0$ from the strict complementarity assumption and $\innp{\nabla f(\vx^*), \mathbf{v} - \vx^*} \geq 0$ from the optimality conditions for $\vx^*$. The first inequality in the previous chain comes from the application of the Cauchy-Schwarz inequality, and the second inequality from the $L$-smoothness of $f$. To see why this upper bound means that $w_c$ can be arbitrarily small for certain problems, consider the setup in Figure~\ref{appx:fig:critical}.

\begin{figure}[h]
\centering
\includegraphics[width=0.4\textwidth]{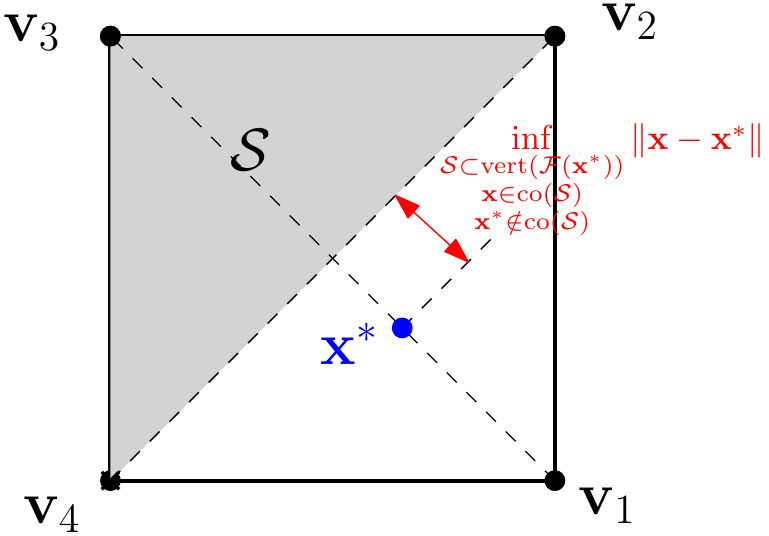} 
\caption{Example problem setup.} \label{appx:fig:critical}
\end{figure}

The figure shows a two dimensional example where $\mathcal{F}(\vx^*) = \cx = \co \left( \mathbf{v}_1,\mathbf{v}_2, \mathbf{v}_3, \mathbf{v}_4 \right)$. Moreover, the set $\cs = \co \left(\mathbf{v}_2, \mathbf{v}_3, \mathbf{v}_4\right)$ shown in light gray is contained in the optimal face, and does not contain $\vx^*$ (which is on the diagonal defined by $\mathbf{v}_1$ and $\mathbf{v}_3$ and just off the diagonal defined by $\mathbf{v}_2$ and $\mathbf{v}_4$). The distance from $\cs$ to $\vx^*$ is precisely $\inf_{\cs \subset \vertex\left(\mathcal{F}(\vx^*)\right),  \vx \in \co(\cs), \vx^* \notin \co(\cs)}  \norm{\vx - \vx^*}$, and it can be made arbitrarily small as $\vx^*$ approaches the diagonal defined by $\mathbf{v}_2$ and $\mathbf{v}_4$ along the diagonal defined by $\mathbf{v}_1$ and $\mathbf{v}_3$, which means that $w_c$ can become arbitrarily small.
\end{remark}

With these tools, we are ready to prove the desired convergence bound.

\begin{theorem} \label{theorem:appx:no-more-drop-steps}
  Assume the AFW algorithm (Algorithm~\ref{alg:fafw}) is run starting with $\vx_0 \in \vertex (\cx)$. If the strict complementarity assumption (Assumption~\ref{assumption:strictComplementarity}) is satisfied and $\vx^* \notin \vertex \left( \cx\right)$, then for $k \geq K_0$ with 
  $$
  K_0 = \frac{32L}{m \ln 2} \left(\frac{D}{\delta(\cx)}\right)^2 \log \left( \frac{2 w(\vx_0, \cs_0)}{\min\{  ( \tau /(2D(\sqrt{L/m} + 1)))^2/L, \tau, LD^2, 2 w_c \}}\right),
  $$
we have that $\vx_k \in \mathcal{F}(\vx^*)$, $\vx^* \in \co \left( \cs_k \right)$ and moreover:
  \begin{align*}
   w (\vx_k, \cs_k) \leq LD\sqrt{2(f\left(\vx_k\right) - f\left(\vx^*\right))/m}.
  \end{align*}
  Where $\delta\left( \cx\right)$ is the pyramidal width in Definition~\ref{defn:delta-scaling}, and $w_c>0$ is the critical strong Wolfe gap in Fact~\ref{fact:asConvergence}.
  \begin{proof}
  Using the primal gap bound in \citet[Theorem 1]{lacoste2015global}, and assuming that $\vx_0 \in \vertex (\cx)$, then for iterations 
  \begin{align*}
  k \geq \frac{8LD^2}{m\delta(\cx)^2} \log \left(\frac{2 w(\vx_o, \cs_0)}{\min\{ ( \tau /(2D(\sqrt{L/m} + 1)))^2/L, \tau, LD^2 \}}\right),
  \end{align*}
  we know that $$ f(\vx_k) -f(\vx^*) < \frac{1}{2} \min\left\{ ( \tau /(2D(\sqrt{L/m} + 1)))^2/L, \tau, LD^2 \right\}.$$ 
  Using the first claim in Theorem~\ref{theorem:drop-step-if-not-in-optimal-face}, as the algorithm will at most have picked up one vertex of the polytope $\cx$ per iteration, then the AFW algorithm will at most need to drop a number of vertices from its active set equal to the number of iterations that have elapsed before $\vx_k \in \mathcal{F}(\vx^*)$. This means that for \begin{align*}
  k \geq  \frac{16LD^2}{m\delta(\cx)^2} \log \left(\frac{2 w(\vx_0, \cs_0)}{\min\{ ( \tau /(2D(\sqrt{L/m} + 1)))^2/L, \tau, LD^2 \}}\right)
    \end{align*}
  we know that $\vx_k \in \mathcal{F}(\vx^*), $ 
  and, moreover, we have that $w (\vx_k, \cs_k) \leq LD\sqrt{2/m}\sqrt{f\left(\vx_k\right) - f\left(\vx^*\right)}$ by the second claim in Theorem~\ref{theorem:drop-step-if-not-in-optimal-face}. It remains to bound the number of iterations until $\vx^*\in\co \left( \cs_k\right)$ in the AFW algorithm. Using Fact~\ref{fact:asConvergence}, and the convergence guarantees in \citet{kerdreux2019restartfw}, we know that it suffices to run $32 LD^2 / (m \delta(\cx)^2 \ln 2 ) \log \left( w(\vx_0, \cs_0)/w_c\right)$ iterations of the AFW algorithm for $w(\vx_k, \cs_k)\leq w_c$. Consequently, we have that for $k \geq K_0$ it holds $\vx^* \in \co \left( \cs_k \right)$.
  \end{proof}
\end{theorem}

\section{Acceleration Analysis} \label{appx:section:acceralated-analysis}

Before delving into the analysis of the accelerated algorithm used in our work, we review some useful properties of gradient mapping and define its inexact variant.

\subsection{Gradient Mapping}

One of the properties that will be useful for our analysis is summarized in the following two propositions.
%
\begin{proposition}\label{prop:opt-gap-ub-via-grad-map}
  Let $\cc\subseteq \rr^n$ and let $f:\rr^n \to \rr$ be differentiable and $m$-strongly convex on $\cc,$ for some $m > 0.$ Then, $\forall \vx \in \cc,$ $\forall m' \geq m:$
  $$
        f(\vx) - \min_{\vu \in \cc} f(\vu) \leq \frac{1}{2m}\|G_m(\vx)\|^2 \leq \frac{1}{2m}\|G_{m'}(\vx)\|^2.
  $$
\end{proposition}
\begin{proof}
Fix any $\vx \in \cc.$ The first inequality in the proposition follows by using strong convexity of $f,$ which states that for all $\vu \in \cc$ we have that $f(\vu) - f(\vx) \geq \innp{\nabla f(\vx), \vu - \vx} + \frac{m}{2}\|\vu - \vx\|^2$, 
and minimizing both sides over $\vu \in \cc$. The second inequality is a consequence of the monotonicity of gradient mapping (see~\citet[Theorem~10.9]{beck2017optimization}).
\end{proof}

\begin{proposition}\label{prop:opt-gap-lb-via-grad-map}
  Let $\cc\subseteq \rr^n$ and let $f:\rr^n \to \rr$ be $L$-smooth on $\cc,$ for some $L < \infty.$ Then, $\forall \vx \in \cc,$ $\forall \eta \geq L:$
  $$
        f(\vx) - \min_{\vu \in \cc} f(\vu) \geq \frac{1}{2\eta}\|G_{\eta}(\vx)\|^2.
  $$
\end{proposition}
The proof is omitted and can instead be found in, e.g.,~\citet[Chapter 10]{beck2017optimization}.

To construct an implementable version of our algorithm, we rely on \emph{inexact} projections, which in turn lead to \emph{inexact} evaluations of the gradient mapping, defined below. This is the primary reason why we cannot rely on black box applications of any of the existing results, and instead require a variant of a parameter-free accelerated method that can work with inexact evaluations of the gradient mapping.

\begin{definition}[Inexact gradient mapping]
  \label{defn:inexact-gradient-mapping}
  Let $\cc$ be a closed convex set and let $f: \rr^n \rightarrow \rr$ be smooth and convex on $\cc$. For any $\vx \in \cc$ and $\eta > 0$, we say that $\tG_{\eta} (\vx) = \eta (\vx - \vy)$ is an $\epsilon^{\ell}$-approximate gradient mapping at $\vx$ for some $\vx \in \cc$, if $\vy \in \cc$ and $\ell (\vy) \le \min_{\vu \in \cc}\ell (\vu) + \epsilon^{\ell}$, where we define $\ell (\vu) = \inner{\nabla f (\vx)}{\vu - \vx} + \frac{\eta}{2} \norm{\vu - \vx}_2^2$.
\end{definition}

\subsection{Proof of Lemma~\ref{lemma:ACC-inner}}

In this section, we prove Lemma~\ref{lemma:ACC-inner}. To make it easier to follow the argument, we restate the lemma below. Much of the proof closely follows the proof of Lemma~3.2 in~\cite{diakonikolas2019lacg}. However, there are differences in both the algorithm and the proof itself required to obtain the bound for the gradient mapping, and we fill in the additional arguments that are needed to prove the lemma below.

\ACClemma*
\begin{proof}
The proof follows the Approximate Duality Gap Technique~\cite{diakonikolas2019approximate}, similar to~\cite{diakonikolas2019lacg}. The main idea is to construct an approximation of the true optimality gap $f_{\sigma}(\vy_k) - f_{\sigma}(\vx^*_{\sigma})$, where $\vx^*_{\sigma} = \argmin_{\vx \in \cc}f_{\sigma}(\vx)$ and show that it contracts at rate $1/A_k,$ where $A_k$ is made as fast growing as possible.

The approximate gap $\Gamma_k$ is constructed as $\Gamma_k = \Upsilon_k - \Lambda_k,$ where $\Upsilon_k = f_{\sigma}(\vy_k)$ and $\Lambda_k$ is defined by~\cite{diakonikolas2019lacg}
\begin{equation}\label{eq:lower-bnd}
    \begin{aligned}
        \Lambda_k = \frac{\sum_{i=0}^k a_i f_{\sigma}(\vx_i) + \min_{\vu \in \cc}m_k(\vu) - \frac{\eta_0\|\vx^*_{\sigma} - \vx_0\|^2}{2}}{A_k},
    \end{aligned}
\end{equation}
$$
    m_k(\vu) \defeq \sum_{i=0}^k a_i \innp{\nabla f_{\sigma}(\vx_i), \vu - \vx_i} + \sum_{i=0}^k a_i \frac{\sigma}{2}\|\vu - \vx_i\|^2 + \frac{\eta_0}{2}\|\vu - \vx_0\|^2.
$$
 By construction, $\Lambda_k \leq f_{\sigma}(\vx^*_{\sigma})$~\cite{diakonikolas2019lacg}, and, thus, $f_{\sigma}(\vx_k) - f_{\sigma}(\vx^*_{\sigma}) \leq \Gamma_k.$ Further, $ \argmin_{\vu \in \cc}m_k(\vu) = \argmin_{\vu \in \cc}M_k(\vu).$ 

Let $\mathbf{v}_k^* = \argmin_{\vu\in \cc}m_k(\vu),$ and observe that, by construction, $m_k(\mathbf{v}_k) - m_k(\mathbf{v}_k^*)\leq \epsilon_k^M.$  Following the same argument as in~\cite{diakonikolas2019lacg}, we have that
\begin{equation}\label{eq:init-gap}
    A_0\Gamma_0 \leq  \frac{\eta_0\|\vx^*_{\sigma} - \vx_0\|^2}{2} + \epsilon_0^M
\end{equation}
and
\begin{equation}\label{eq:change-in-m_k}
    \begin{aligned}
        m_{k}(\mathbf{v}_k^*)& - m_{k-1}(\mathbf{v}_{k-1}^*) \\
&\geq a_k \innp{\nabla f_{\sigma}(\vx_k), \mathbf{v}_k - \vx_k} + \frac{\sigma A_k}{4}\|\mathbf{v}_k - (1-\theta_k)\mathbf{v}_{k-1} - \theta_k \vx_k\|^2 - \epsilon^M_k - \epsilon^M_{k-1}.
    \end{aligned}
\end{equation}
From Algorithm~\ref{algo:ACC-iter}, $\vyh_k - \vx_k = \theta_k(\mathbf{v}_k - (1-\theta_k)\mathbf{v}_{k-1} - \theta_k \vx_k),$ and, thus, combining with Eq.~\eqref{eq:change-in-m_k} and the definition of the lower bound in Eq.~\eqref{eq:lower-bnd}, it follows that:
\begin{equation}\label{eq:lb-change}
    A_k\Lambda_k - A_{k-1}\Lambda_{k-1} \geq a_k f_{\sigma}(\vx_k) + a_k \innp{\nabla f_{\sigma}(\vx_k), \mathbf{v}_k - \vx_k} + \frac{\sigma A_k}{4{\theta_k}^2}\|\vyh_k - \vx_k\|^2 - \epsilon^M_k - \epsilon^M_{k-1}.    
\end{equation}
On the other hand, from $\Upsilon_k = f(\vy_k)$, we have
\begin{align}
    A_k \Upsilon_k - A_{k-1}\Upsilon_{k-1} =\;& A_k f_{\sigma}(\vy_k) - A_{k-1}f_{\sigma}(\vy_{k-1})\notag\\
    =\;& A_k(f_{\sigma}(\vy_k) - f_{\sigma}(\vyh_k))  + a_k f_{\sigma}(\vx_k)\notag\\
    &+ A_k(f_{\sigma}(\vyh_k) - f_{\sigma}(\vx_k)) + A_{k-1}(f_{\sigma}(\vx_k) - f_{\sigma}(\vy_{k-1})).\label{eq:change-in-ub-1}
\end{align}
By the exit criterion of the repeat-until loop in Algorithm~\ref{algo:ACC-iter}, we have $f_{\sigma}(\vyh_k) - f_{\sigma}(\vx_k) \leq \innp{\nabla f_{\sigma}(\vx_k), \vyh_k - \vx_k} + \frac{\eta_k + \sigma}{2}\|\vyh_k - \vx_k\|^2$, while by convexity of $f_{\sigma},$ $f_{\sigma}(\vx_k) - f_{\sigma}(\vy_{k-1}) \leq \innp{\nabla f_{\sigma}(\vx_k), \vx_k - \vy_{k-1}}.$ Plugging into Eq.~\eqref{eq:change-in-ub-1}, we have
\begin{align*}
    A_k \Upsilon_k - A_{k-1}\Upsilon_{k-1} \leq\;& A_k(f_{\sigma}(\vy_k) - f_{\sigma}(\vyh_k))  + a_k f_{\sigma}(\vx_k)\\
    &+ \innp{\nabla f_{\sigma}(\vx_k), A_k \vyh_k - A_{k-1}\vy_{k-1} - a_k \vx_k} + \frac{A_k (\eta_k + \sigma)}{2}\|\vyh_k - \vx_k\|^2.
\end{align*}
By definition of $\vyh_k,$ $A_k \vyh_k - A_{k-1}\vy_{k-1} - a_k \vx_k = a_k(\mathbf{v}_k - \vx_k).$ Thus, combining the last inequality with Eq.~\eqref{eq:lb-change}
\begin{align*}
    A_k\Gamma_k - A_{k-1}\Gamma_{k-1} \leq\;& A_k(f_{\sigma}(\vy_k) - f_{\sigma}(\vyh_k)) + \frac{A_k}{4}\Big(2(\eta_k + \sigma) - \frac{\sigma}{{\theta_k}^2}\Big)\|\vyh_k - \vx_k\|^2 + \epsilon_{k}^M + \epsilon_{k-1}^M\\
    \leq \;& A_k(f_{\sigma}(\vy_k) - f_{\sigma}(\vyh_k)) + \epsilon_{k}^M + \epsilon_{k-1}^M,
\end{align*}
as $\theta_k = \frac{a_k}{A_k} \leq \sqrt{\frac{\sigma}{2(\eta_k + \sigma)}}.$ 

To bound $f_{\sigma}(\vy_k) - f_{\sigma}(\vyh_k),$ observe that, by the exit condition in the repeat-until loop in Algorithm~\ref{algo:ACC-iter}, we have
\begin{align}\label{eq:acc-descent}
    f_{\sigma}(\vy_k) - f_{\sigma}(\vyh_k) \leq \innp{\nabla f_{\sigma}(\vyh_k), \vy_k - \vyh_k} + \frac{\eta_k + \sigma}{2}\|\vy_k - \vyh_k\|^2 = \ell_k(\vy_k).
\end{align}
Let $\vyb_k = \argmin_{\vu\in \cc}\ell_k(\vu)$, and let us argue now that $\ell_k(\vyb_k) \leq - \frac{1}{2(\eta_k+\sigma)}\|G_{\eta_k + \sigma}^{{\sigma}}(\vyh_k)\|^2.$ This can be proved using the definitions of $G_{\eta_k}^{{\sigma}}(\vyh_k)$ and $\vyb_k$. In particular, this property follows from
  \begin{align*}
      \ell_k(\vyb_k) &= \innp{\nabla f_{\sigma}(\vyh_k), \vyb_k - \vyh_k} + \frac{\eta_k + \sigma}{2}\|\vyb_k - \vyh_k\|^2\\
      &= \innp{\nabla f_{\sigma}(\vyh_k) - G_{\eta_k + \sigma}^{{\sigma}}(\vyh_k), \vyb_k - \vyh_k} - \frac{1}{2(\eta_k + \sigma)}\|G_{\eta_k + \sigma}^{{\sigma}}(\vyh_k)\|^2
  \end{align*}
  and
  \begin{align*}
      \innp{\nabla f_{\sigma}(\vyh_k) - G_{\eta_k + \sigma}^{{\sigma}}(\vyh_k), \vyb_k - \vyh_k} &= \innp{\nabla f_{\sigma}(\vyh_k) + (\eta_k + \sigma)(\vyb_k - \vyh_k), \vyb_k - \vyh_k} = \innp{\nabla \ell_k(\vyb_k), \vyb_k - \vyh_k} \leq 0,
  \end{align*}
  where the last inequality follows from the first-order optimality of $\vyb_k = \argmin_{\vy \in \cc} \ell_k(\vy).$
  
  Combining with Eq.~\eqref{eq:acc-descent}, and using that $\ell_k(\vy_k) - \ell_k(\vyb_k) \leq \epsilon_k^{\ell}$ (by definition of $\vy_k$), we have
  $$
        f_{\sigma}(\vy_k) - f_{\sigma}(\vyh_k) \leq -\frac{1}{2(\eta_k + \sigma)}\|G_{\eta_k}^{{\sigma}}(\vyh_k)\|^2 + \epsilon_k^{\ell}. 
  $$
  Thus, it follows that
  \begin{equation}\label{eq:change-in-gap-final}
      A_k \Gamma_k - A_{k-1}\Gamma_{k-1} \leq -\frac{A_k}{2(\eta_k + \sigma)}\|G_{\eta_k}^{{\sigma}}(\vyh_k)\|^2 + A_k\epsilon_k^{\ell} + \epsilon_{k}^M + \epsilon_{k-1}^M.
  \end{equation}
  Telescoping Eq.~\eqref{eq:change-in-gap-final} and using the bound on the initial gap from Eq.~\eqref{eq:init-gap}, we have
  \begin{align*}
      A_k \Gamma_k &\leq A_0 \Gamma_0 - \sum_{i=1}^k \frac{A_i}{2(\eta_i + \sigma)}\|G_{\eta_i + \sigma}^{{\sigma}}(\vyh_i)\|^2 + \sum_{i=0}^k (2\epsilon_i^M + A_i\epsilon_i^{\ell})\\
      &\leq -\frac{A_k}{2(\eta_k + \sigma)}\|G_{\eta_k + \sigma}^{{\sigma}}(\vyh_k)\|^2 + \frac{\eta_0\|\vx^*_{\sigma} - \vx_0\|^2}{2} + \sum_{i=0}^k (2\epsilon_i^M + A_i\epsilon_i^{\ell}).
  \end{align*}
  As $\Gamma_k \geq f_{\sigma}(\vy_k) - f_{\sigma}(\vx^*_{\cc}) \geq 0,$ we finally have
  \begin{equation}\label{eq:grad-map-gen}
      \frac{1}{\eta_k + \sigma}\|G_{\eta_k + \sigma}^{{\sigma}}(\vyh_k)\|^2 \leq \frac{2}{A_k} \Big(\frac{\eta_0\|\vx^*_{\sigma} - \vx_0\|^2}{2} + \sum_{i=0}^k (2\epsilon_i^M + A_i\epsilon_i^{\ell})\Big),
  \end{equation}
  which gives the first part of the lemma.

For the remaining part, as $f$ is $L$-smooth, it follows that the condition from the repeat-until loop is satisfied for any $\eta_k \geq L.$ As $\eta_k$ gets doubled each time the condition is not satisfied, we have $\eta_k \leq 2L,$ $\forall k.$ Thus the total number of times the repeat-until loop is entered is at most $k$ (the total number of iterations) plus $\log(\frac{2L}{\eta_0}).$ As each pass through the loop requires two gradient evaluations and two calls to a projection oracle, what remains to be shown is the stated bound on $k.$

Choosing $A_i \epsilon^{\ell}_i \leq \epsilon_i^M = \frac{a_i\epsilon^2}{8}$ and the already argued bound on $\eta_k,$ we have, from Eq.~\eqref{eq:grad-map-gen}:
\begin{gather}
    \frac{1}{\eta_k + \sigma}\|G_{\eta_k + \sigma}^{{\sigma}}(\vyh_k)\|^2 \leq \frac{2L\|\vx^*_{\sigma}- \vx_0\|^2}{A_k}  + \frac{3\epsilon^2}{4},
    %
\end{gather}
and it remains to bound the number of iterations $k$ until $\frac{2L\|\vx^*_{\sigma} - \vx_0\|^2}{A_k} \leq \frac{\epsilon^2}{4}$. To do so, notice that, as $a_0 = A_0 = 1$ and $\frac{a_k}{A_k} = \sqrt{\frac{\sigma}{2(\eta_k + \sigma)}} \leq \sqrt{\frac{\sigma}{2(2L+ \sigma)}}$ for $k \geq 1$, we have that $\frac{A_{k-1}}{A_k} \leq 1 - \sqrt{\frac{\sigma}{2(2L + \sigma)}},$ and thus, 
$$
    \frac{1}{A_k} = \frac{A_0}{A_1}\frac{A_1}{A_2}\dots  \frac{A_{k-1}}{A_k} \leq \bigg(1 - \sqrt{\frac{\sigma}{2(2L + \sigma)}}\bigg)^k \leq \exp\bigg(-k \sqrt{\frac{\sigma}{2(2L + \sigma)}}\bigg).
$$
Thus, $k$ is bounded by
$$
    k \leq \sqrt{\frac{2(2L + \sigma)}{\sigma}} \log \Big(\frac{16 L\|\vx^*_{\sigma} - \vx_0\|^2}{\epsilon^2}\Big) = O\bigg(\sqrt{\frac{L}{\sigma}}\log\Big(\frac{\sqrt{L}\|\vx^*_{\sigma} - \vx_0\|}{\epsilon}\Big)\bigg),
$$
as claimed.
\end{proof}

\subsection{Proof of Theorem~\ref{thm:ACC-full}}

We now proceed with proving Theorem~\ref{thm:ACC-full}. Before providing the full proof, we first state and prove some supporting claims.

\begin{proposition}\label{prop:grad-map-approx}
  Given $\vyh \in \cc,$ let $\ell(\vu) = \innp{\nabla f_{\sigma}(\vyh), \vu - \vyh} + \frac{\eta + \sigma}{2}\|\vu - \vyh\|^2,$ $\vyb = \argmin_{\vu \in \cc}\ell(\vu),$ and let $\vy$ be such that $\ell(\vy) - \ell(\vyb) \leq \epsilon^{\ell}.$ Then:
  $$
        \frac{1}{\eta + \sigma}\|\tG_{\eta + \sigma}^{\sigma}(\vyh) - G_{\eta + \sigma}^{\sigma}(\vyh)\|^2 \leq 2 \epsilon^{\ell},
  $$
  where $\tG_{\eta + \sigma}^{\sigma}(\vyh) = (\eta + \sigma)(\vyh - \vy),$ $G_{\eta + \sigma}^{\sigma}(\vyh) = (\eta + \sigma)(\vyh - \vyb)$.
\end{proposition}
\begin{proof}
By definition of $\tG_{\eta + \sigma}^{\sigma}(\vyh), G_{\eta + \sigma}^{\sigma}(\vyh)$, 
$$
    \|\tG_{\eta + \sigma}^{\sigma}(\vyh) - G_{\eta + \sigma}^{\sigma}(\vyh)\| = (\eta + \sigma)\|\vy - \vyb\|.
$$
As $\ell$ is strongly convex and by definition of $\vy,$
$$
    \frac{\eta + \sigma}{2}\|\vy - \vyb\|^2 \leq \ell(\vy) - \ell(\vyb) \leq \epsilon^{\ell}.
$$
Thus, $\|\tG_{\eta + \sigma}^{\sigma}(\vyh) - G_{\eta + \sigma}^{\sigma}(\vyh)\| \leq (\eta + \sigma)\sqrt{\frac{2 \epsilon^{\ell}}{\eta + \sigma}}$, as claimed.
\end{proof}

It is always possible to obtain an upper estimate of the strong convexity parameter $m$ of $f.$ To do so, we can pick two arbitrary points $\vx, \vy,$ $\vy \neq \vx,$ from the feasible set $\cc$. By strong convexity,
$$
    f(\vy) \geq \innp{\nabla f(\vx), \vy - \vx} + \frac{m}{2}\|\vy - \vx\|^2.
$$
Thus, we can use the estimate:
\begin{equation}\label{eq:init-sigma}
    \sigma_0 = \frac{2(f(\vy) - f(\vx) - \innp{\nabla f(\vx), \vy - \vx})}{\|\vy - \vx\|^2} \geq m. 
\end{equation}
Note that by smoothness of $f,$ we also have $\sigma_0 \leq L.$ 
We will assume throughout this section that the algorithm is started with such an estimate. 

Further, we assume that the algorithm can be started with points $\vx_0, \vy_0$ such that 
$$
\ell_0(\vy_0) - \min_{\vu\in \cc}\ell_0(\vu) \leq \frac{1}{8(\eta_0 + \sigma_0)} \|\tG_{\eta_0 + \sigma_0}^{\sigma_0}(\vx_0)\|^2 = \frac{\eta_0 + \sigma_0}{8}\|\vy_0 - \vx_0\|^2.
$$ 
This can be achieved by running a minimization procedure for $\ell_0$ and halting it when the estimated optimality gap at its current iterate $\vy_{0, k}$ is lower than $\frac{\eta_0 + \sigma_0}{8}\|\vy_0^k - \vx_0\|^2$, and outputting $\vy_0 = \vy_{0, k}.$ For implementing a minimization procedure for $\ell_0$ in our context, see~\cite{diakonikolas2019lacg}. Note also that, as $f_{\sigma_0}(\vx_0) = f(\vx_0),$ we have $G_{\eta_0 + \sigma_0}^{\sigma_0}(\vx_0) = G_{\eta_0 + \sigma_0}(\vx_0).$ 

We now present the pseudocode for one call to the accelerated algorithm, between restarts.

\begin{algorithm}
  \caption{ACC($\vx_0, \eta_0, \sigma$)}\label{algo:ACC}
  \begin{algorithmic}[1]
    \State $\sigma = 2 \sigma$
    \Repeat 
    \State $\sigma = \sigma / 2$
    \State Run a minimization procedure for $\ell_0(\vu) = \innp{\nabla f(\vx_0), \vu - \vx_0} + \frac{\eta_0 + \sigma}{2}\|\vu - \vx_0\|^2.$ Halt when the current iterate $\vy$ of the procedure satisfies $\ell_0(\vy) - \min_{\vu \in \cc}\ell_0(\vu) \leq \epsilon_0,$ where $\epsilon_0 = \frac{\eta_0 + \sigma}{32}\|\vy_0 - \vx_0\|^2 = \frac{1}{32(\eta_0 + \sigma)}\|\tG_{\eta_0 + \sigma}(\vx_0)\|^2.$ 
    \State Set $\vyh_0 = \mathbf{v}_0 = \vy_0$; $\vz_0 = (\eta_0 + \sigma)\vx_0 - \nabla f(\vx_0)$
    \State $a_0 = A_0 = 1$
        \Repeat
            \State $k = k + 1$
            \State $\eta_k, A_k, \vz_k, \mathbf{v}_k, \vyh_k, \vy_k, \tG_{\eta_k + \sigma}^{\sigma}(\vyh_k) = \mathrm{AGD-Iter}(\vy_{k-1}, \mathbf{v}_{k-1}, \vz_{k-1}, A_{k-1}, \eta_{k-1}, \sigma, \epsilon_0, \eta_0)$
        \Until{$\frac{1}{\eta_k + \sigma}\|\tG_{\eta_k + \sigma}^{\sigma}(\vyh_k)\|^2 \leq \frac{9\epsilon_0}{4}$}
    \Until{$\frac{\sigma}{\sqrt{\eta_k + \sigma}}\|\vyh_k - \vx_0\| \leq \sqrt{\epsilon_0}$}
    \State\Return $\vyh_k, \eta_k \sigma$
  \end{algorithmic}
\end{algorithm}

\thmACC*
\begin{proof}
The proof outline is as follows. We focus on one call to Algorithm~\ref{algo:ACC} to prove that it halves the value of $\frac{1}{\eta + \sigma}\|G(\vx^{\mathrm{out}})\|^2$ within $k = O\Big(\sqrt{\frac{L}{\sigma}}\log\Big(\frac{L}{m}\Big)\Big)$ iterations. The stated bound then follows by showing that $\sigma$ cannot be halved by more than $O(\log(\frac{L}{m}))$ times (cumulatively over the entire restarted algorithm run). Note that we have already argued in Lemma~\ref{lemma:ACC-inner} that the total number of times that $\eta_k$ can get doubled is $O(\log(\frac{\eta_k}{\eta_0})) = O(\log(\frac{L}{m})),$ this bound holds cumulatively over the entire (restarted) algorithm run, and is absorbed in the bound on $K$ from the statement of the theorem.

Let us start by bounding the total number of iterations in the inner repeat-until loop of Algorithm~\ref{algo:ACC}. To keep the notation simple, we will use the same notation as in the pseudocode for Algorithm~\ref{algo:ACC}. As the algorithm is started with $\sigma \leq \eta_0$ and in any iteration $\eta_k$ can be only increased while $\sigma$ can only be decreased, we have that $\theta_k = \sqrt{\frac{\sigma}{2(\eta_k + \sigma)}}\leq \frac{1}{2}.$ Thus, $\epsilon_k^{\ell} \leq \frac{\epsilon_0}{8}.$ From Proposition~\ref{prop:grad-map-approx}, 
\begin{equation}\label{eq:ACC-grad-map-err}
    \frac{1}{\sqrt{\eta_k + \sigma}}\|\tG_{\eta_k + \sigma}^{\sigma}(\vyh_k) - G_{\eta_k + \sigma}^{\sigma}(\vyh_k)\| \leq \sqrt{2\epsilon_0^k} \leq \frac{\sqrt{\epsilon_0}}{2}. 
\end{equation}
Suppose that $\frac{1}{\eta_k + \sigma}\|G_{\eta_k + \sigma}^{\sigma}(\vyh_k)\|^2 \leq \epsilon_0.$ Then, Eq.~\eqref{eq:ACC-grad-map-err} implies that it must be
$$
    \frac{1}{\eta_k + \sigma}\|\tG_{\eta_k + \sigma}^{\sigma}(\vyh_k)\| \leq \frac{9\epsilon_0}{4},
$$
and the inner repeat-until loop reaches its exit condition. Applying Lemma~\ref{lemma:ACC-inner}, this happens within
\begin{equation}
    k = O\bigg(\sqrt{\frac{L}{\sigma}}\log\Big(\frac{\sqrt{L}\|\vx^*_{\sigma} - \vx_0\|}{\sqrt{\epsilon_0}}\Big)\bigg)
\end{equation}
iterations. To further bound $k,$ we have by Proposition~\ref{prop:opt-gap-ub-via-grad-map} and $(\sigma+m)$ strong convexity of $f_{\sigma}$ that
$$
    \frac{\sigma + m}{2}\|\vx^*_{\sigma} - \vx_0\|^2 \leq f_{\sigma}(\vx_0) - f_{\sigma}(\vx^*_{\sigma}) \leq \frac{1}{2(\sigma + m)}\|G_{m+\sigma}^{\sigma}(\vx_0)\|^2 \leq \frac{1}{2(\sigma + m)}\|G_{\eta_0 + \sigma}(\vx_0)\|^2,
$$
where the last inequality follows by $\eta_0 \geq m$ and $G_{\eta_0 + \sigma}^{\sigma}(\vx_0) = G_{\eta_0 + \sigma}(\vx_0)$. Thus,
\begin{equation}\label{eq:ACC-init-dist}
    \|\vx^*_{\sigma} - \vx_0\| \leq \frac{1}{\sigma + m}\|G_{\eta_0 + \sigma}(\vx_0)\|.
\end{equation}
By the definition of $\epsilon_0$ is Algorithm~\ref{algo:ACC} and Proposition~\ref{prop:grad-map-approx}, 
\begin{equation}\label{eq:ACC-init-gm-error}
    \frac{1}{\sqrt{\eta_0 + \sigma}}\|G_{\eta+\sigma}(\vx_0) - \tG_{\eta+\sigma}(\vx_0)\|\leq \sqrt{2\epsilon_0},
\end{equation}
and, thus,
$$
    \frac{1}{\sqrt{\eta_0 + \sigma}}\|G_{\eta+\sigma}(\vx_0)\| \leq \frac{1}{\sqrt{\eta_0 + \sigma}}\|\tG_{\eta+\sigma}(\vx_0)\| + \sqrt{2\epsilon_0} = 3\sqrt{2{\epsilon_0}}.
$$
Thus, Eq.~\eqref{eq:ACC-init-dist} implies 
\begin{equation}\label{eq:ACC-init-dist-2}
     \|\vx^*_{\sigma} - \vx_0\| \leq \frac{\sqrt{\eta_0 + \sigma}}{\sigma + m}3\sqrt{2{\epsilon_0}} = O\Big(\frac{\sqrt{L}}{m}\sqrt{\epsilon_0}\Big),
\end{equation}
which leads to 
$$
    k = O\bigg(\sqrt{\frac{L}{\sigma}}\log\Big(\frac{L}{m}\Big)\bigg).
$$
To complete the proof, it remains to:
\begin{itemize}
    \item Argue that $\sigma \geq c m$ throughout the algorithm run, for some absolute constant $c > 0$; note that as initially $\sigma \in [m, L]$ and $\sigma$ only gets halved (but never increased), the total number of times that the outer repeat-until loop is accessed beyond the first pass is $O(\log(\frac{L}{m})),$ which, combined with the bound on the number of iterations $k$ above, gets absorbed by the stated bound on the total iteration count; 
    \item Argue that each call to Algorithm~\ref{algo:ACC} reduces the value of $\frac{1}{{\eta_k + \sigma}}\|G_{\eta_k + \sigma}(\vyh_k)\|^2$ by a constant factor. Then, the total number of calls to Algorithm~\ref{algo:ACC} until $\|G_{\eta_k + \sigma}(\vyh_k)\| \leq \epsilon$ is $O(\log(\frac{\|G_{\eta_0 + \sigma_0}(\vx_0)\|}{\epsilon}))$.
\end{itemize}

To argue about the lower bound on $\sigma,$ let us bound the value of $\|\vyh_k - \vx_0\|.$ Applying triangle inequality,
\begin{equation}\label{eq:ACC-yhk-x_0-1}
    \|\vyh_k - \vx_0\| \leq \|\vyh_k - \vx^*_{\sigma}\| + \|\vx^*_{\sigma} - \vx_0\| \leq \|\vyh_k - \vx^*_{\sigma}\| + \frac{\sqrt{\eta_0 + \sigma}}{\sigma + m}3\sqrt{2{\epsilon_0}}.
\end{equation}
Using strong convexity of $f_{\sigma}$ and applying Proposition~\ref{prop:opt-gap-ub-via-grad-map}, similar to what we did for bounding $\|\vx^*_{\sigma} - \vx_0\|,$ we have
$$
    \|\vyh_k - \vx^*_{\sigma}\| \leq \frac{1}{\sigma + m}\|G_{\eta_k + \sigma}\|\leq \frac{2\|\tG_{\eta_k + \sigma}(\vyh_k)\| + \sqrt{(\eta_k + m)\epsilon_0}}{2(\sigma + m)} \leq \frac{2\sqrt{(\eta_k + m)\epsilon_0}}{\sigma + m},
$$
where the last two inequalities are by Eq.~\eqref{eq:ACC-grad-map-err} and the exit condition of the inner repeat-until loop in Algorithm~\ref{algo:ACC}. Combining with Eq.~\eqref{eq:ACC-yhk-x_0-1}, we have
\begin{equation}\label{eq:ACC-yhk-x_0-2}
    \|\vyh_k - \vx_0\| \leq \frac{(2\sqrt{\eta_k + m} + 3\sqrt{2}\sqrt{\eta_0 + \sigma})}{\sigma + m}\sqrt{\epsilon_0}.
\end{equation}
It follows that, as $\eta_k \geq \eta_0,$
\begin{align*}
    \frac{\sigma}{\sqrt{\eta_k + \sigma}}\|\vyh_k - \vx_0\| \leq \frac{5\sqrt{2}\sigma}{\sigma + m}\sqrt{\epsilon_0},
\end{align*}
which is bounded above by $\sqrt{\epsilon_0}$ for $\sigma \leq \frac{m}{5\sqrt{2} -1},$ and, thus, satisfies the exit condition in the outer repeat-until loop in Algorithm~\ref{algo:ACC}. As $\sigma$ can never decrease by more than a factor of 2, we have that it always holds that $\sigma \geq \frac{m}{2(5\sqrt{2} -1)},$ completing the proof that $\sigma$ is bounded below by a constant factor times $m.$

For the remaining part of the proof, using the definition of a gradient mapping and $f_{\sigma}(\vx) = f(\vx) + \frac{\sigma}{2}\|\vx - \vx_0\|^2,$ we have
\begin{align*}
    \|G_{\eta_k + \sigma}(\vyh_k) - G_{\eta_k + \sigma}^{\sigma}(\vyh_k)\| &= (\eta_k + \sigma)\Big\|P_{\cc}\Big(\vyh_k - \frac{1}{\eta_k + \sigma}\nabla f(\vyh_k)\Big) - P_{\cc}\Big(\vyh_k - \frac{1}{\eta_k + \sigma}\nabla f_{\sigma}(\vyh_k)\Big)\Big\|\\
    &\leq \|\nabla f(\vyh_k) - \nabla f_{\sigma}(\vyh_k)\|\\
    & = \sigma\|\vyh_k - \vx_0\|,
\end{align*}
where we have used the fact that the projection operator is non-expansive. Now, using Eq.~\eqref{eq:ACC-grad-map-err} and the exit condition in the inner repeat-until loop in Algorithm~\ref{algo:ACC}, we have
\begin{align*}
    \frac{1}{\sqrt{\eta_k + \sigma}}\|G_{\eta_k + \sigma}(\vyh_k)\| &\leq \frac{1}{\sqrt{\eta_k + \sigma}}\|G_{\eta_k + \sigma}^{\sigma}(\vyh_k)\| + \frac{\sigma}{\sqrt{\eta_k + \sigma}}\|\vyh_k - \vx_0\|\\
    &\leq \frac{1}{\sqrt{\eta_k + \sigma}}\|\tG_{\eta_k + \sigma}^{\sigma}(\vyh_k)\| + \frac{\sqrt{\epsilon_0}}{2} + \frac{\sigma}{\sqrt{\eta_k + \sigma}}\|\vyh_k - \vx_0\|\\
    &\leq 3\sqrt{\epsilon_0},
\end{align*}
where the last inequality uses $\frac{\sigma}{\sqrt{\eta_k + \sigma}}\|\vyh_k - \vx_0\| \leq \sqrt{\epsilon_0},$ which holds by the exit criterion of the repeat-until loop. Using the definition of $\epsilon_0,$ we further have
\begin{align*}
    \frac{1}{\sqrt{\eta_k + \sigma}}\|G_{\eta_k + \sigma}(\vyh_k)\| &\leq \frac{3}{\sqrt{32}\sqrt{\eta_0 + \sigma}}\|\tG_{\eta_0 + \sigma}(\vx_0)\|.
\end{align*}
Applying Eq.~\eqref{eq:ACC-init-gm-error} and the definition of $\epsilon_0,$ a simple calculation leads to $\|\tG_{\eta_0 + \sigma}(\vx_0)\| \leq \frac{1}{1 - 1/\sqrt{32}}\|G_{\eta_0 + \sigma}(\vx_0)\|$, and we can finally conclude
\begin{align*}
    \frac{1}{\sqrt{\eta_k + \sigma}}\|G_{\eta_k + \sigma}(\vyh_k)\| &\leq \frac{3}{\sqrt{32}-1\sqrt{\eta_0 + \sigma}}\|G_{\eta_0 + \sigma}(\vx_0)\| \leq \frac{1}{\sqrt{2}}\|G_{\eta_0 + \sigma}(\vx_0)\|.
\end{align*}
Thus, each call to Algorithm~\ref{algo:ACC} halves $\frac{1}{\eta_k + \sigma}\|G_{\eta_k + \sigma}(\vx_k^{\mathrm{out}})\|,$ and the total number of calls until $\|G_{\eta_k + \sigma}(\vx_k^{\mathrm{out}})\|\leq \epsilon$ is bounded by $O(\log(\frac{L}{m}\frac{\|G_{\eta_0 + \sigma_0}\|}{\epsilon})),$ leading to the bound from the statement of the theorem.
\end{proof}

\section{Coupling AFW with Acceleration} \label{appx:section:PFLaCG}

In this section, we prove how the AFW algorithm (Algorithm~\ref{alg:fafw}) can be coupled with the ACC algorithm (Algorithm~\ref{algo:ACC}) to achieve an algorithm, dubbed the Parameter-Free Locally accelerated Conditional Gradients (PF-LaCG), that achieves an optimal convergence rate in primal gap (up to poly-logarithmic factors), this is formalized in Theorem~\ref{thm:appx:main}.

\begin{theorem}\label{thm:appx:main}
  Let $\cx \in \rr^n$ be a closed convex polytope of diameter $D$, and let $f:\rr^n \to \rr$ be a function that is $L$-smooth and $m$-strongly convex on $\cx.$ Let Assumptions~\ref{defn:delta-scaling} and~\ref{assumption:strictComplementarity} be satisfied for $f, \cx$. Denote $\vx^{*} = \argmin_{\vx \in \cx}f(\vx).$ Given $\epsilon > 0,$ let $\vx^{\out}$ be the output point of PF-LaCG (Algorithm~\ref{algo:PF_LaCG}), initialized at an arbitrary vertex $\vx_0$ of $\cx$. Then $w(\vx^{\out}, \cs^{\out}) \leq \epsilon$ and PF-LaCG uses a total of at most
  \begin{align*}
        K = O\bigg(\min\bigg\{& \log \left( \frac{w (\vx_0, \cs_0)}{LD^2} \right) + \frac{LD^2}{m \delta^2} \log \left( \frac{w (\vx_0, \cs_0)}{\epsilon} \right), K_0 + K_1 + \sqrt{\frac{L}{m}}\log\left(\frac{L}{m}\right)\log\Big(\frac{LD}{m\delta}\Big)\log\left(\frac{LD}{\epsilon}\right)\bigg\}\bigg)
  \end{align*}
  queries to the FOO for $f$ and the LMO for $\cx$, where 
  $$
    K_0 = \frac{32LD^2}{m \delta(\cx)^2\ln 2} \log \left( \frac{2 w(\vx_0, \cs_0)}{\min\{ ( \tau /(2D(\sqrt{L/m} + 1)))^2/L, \tau, LD^2, 2 w_c \}}\right),$$
  and $K_1 = \frac{128 LD^2}{m\delta^2}$. 
\end{theorem}
\begin{proof}
As the algorithm is specified so that we always have $w^{\out} = w(\vx^{\out}, \cs^{\out})$ and it terminates when $w^{\out} \leq \epsilon,$ it must be $w(\vx^{\out}, \cs^{\out}) \leq \epsilon$ when the algorithm terminates.

Observe that the algorithm monotonically decreases $w(\vx^{\out}, \cs^{\out})$ between restarts (by a factor of 2) and $w^{\out}$ is never larger than $w^{\AFW}$ at the end of a restart, as the algorithm sets $w^{\out} = \min\{w^{\AFW}, w^{\ACC}\}$. Further, AFW is run almost completely independently of ACC: its running iterate and the active set are updated to those of ACC only if $w^{\ACC} < w^{\AFW}$ and $|\cs^{\ACC}| \leq |\cs^{\AFW}|.$ Thus $w^{\AFW}$ decreases at least as fast as guaranteed by AFW, and the same holds for $w^{\out},$ as $w^{\out}\leq w^{\AFW}.$ This gives the first term in the bound on $K$ from the theorem statement.

By Theorem~\ref{theorem:no-more-drop-steps} we know that $w(\vx_k, S) \leq w_c$ for $k \geq K_0$, which means that every active set of AFW contains $\vx^*$ in its convex hull. As, due to Proposition~\ref{prop:fafw-halve-rate}, $w^{\AFW}$ is halved after at most $K_1 = \frac{128 LD^2}{m\delta^2}$ additional iterations, between iterations $K_0$ and $K_0 + K_1$, PF-LaCG enters the if branch from Line~\ref{Ln:PF-LaCG-restart}. Then if $w^{\ACC} \leq w_c$, we know that $\vx^* \in \co \left(S^{\ACC}\right)$. Otherwise, if $w^{\ACC} > w_c$, then the AFW algorithm will obtain a point such that $w^{\AFW} < w^{\ACC}$ and the ACC algorithm will get updated with $\vx^{\ACC} = \vx^{\AFW}$ and $\cs^{\ACC} = \cs^{\AFW}$. From this point on, it must be the case that $\vx^* \in \co \left(S^{\ACC}\right)$, as PF-LaCG can only update $\cs^{\ACC}$ to $\cs^{\AFW}$, and for $t\geq K_0$, we have $\vx^* \in \co \left( \cs^{\AFW}\right)$. It remains to argue that in the remaining iterations $w^{\out}$ is reduced at an accelerated rate. 

Let us consider what happens between two successive restarts of PF-LaCG. 
Suppose first that there are $r$ calls to ACC (Algorithm~\ref{algo:ACC}) in this time frame, and, to keep the notation simple, let $\vx$ and $\vx^+$ denote the output points of ACC at the beginning and at the end of the considered restart period of LaCG, and $(\eta, \sigma)$ and $(\eta^+, \sigma^+)$ denote their respective smoothness and strong convexity parameter estimates from ACC. For simplicity, we let $\vx^+$ coincide with the output point of the $r^{\mathrm{th}}$ call to ACC; if this were not the case, we could choose to output the point with the lower value of the gradient mapping between $\vx^+$ and the output point of the $r^{\mathrm{th}}$ call to ACC, and the same bound would hold. Theorem~\ref{thm:ACC-full} guarantees that
\begin{equation}\label{eq:ACC-gm-reduction-r-restarts}
    \frac{1}{\eta^+ + \sigma^+}\|G_{\eta^+ + \sigma^+}(\vx^+)\|^2 \leq \Big(\frac{1}{2}\Big)^r \frac{1}{\eta + \sigma}\|G_{\eta + \sigma}(\vx)\|^2. 
\end{equation}
By Proposition~\ref{prop:opt-gap-lb-via-grad-map}, if $\eta + \sigma \geq L,$ then $\frac{1}{2(\eta + \sigma)}\|G_{\eta + \sigma}(\vx)\|^2 \leq f(\vx) - f(\vx^*)$. Otherwise, using monotonicity of gradient mapping (see Proposition~\ref{prop:opt-gap-ub-via-grad-map}), $\frac{1}{2(\eta + \sigma)}\|G_{\eta + \sigma}(\vx)\|^2 \leq \frac{1}{2(\eta + \sigma)}\|G_{L}(\vx)\|^2 \leq \frac{L}{\eta + \sigma}(f(\vx) - f(\vx^*)).$ Either way, as $\eta \geq m$:
\begin{align}
    \frac{\|G_{\eta + \sigma}(\vx)\|^2}{2(\eta + \sigma)} &\leq \max\left\{1, L/(\eta + \sigma)\right\}(f(\vx) - f(\vx^*))\notag\\
    &\leq L  (f(\vx) - f(\vx^*)) / m \notag\\
    &\leq 2L (w^{\ACC}_{\mathrm{prev}}/(m \delta))^2, \label{eq:opt-gap-ub-via-strong-wg}
\end{align}
where the last inequality is by $w(\vx) \leq w(\vx, \vs) = w^{\ACC}_{\mathrm{prev}}$. On the other hand, as $\eta^+ + \sigma^+ > m,$ Proposition~\ref{prop:opt-gap-ub-via-grad-map} gives $f(\vx^+) - f(\vx^*) \leq \frac{1}{2m}\|G_{\eta^+ + \sigma^+}(\vx^+)\|^2.$ Using the strong Wolfe gap bound in Theorem~\ref{theorem:no-more-drop-steps} it follows that

\begin{equation}\label{eq:opt-gap-lb-via-strong-wg}
   \frac{1}{\eta^+ + \sigma^+}\left(\frac{m w^{\ACC}}{LD} \right)^2  \leq \frac{2m}{\eta^+ + \sigma^+}(f(\vx^+) - f(\vx^*)) \leq \frac{1}{\eta^+ + \sigma^+}\|G_{\eta^+ + \sigma^+)}(\vx^+)\|^2.
\end{equation}

Combining Eqs.~\eqref{eq:ACC-gm-reduction-r-restarts}--\eqref{eq:opt-gap-lb-via-strong-wg}, we have
\begin{equation*}
    (w^{\ACC})^2 \leq \left(\frac{1}{2}\right)^{r} \frac{4 L^3 D^2 (\eta^+ + \sigma^+)}{m^4 \delta^2} (w^{\ACC}_{\mathrm{prev}})^2.
\end{equation*}
Thus, if $r \geq r^* = \log_2\left(\frac{16 L^3 D^2 (\eta^+ + \sigma^+)}{m^4 \delta^2}\right) = O\big(\log\big(\frac{LD}{m\delta}\big)\big),$ we have that $w^{\ACC} \leq \frac{1}{2}w^{\ACC}_{\mathrm{prev}}.$ The total number of iterations in this case is $r \cdot O\left(\sqrt{\frac{L}{m}}\log(\frac{L}{m})\right),$ due to Lemma~\ref{lemma:ACC-inner}. 

Observe that if $r \geq r^* + p$ for some $p \geq 1,$ then
\begin{equation}\notag
    (w^{\ACC})^2 \leq \left(\frac{1}{2}\right)^p (w^{\ACC}_{\mathrm{prev}})^2,
\end{equation}
and past the first $r$ calls to ACC, $w^{\ACC}$ halves on every other call to ACC, i.e., it contracts much faster than every $r^*$ iterations. 

We now argue that $w^{\ACC}$ halves at least as often as every 
$$
r^* \cdot O\left(\sqrt{\frac{L}{m}}\log\left(\frac{L}{m}\right)\right) = O\left(\sqrt{\frac{L}{m}}\log\left(\frac{L}{m}\right)\log\left(\frac{LD}{m\delta}\right)\right)
$$ 
iterations. To do so, we need to argue that a restart that updates $\vx^{\ACC}$ to $\vx^{\AFW}$ does not slow down the overall convergence. Suppose that there is such a restart. Then, by the condition from Line~\ref{Ln:PF-LaCG-ACC-update} of PF-LaCG, one of the following two situations must  occur. If the restart period (number of calls to ACC) was longer than $r^*,$ then $w^{\ACC}$ was contracting between restarts at least as fast as if ACC was run independently, and, as $w^{\AFW} \leq w^{\ACC}$ on restart and $w^{\ACC}$ is updated to $w^{\AFW}$, $w^{\ACC}$ must halve at least as frequently as every $r^*$ calls to ACC. If the restart period was shorter than $r^*,$ then, as $w^{\AFW} \leq w^{\ACC}_{\mathrm{prev}}/2$ and $w^{\ACC}$ is updated to $w^{\AFW},$ we get that in this case $w^{\ACC}$ is halved in fewer than $r^*$ calls to ACC. 

Finally, as $w^{\out} \leq w^{\ACC}$ and each iteration requires a constant number of calls to the FOO for $f$ and an LMO for $\cx,$ the claimed bound on $K$ follows.       
\end{proof}

\section{Computational Results} \label{appx:section:comp-results}

In this section we provide a complete overview of the implementation of PF-LaCG and a detailed comparison of the performance of PF-LaCG relative to other state of the art parameter-free algorithms. All the experiments in this paper were run on a Linux machine with an Intel Xeon Processor (Skylake, IBRS) and $64$ GB of RAM.

\subsection{Running ACC in Parallel with AFW}

One of the key computational advantages of PF-LaCG is that it allows local acceleration speedups through parallelism, guaranteeing nearly as least as much progress as state-of-the-art CG algorithms such as AFW and PFW in terms of wall-clock time. This is possible due to the way PF-LaCG is structured: between restarts where we use $w (\vx, \cs)$ as the measure of optimality, the ACC and the AFW algorithms are executed completely independently of each other while AFW checks whether ACC has made sufficient progress whenever AFW has halved $w(\vx, \cs)$ without interrupting ACC's execution. As such, PF-LaCG has the potential to utilize twice the computational power when compared to AFW by running the locally accelerated algorithm ACC on a secondary connected machine or on a separate process within a single machine. As a proof of concept in our experiments, we implement the parallelism of PF-LaCG using Python's \texttt{multiprocessing} library. We utilize its recently developed functionality \texttt{shared\_memory} to provide efficient inter-process communications whenever a restart happens and to orchestrate iteration synchronization through semaphores and locks. Moreover, in order to simulate the performance comparisons between PF-LaCG and other CG algorithms in the setting where PFLaCG has access to twice the computational power, we limit each process (note that PF-LaCG runs on two processes) to run on one virtualized CPU core through restricting the number of threads used by Intel's high-performance computing library \texttt{MKL} to one. In Section~\ref{Section:Appx:2coreperformance}, we also show how PF-LaCG outperforms other CG algorithms even when given equal amount of total computational power.

\subsection{Solving Minimization Subproblems in ACC and AGD-Iter}

Two of the key steps in the AGD-Iter algorithm are the approximate projections that need to be carried out in Algorithm~\ref{algo:ACC-iter}. We give a brief description of how these subproblems can be solved, in a very similar way as how they were solved in \citet[Section C.3]{diakonikolas2019lacg}. Both of these subproblems can be written without loss of generality as:
\begin{align}
    \mathbf{v} = \stackrel{\epsilon}{\sim}\argmin_{\vu \in \cc} \innp{\vz, \vu} + \norm{\vu}^2, \label{eq:appx:subproblem}
\end{align}
where $\cc$ is the convex hull of a known set of vertives, i.e., $\cc = \co \left( \cs \right)$, and $\stackrel{\epsilon}{\sim}\argmin$ is used to indicate that the objective function that follows the $\argmin$ is minimized to additive error $\epsilon$, as we described in the main body of the paper. We can write the subproblem shown in Eq.~\eqref{eq:appx:subproblem} as an equivalent problem over the unit probability simplex of dimension $\card{\cs}$, which we denote by $\Delta_{\card{\cs}}$, where $\card{\cs}$ is the cardinality of the set $\cs$. This allows us to write $\vu = \mathcal{V} \vlambda$, where $\mathcal{V}$ is the matrix whose columns are the elements of the set $\cs$ and $\vlambda \in \Delta_{\card{\cs}}$. This leads to $\mathbf{v} = \mathcal{V} \vlambda_{\mathbf{v}}$, where:
\begin{align}
    \vlambda_{\mathbf{v}} = \stackrel{\epsilon}{\sim}\argmin_{\vlambda \in \Delta_{\card{\cs}}} \innp{\vz, \mathcal{V} \vlambda} + \norm{\mathcal{V} \vlambda}^2. \label{eq:appx:subproblem-barycentric}
\end{align}
As noted earlier, because Euclidean projections onto $\Delta_{\card{\cs}}$ can be computed in closed-form with reasonable complexity, we can use accelerated projection-based methods to compute an $\epsilon$-optimal solution to Problem~\eqref{eq:appx:subproblem-barycentric} efficiently. In the code, we use the projections onto the probability simplex based on the \texttt{quicksort} algorithm with a worst-case complexity of $\mathcal{O}\left(\card{\cs} \log \left( \card{\cs} \right) \right)$ \cite{held1974validation, duchi2008efficient}. Note that there exist projections onto the simplex with worst-case complexity of $\mathcal{O}\left(\card{\cs}\right)$ using a variation of the \texttt{quicksort} algorithm that uses median-pivot partitionining \cite{condat2016fast}. However, we have used the projections onto the simplex using the standard \texttt{quicksort} algorithm due to its simplicity and the fact that there are fast and reliable implementations of the aforementioned sorting algorithm in Python.

In order to ensure that we reach an $\epsilon$-optimal solution to Problem~\eqref{eq:appx:subproblem-barycentric} we use one of the following two criteria:

\paragraph{\textbf{Frank-Wolfe gap stopping criterion. }} In order to compute the approximate solutions to Problem~\eqref{eq:appx:subproblem-barycentric} one could use the Frank-Wolfe gap as a stopping criterion, that is, we stop running the accelerated projection-based algorithm when $\vlambda_{\mathbf{v}}$ satisfies:
\begin{align}
    \max_{\vlambda \in \Delta_{\card{\cs}}} \innp{\mathcal{V}^T \vz + 2 \mathcal{V}^T \mathcal{V} \vlambda_{\mathbf{v}}, \vlambda_{\mathbf{v}} - \vlambda} \leq \epsilon \label{eq:appx:subproblem-FWgap}
\end{align}
Note that Problem~\eqref{eq:appx:subproblem-barycentric} is convex, and so the Frank-Wolfe gap provides a useful upper bound on the primal gap. Ensuring that the Frank-Wolfe gap is below the tolerance $\epsilon$ ensures that the primal gap is below the tolerance $\epsilon$ too. Moreover, note that the quantity $\mathcal{V}^T \vz + 2 \mathcal{V}^T \mathcal{V} \vlambda_{\mathbf{v}}$ is readily computed at each iteration, as it constitutes the gradient of the objective function at $\vlambda_{\mathbf{v}}$, and is used in the projection-based accelerated algorithm. Furthermore, solving a linear optimization problem like the one shown in Eq.~\eqref{eq:appx:subproblem-FWgap} over the probability simplex has complexity $\card{\cs}$, and so using the Frank-Wolfe as a stopping criterion does not add a noticeable overhead to the resolution of the subproblems.

\paragraph{\textbf{Gradient mapping stopping criterion.}}  Alternatively, one could use the norm of the gradient mapping for Problem~\eqref{eq:appx:subproblem-barycentric} 
to bound above the primal gap in the case where the objective function in the subproblem is strongly-convex. That is, if we denote the objective function being minimized in Eq.~\eqref{eq:appx:subproblem-barycentric} as $g(\vlambda)$, and we have that the smallest eigenvalue of the Hessian of $g(\vlambda)$, which we denote by $m$ for simplicity, is greater than zero, this means that if:
\begin{align*}
    \frac{m}{2} \norm{\vlambda_{\mathbf{v}} - P_{\Delta_{\card{\cs}}} \left(\vlambda_{\mathbf{v}} - \frac{1}{m} \nabla g(\vlambda_{\mathbf{v}})\right)}^2 \leq \epsilon,
\end{align*}
then the primal gap at $\vlambda_{\mathbf{v}}$ for Problem~\eqref{eq:appx:subproblem-barycentric} is also smaller than $\epsilon$. To compute this stopping criterion we require knowledge of the smallest eigenvalue of $\nabla^2 g(\vlambda) = 2 \mathcal{V}^T \mathcal{V}$, which is a quantity that is already computed in the projection-based accelerated gradient descent algorithms, as it is used to set the step size of the accelerated algorithms. Lastly, note that as we have mentioned above, there are efficient ways to compute closed-form projections onto the probability simplex $\Delta_{\card{\cs}}$, so computing $P_{\Delta_{\card{\cs}}} \left(\vlambda_{\mathbf{v}} - \frac{1}{m} \nabla g(\vlambda_{\mathbf{v}})\right)$ does not pose a high cost.

\begin{remark}[LLVM-enhanced subproblem solver]
We also use \texttt{Numba} \cite{lam2015numba}, a \emph{Just-In-Time} Python compiler that uses the LLVM compiler library that transforms the accelerated projection-based optimization algorithm to machine code, in order to more efficiently solve the subproblems from Eq.~\eqref{eq:appx:subproblem-barycentric}.
\end{remark}

\begin{remark}[Computing LMO]
For the problems without closed form solutions for LMO (such as structured LASSO and constrained matching), we use Scipy's \texttt{linprog} function to compute the LMO.
\end{remark}

\subsection{PF-LaCG over the Probability Simplex}
The structure of the unit probability simplex can give us deep insight into how the PF-LaCG algorithm works. Although it can be considered as a toy example, as projections onto this feasible region can be computed with a complexity equal to that of solving a linear program, it allows us to know exactly when the acceleration should kick in. Assume we are minimizing a smooth and strongly convex function $f(\vx)$ over the unit probability simplex $\Delta_n$ in $\rr^n$. As it is immediate to compute a proper support for any point $\vx\in \Delta_n$ (it suffices to find the non-zero elements in the vector $\vx$) we can easily find $\mathcal{F}\left( \vx^*\right)$ if we know $\vx^*$, or a high accuracy solution to the optimization problem. This enables us to pinpoint when the active set of the AFW algorithm is equal to the vertices of the optimal face, that is $\mathcal{S}_k = \vertex \left( \mathcal{F}\left( \vx^*\right)\right)$. We know from the proof of Theorem~\ref{thm:appx:main} that immediately following the restart after the iteration where we have that $\mathcal{S}_k = \vertex \left( \mathcal{F}\left( \vx^*\right)\right)$ we should observe that the ACC algorithm converges at an accelerated rate. 

We present in Figure~\ref{fig:simplex:Appx} a comparison of several CG algorithms minimizing a quadratic function over the unit probability simplex. The function being minimized in this example is $f(\vx) =  \vx^T \left( M^T M + \alpha \mathbf{1}_n \right)\vx/2 + \vb^T\vx$, where $M\in \rr^{n\times n}$ and $\vb\in \rr^{n}$ have entries sampled uniformly at random between $0$ and $1$ and $n = 10000$. The parameter $\alpha = 500$ is set so that the objective function satisfies $m \approx 500$. The resulting condition number is $L/m = 50000$, and the number of nonzero elements in $\vx^*$ is around $320$. The AFW algorithm in PF-LaCG (AFW) satisfies that $\vx^* \in \co \left( \cs_k \right)$ around iteration $400$, consequently we achieve the accelerated convergence rate from then onwards. The same can be said regarding PF-LaCG (PFW) around iteration $350$.

    \begin{figure*}[th!]
        \centering
        \hspace{\fill}
        \subfloat[$f(\vx_k) - f(\vx^*)$ vs iteration count]{{\includegraphics[width=7.35cm]{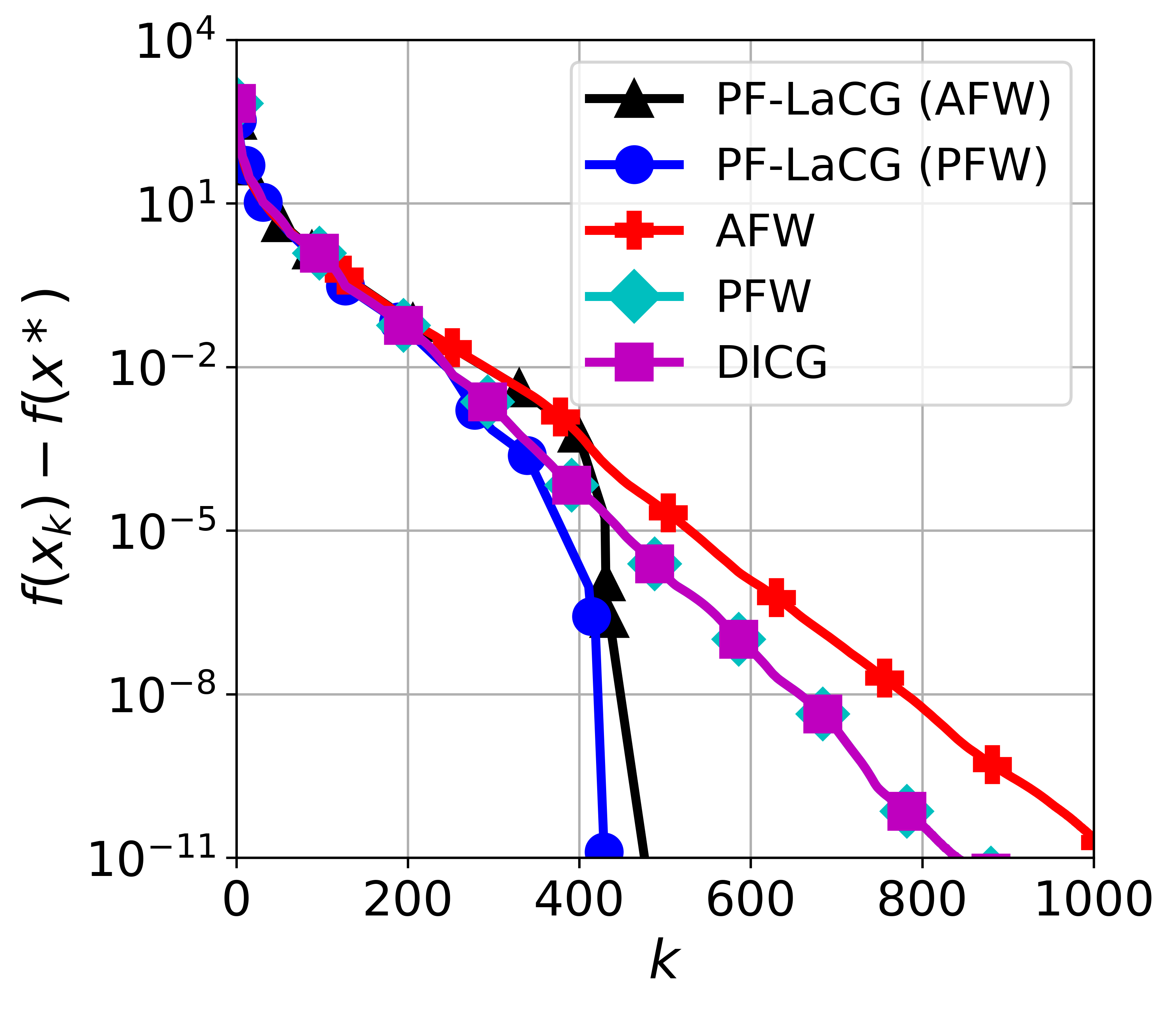} }\label{fig:simplexPGIt:Appx}}%
        \hspace{\fill}
        \subfloat[$f(\vx_k) - f(\vx^*)$ vs time (seconds)]{{\includegraphics[width=7.35cm]{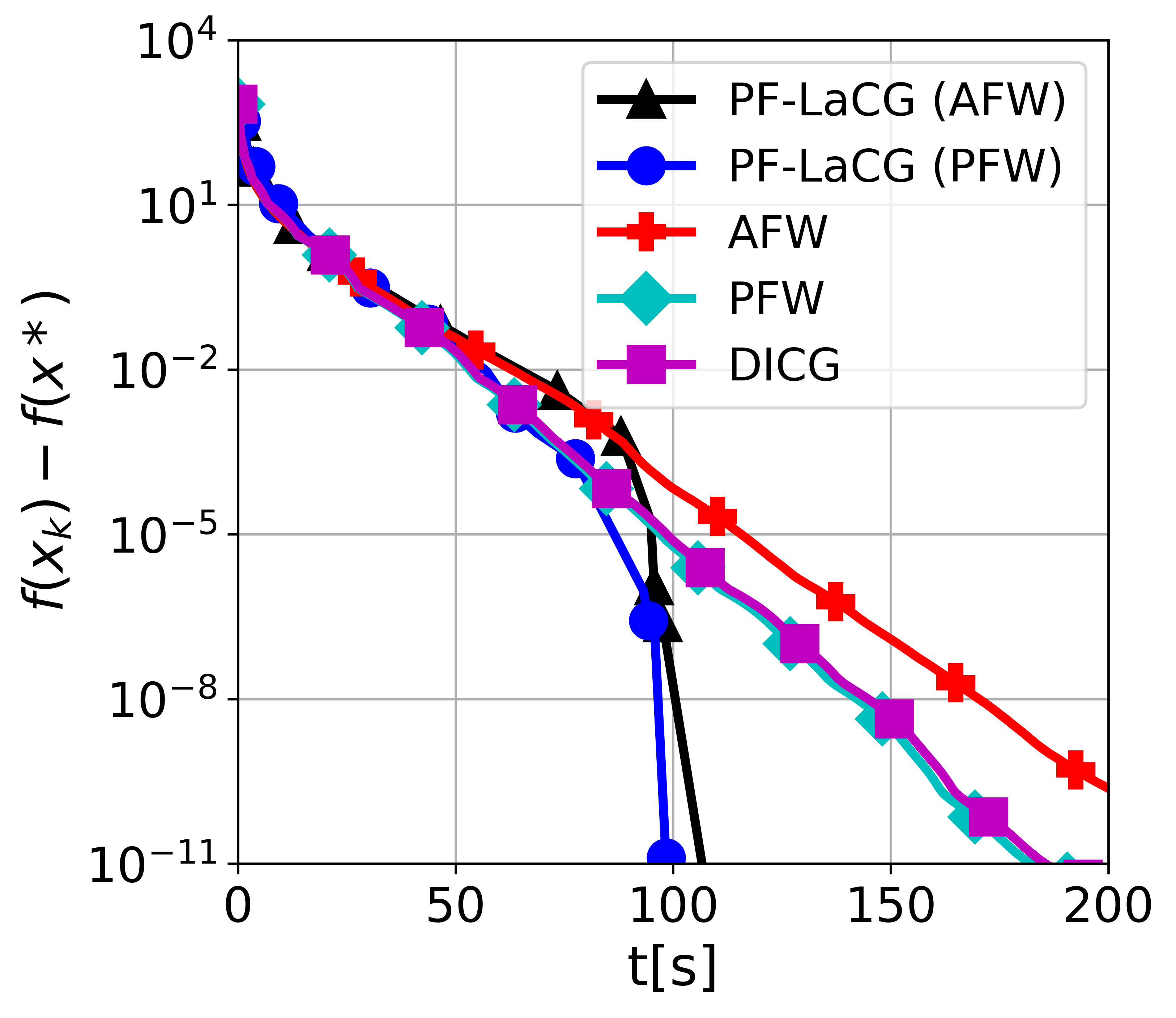} }\label{fig:simplexPGTime:Appx}}%
        \hspace{\fill}
        
        \bigskip 
        
        \hspace{\fill}
        \subfloat[$w(\vx_k, \cs_k)$ vs iteration count]{{\includegraphics[width=7.35cm]{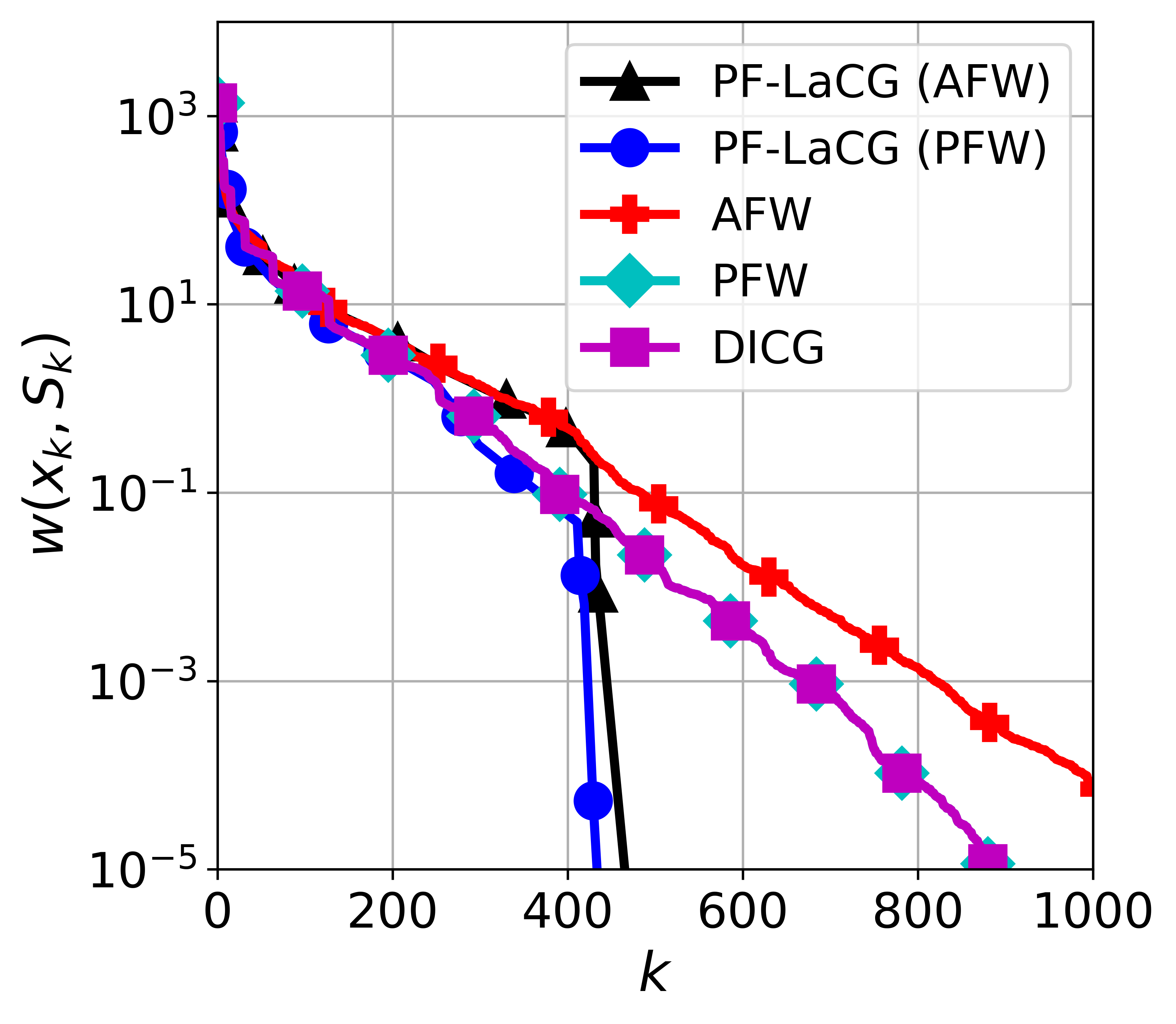} }\label{fig:simplexSWGIt:Appx}}%
        \hspace{\fill}
        \subfloat[$w(\vx_k, \cs_k)$ vs time (seconds)]{{\includegraphics[width=7.35cm]{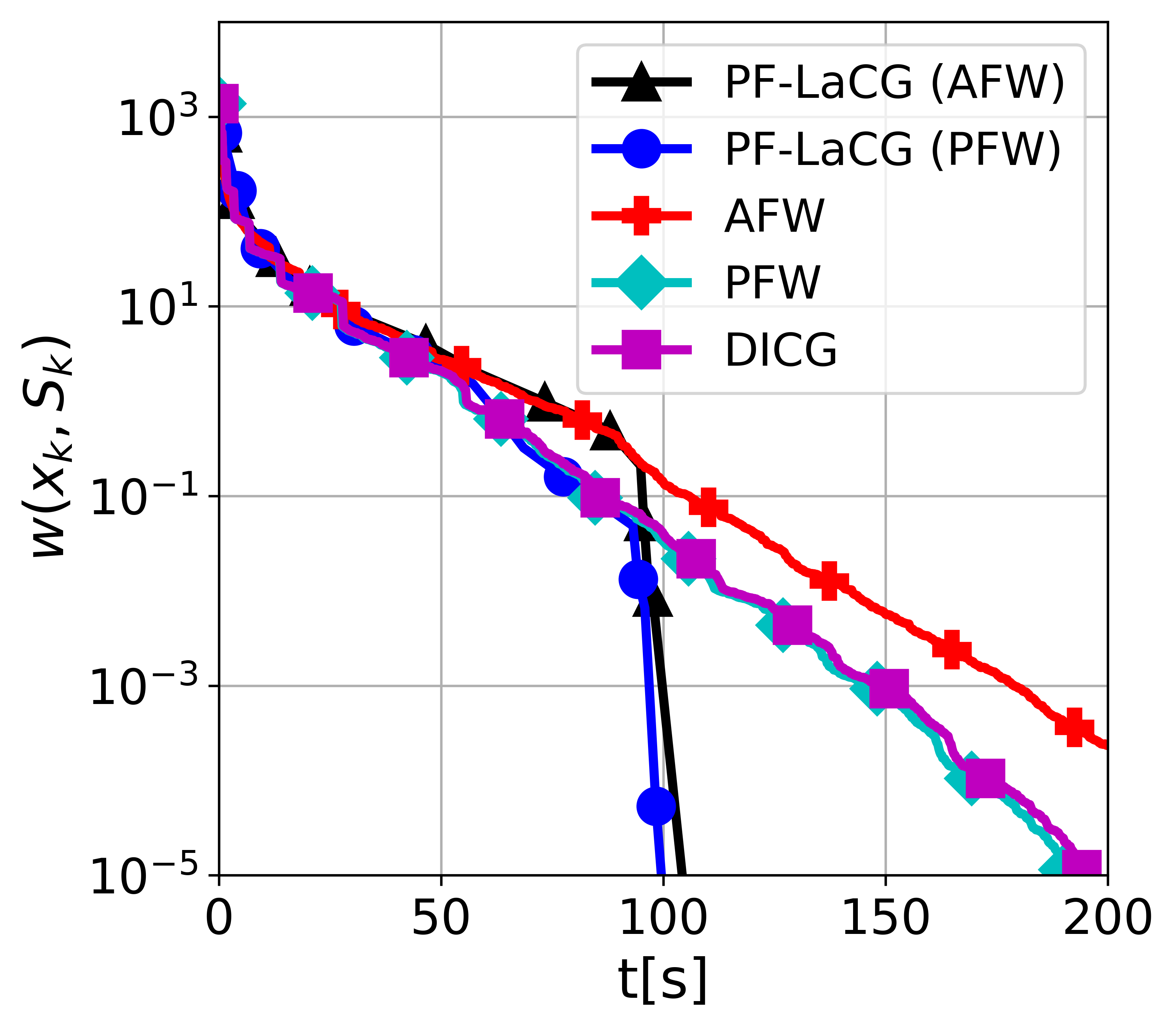} }\label{fig:simplexSWGTime:Appx}}%
        \hspace{\fill}
        \caption{\textbf{Strongly convex and smooth problem over the probability simplex: } Algorithm convergence in terms of $f(\vx_k) - f(\vx^*)$ and $w(\vx_k, \cs_k)$ versus iteration count $k$ and versus wall-clock time in seconds.}%
        \label{fig:simplex:Appx}%
    \end{figure*}

\begin{remark}[Using the structure of $\Delta_n$]

The structure of the unit probability simplex allows us to simplify the CG-variants employed in the comparison, as well as the PF-LaCG algorithm. Note that at each iteration the AFW and PFW algorithms (as well as the AFW algorithm) require maintaining an active set $\mathcal{S}_k \subseteq \vertex \left( \cx \right)$ that contains the vertices that give rise to the iterate $\vx_k$ as a convex combination. This active set is used to compute the \emph{away vertices} in the algorithms, denoted by $\vs_k$. In order to compute an away vertex one has to solve the linear optimization problem $\vs_k = \argmax_{\mathbf{v} \in \cs_k} \innp{\nabla f(\vx_k), \mathbf{v}}$, which typically requires looping through all the elements in $\cs_k$. This can become computationally expensive as $\card{\cs_k}$ grows. However, computing $\vs_k$ is extremely easy when $\vx_k \in \Delta_n$, as one need only look at the non-zero elements in $\vx_k$ to solve the linear program. Therefore no CG algorithm needs to explicitly maintain an active set when solving a problem over the probability simplex. This significantly reduces the running time of all the CG algorithms in our implementation. Note that the absence of an active set is one of the main advantages of the DICG algorithm \cite{garber2016linear} (aside from the notable fact that it is able to achieve a primal gap linear convergence guarantee that depends on the dimensionality of $\mathcal{F}\left( \vx^*\right)$, as opposed to $n$), which often allows it to outperform other CG variants when the feasible region $\cx$ is a $0-1$ polytope. In this example over the probability simplex, the PFW and DICG algorithms are equivalent, and we do not observe a significant advantage from using the DICG algorithm.

Moreover, if we focus on the problem shown in Eq.~\eqref{eq:appx:subproblem} we can make use of the fact that:
\begin{align*}
\argmin_{\vlambda \in \Delta_{\card{\cs}}} \innp{\vz, \mathcal{V} \vlambda} + \norm{\mathcal{V} \vlambda}^2 & = \argmin_{\vlambda \in \Delta_{\card{\cs}}} \innp{\mathcal{V}^T \vz, \vlambda} + \norm{ \vlambda}^2 \\
& = \argmin_{\vlambda \in \Delta_{\card{\cs}}} -2 \innp{-\frac{\mathcal{V}^T \vz}{2}, \vlambda} + \norm{ \vlambda}^2 + \norm{\frac{\mathcal{V}^T \vz}{2}}^2  \\
& = \argmin_{\vlambda \in \Delta_{\card{\cs}}} \norm{-\frac{\mathcal{V}^T \vz}{2} - \vlambda}^2.
\end{align*}
The last expression in the chain of equalities is nothing but the Euclidean projection of $-\frac{\mathcal{V}^T \vz}{2}$ onto $\Delta_{\card{\cs}}$, which as we have stated before, can be computed in closed-form efficiently. Moreover, we can recover the active set of any iterate by simply looking at the non-zero elements of the vector, allowing us to easily recover $\cs$ given $\vx$. We make use of this fact in our implementation of PF-LaCG over the probability simplex. This means that as the subproblems are solved to optimality in Algorithm~\ref{algo:ACC-iter} we do not require an accelerated algorithm, or a stopping criterion to solve the problems in Eq.~\ref{eq:appx:subproblem-barycentric}
\end{remark}

\subsection{PF-LaCG over Structured LASSO Regression Problems}

We give a brief motivation for the structured LASSO regression problem solved in Section~\ref{sec:compResults}. The \emph{Least Absolute Shrinkage and Selection Operator} (LASSO) \cite{tibshirani1996regression} is an immensely popular regression analysis that simultaneously performs variable selection and regularization to solve a linear regression problem. The formulation is intimately related to the \emph{Basis Pursuit Denoising} (BPD) \cite{chen1998atomic} problem in the signal processing community. One of the attractive properties of the LASSO is its ability to return sparse solutions that capture the variables that contribute the most towards producing an output. One of the domains in which CG-type algorithms have received attention is in sparse regression problems in physics \cite{carderera2020second}.

In many situations we can describe a physical system by the differential equations that govern the phenomenon. These equations allow us to compactly describe the current state of a physical system, or to predict its future state. However, in many situations we do not have any physics-informed differential equation models to describe a natural phenomenon, and we only have access to the state of the system at various times. Our goal then is to find this system of differential equations given some training data. That is, if we denote the state of the system at time $t$ by $\vx(t)$, we want to find $d \vx(t)/ d t = F(\vx(t)))$, where we assume that $F(\vx(t)))$ can be expressed as a linear combination of simple \emph{ansatz} functions (like polynomials) that belong to a dictionary $\mathcal{D} = \left\{ \psi_{i} \mid i\in [1,m] \right\}$, with $\psi_{i}: \rr^n \rightarrow \rr$. This allows us to write $F\left(\vx(t)\right) = \Xi^T \bm{\psi}(\vx(t))$ where $\Xi \in \rr^{m \times n}$ is a sparse matrix $\Xi = \left[\xi_1, \cdots, \xi_n \right]$ formed by column vectors $\xi_i \in \rr^m$ for $i \in [ 1, n ]$ and $\bm{\psi}(\vx(t)) = \left[ \psi_1(\vx(t)), \cdots, \psi_m(\vx(t)) \right]^T \in \rr^{m}$. Therefore the sparse matrix $\Xi$ allows us to reconstruct the differential equations that govern the system. If we are given a series of data points $\left\{\vx(t_i), d\vx(t_1)/dt\right\}_{i=1}^{r}$, in the absence of noise we will have:
\begin{align*}
\begin{bmatrix}
\vertbar & & \vertbar\\
d\vx(t_1)/dt & \cdots & d\vx(t_r)/dt\\
\vertbar & & \vertbar
\end{bmatrix} = 
\begin{bmatrix}
\horzbar & \xi_1 & \horzbar\\
 & \vdots & \\
\horzbar & \xi_n & \horzbar
\end{bmatrix} 
\begin{bmatrix}
\vertbar & & \vertbar\\
\bm{\psi}\left(\vx(t_1)\right) & \cdots & \bm{\psi}\left(\vx(t_r)\right)\\
\vertbar & & \vertbar
\end{bmatrix}.
\end{align*}
However, we are typically only given access to noise-corrupted measurements of $\left\{\vx(t_i)\right\}_{i=1}^{r}$, which means that we have access to some noisy $\left\{\vy(t_i)\right\}_{i=1}^{r}$, from which we have to estimate the derivatives $\left\{d \vy(t_i)/dt\right\}_{i=1}^{r}$. In the presence of this noise, and in pursuit of a sparse matrix $\Xi$, we can attempt to use the LASSO regression analysis for some $\alpha >0$, resulting in:
\begin{gather}
\argmin\limits_{\substack{ \norm{\Omega}_{1,1} \leq \tau  \\ \Omega \in \rr^{m \times n}}}  \norm{\dot{Y} - \Omega^T \Psi(Y)}^2_F. \label{eq:appx:LASSO}
\end{gather}
Where $\norm{\cdot}_{1,1}$ and $\norm{\cdot}_{F}$ represent the $\ell_{1,1}$ and Frobenius norm of a matrix,  and we have collected the data into matrices $\dot{Y} = \left[ d\vy(t_1)/dt,\cdots, d\vy(t_r) /dt\right] \in\rr^{n\times r}$ and  $\Psi\left(Y\right) = \left[ \bm{\psi}(\vx(t_1)),\cdots, \bm{\psi}(\vx(t_r))\right]\in\rr^{m\times r}$. Moreover, we can also try to impose physics-informed linear constraints on the problem shown in Eq.~\eqref{eq:appx:LASSO} to reflect symmetries or conservation properties in the system, in the hope that the learnt dynamics will generalize better to unseen data. This transforms the problem to:
\begin{gather}
\argmin\limits_{\Omega \in \cx}  \norm{\dot{Y} - \Omega^T \Psi(Y)}^2_F. \label{eq:appx:LASSO:constrained}
\end{gather}
with $\cx= \left\{ \Omega \in \rr^{m \times n} \mid \norm{\Omega}_{1,1} \leq \tau, \trace( A_l^T \Xi ) \leq b_l, l\in [ 1,L] \right\}$, $A_l \in \rr^{m\times n}$ and $b_l \in \rr$ for all $l \in [1,L]$, and where we have added $L$ additional linear constraints to the problem in Eq.~\eqref{eq:appx:LASSO} to reflect the underlying structure of the dynamical system that we want to impose. 

    \begin{figure*}[th!]
        \centering
        \hspace{\fill}
        \subfloat[$f(\vx_k) - f(\vx^*)$ vs iteration count]{{\includegraphics[width=7.35cm]{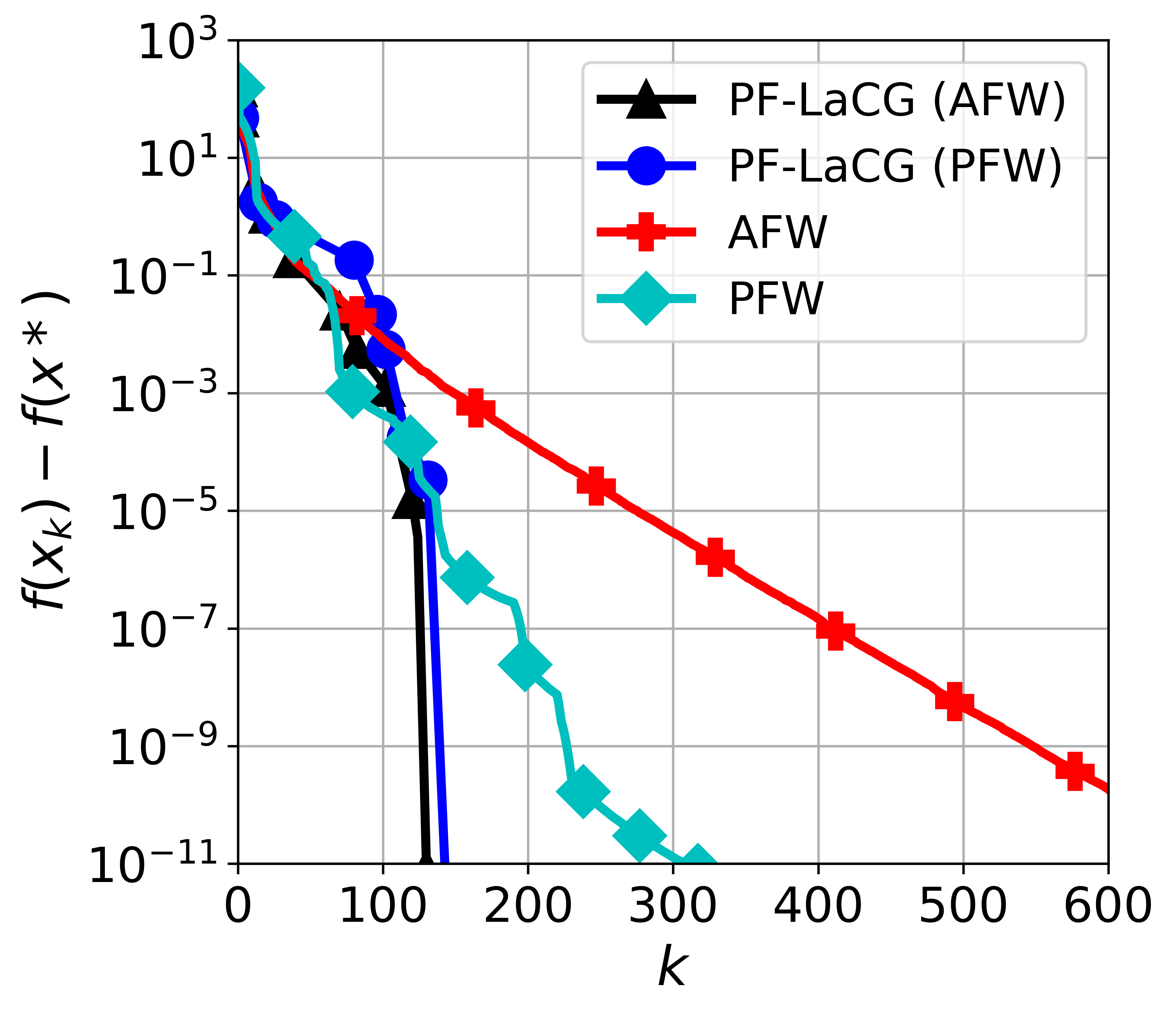} }\label{fig:LASSOPGIt:Appx}}%
        \hspace{\fill}
        \subfloat[$f(\vx_k) - f(\vx^*)$ vs time (seconds)]{{\includegraphics[width=7.35cm]{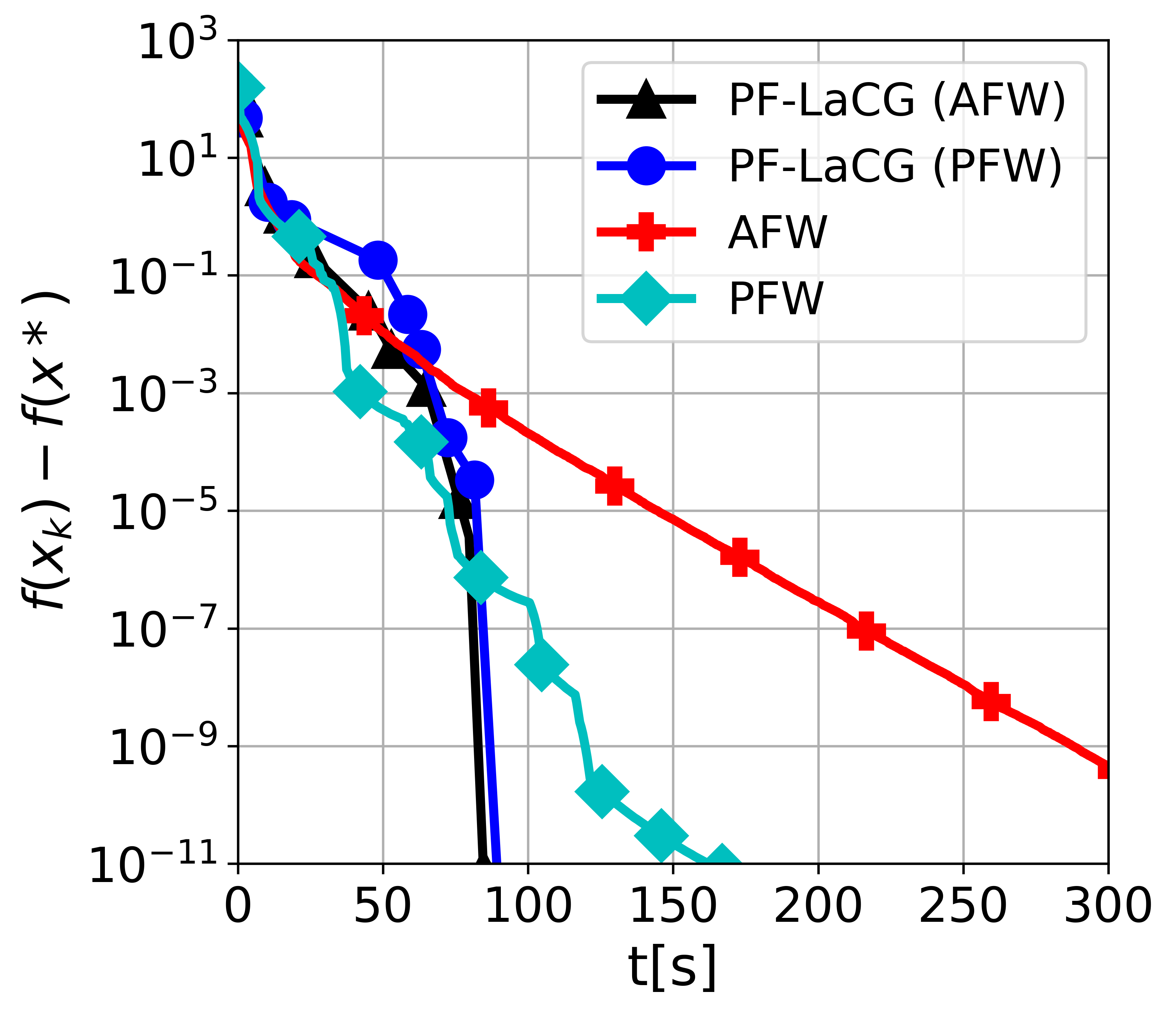} }\label{fig:LASSOPGTime:Appx}}%
        \hspace{\fill}
        
        \bigskip 
        
        \hspace{\fill}
        \subfloat[$w(\vx_k, \cs_k)$ vs iteration count]{{\includegraphics[width=7.35cm]{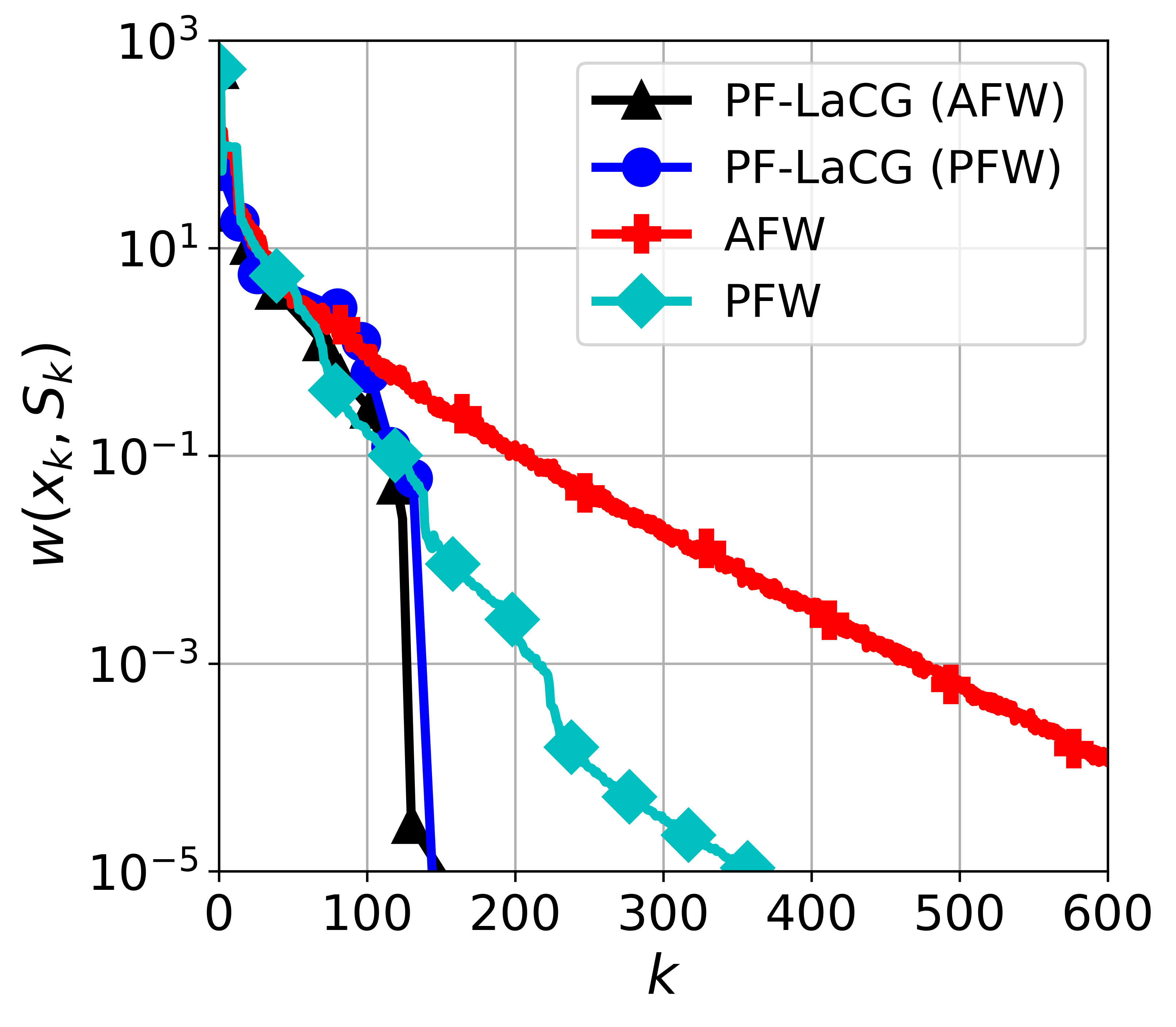} }\label{fig:LASSOSWGIt:Appx}}%
        \hspace{\fill}
        \subfloat[$w(\vx_k, \cs_k)$ vs time (seconds)]{{\includegraphics[width=7.35cm]{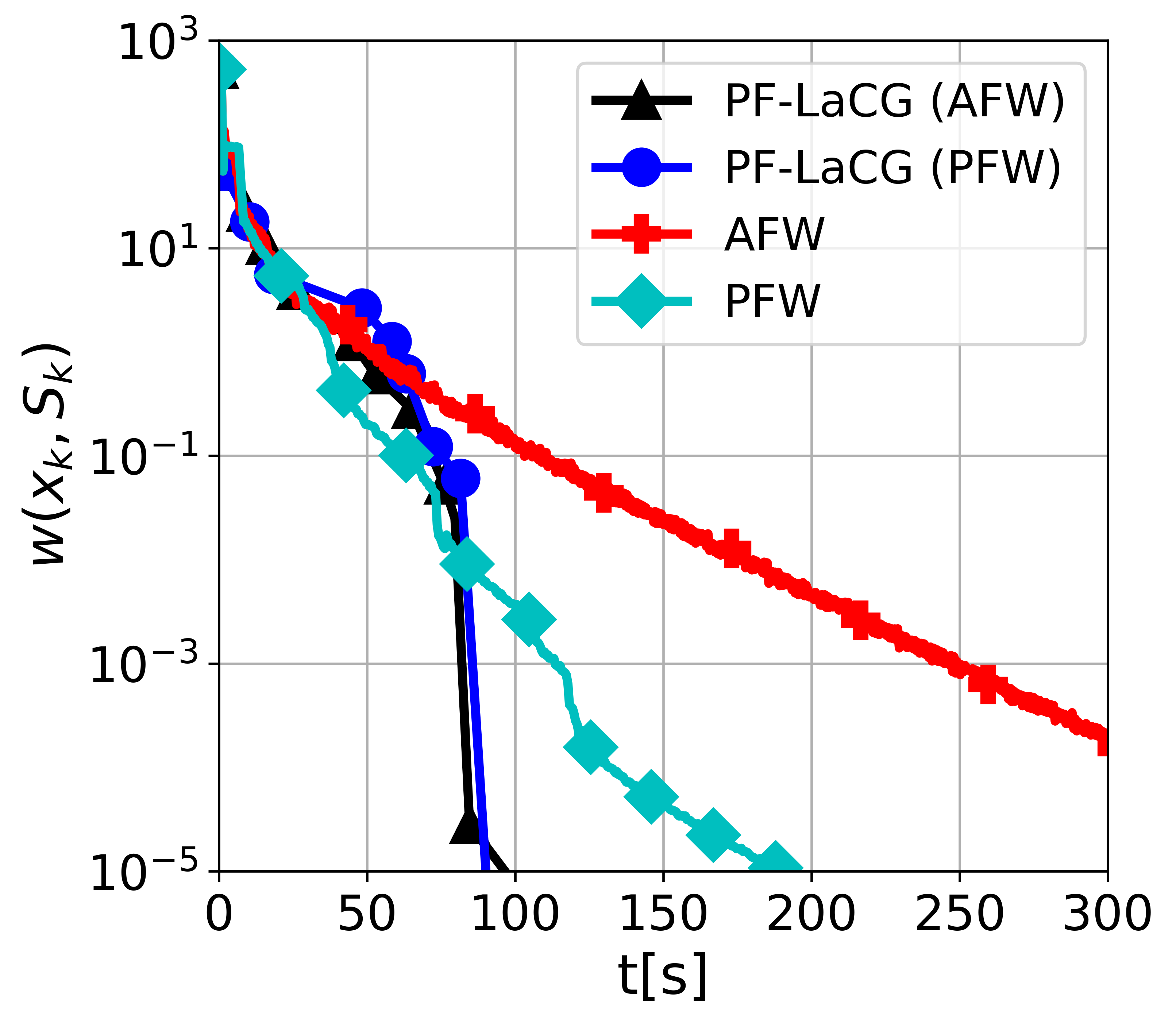} }\label{fig:LASSOSWGTime:Appx}}%
        \hspace{\fill}
        \caption{\textbf{Strongly convex and smooth problem over a structured LASSO domain: } Algorithm convergence in terms of $f(\vx_k) - f(\vx^*)$ and $w(\vx_k, \cs_k)$ versus iteration count $k$ and versus wall-clock time in seconds.}%
        \label{fig:LASSO:Appx}%
    \end{figure*}

We solve a stylized version of the problem in Eq.~\eqref{eq:appx:LASSO:constrained}, where the objective function is a quadratic $f(\vx) =  \vx^T \left( M^T M + \alpha \mathbf{1}_n \right)\vx/2 + \vb^T\vx$, where $M\in \rr^{n\times n}$ has entries sampled uniformly at random between $0$ and $1$, $\vb\in \rr^{n}$ has entries sampled uniformly at random from $0$ to $100$, $n = 1000$ and $\alpha = 100$. The resulting condition number of the function is $L/m = 250000$. The additional linear constraints we impose on the system are very similar to those used in \citet{carderera2021cindy}, with the exception here that we generate them at random. To generate the additional equality constraints, we sample $125$ pairs of distinct integers $(i,j)$ from $1\leq i,j \leq n$ without replacement, and we set $x_i = x_j$ for each pair, adding $125$ linear constraints. Lastly, the radius of the $\ell_1$ ball is set to $\tau = 1$. In this example, as the polytope is not a $0-1$ polytope, we cannot use the DICG algorithm of \citet{garber2016linear}. Despite the fact that we could resort to the more general decomposition invariant CG algorithm in \citet{bashiri2017decomposition}, we did not find it to be numerically comparable to the remaining CG algorithms tested in this section, and so have not included it in the comparison. The results obtained can be found in Figure~\ref{fig:LASSO:Appx}

\subsection{PF-LaCG over Constrained Matching Problems}

We also solve a matching-type problem over the intersection of the Birkhoff polytope and a set of additional linear constraints. The Birkhoff polytope in $\rr^{n \times n}$, also called the polytope of doubly-stochastic matrices, is the set of all square matrices whose columns and rows all sum up to 1. This polytope, with close ties to graph theory, is often used in matching problems. For example, if we interpret the rows of the matrix as workers, and the columns of the matrix as tasks that need to be completed, we can view the matrix element $A_{i,j}$ on the $i^{\text{th}}$ row and the $j^{\text{th}}$ column as being either $0$ or $1$, depending on if the $j^{\text{th}}$ task has been assigned to the $i^{\text{th}}$ worker (if $A_{i,j} = 1$), or if the $j^{\text{th}}$ task has not been assigned to the $i^{\text{th}}$ worker (if $A_{i,j} = 0$).

    \begin{figure*}[th!]
        \centering
        \hspace{\fill}
        \subfloat[$f(\vx_k) - f(\vx^*)$ vs iteration count]{{\includegraphics[width=7.35cm]{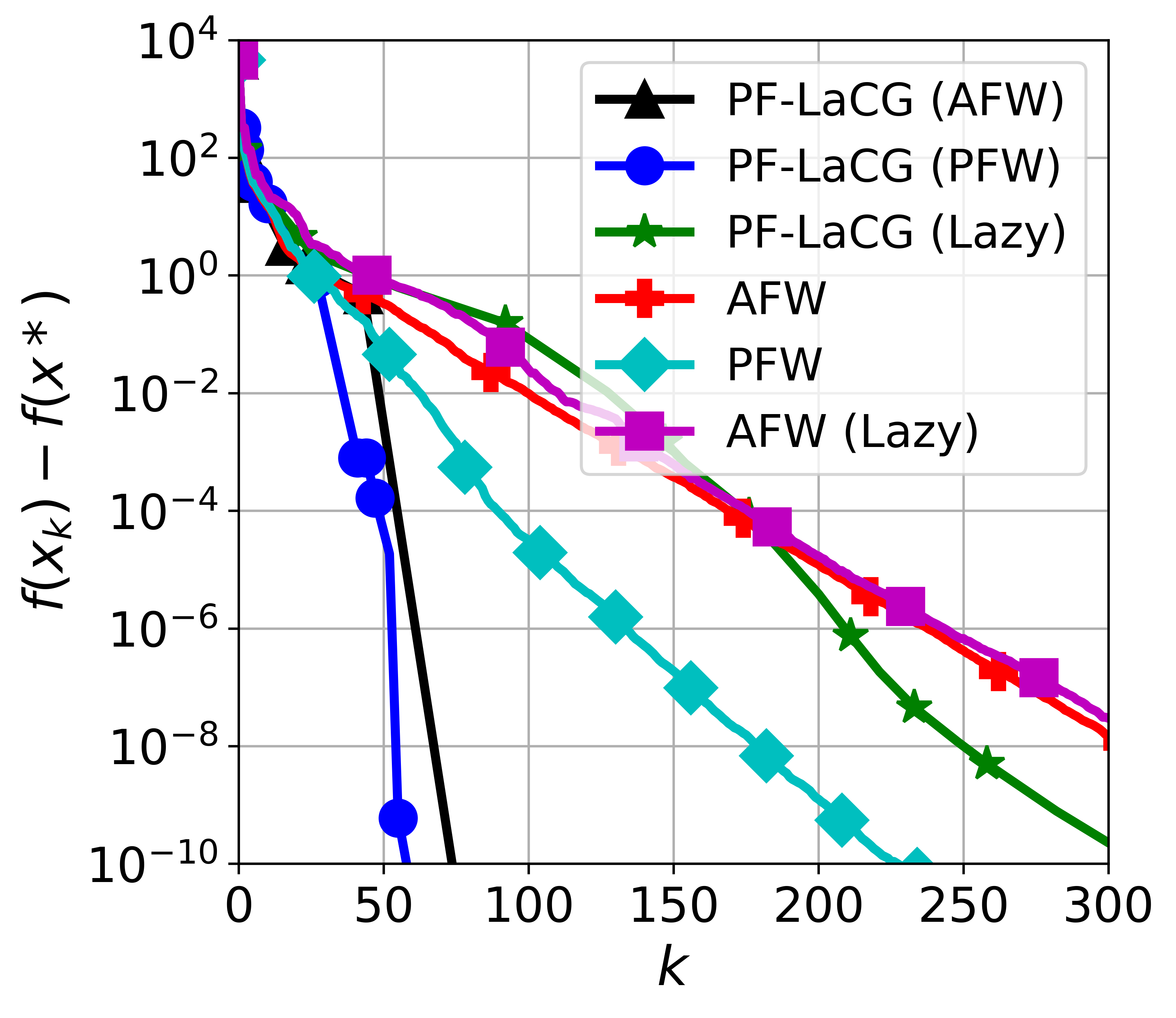} }\label{fig:matchingPGIt:Appx}}%
        \hspace{\fill}
        \subfloat[$f(\vx_k) - f(\vx^*)$ vs time (seconds)]{{\includegraphics[width=7.35cm]{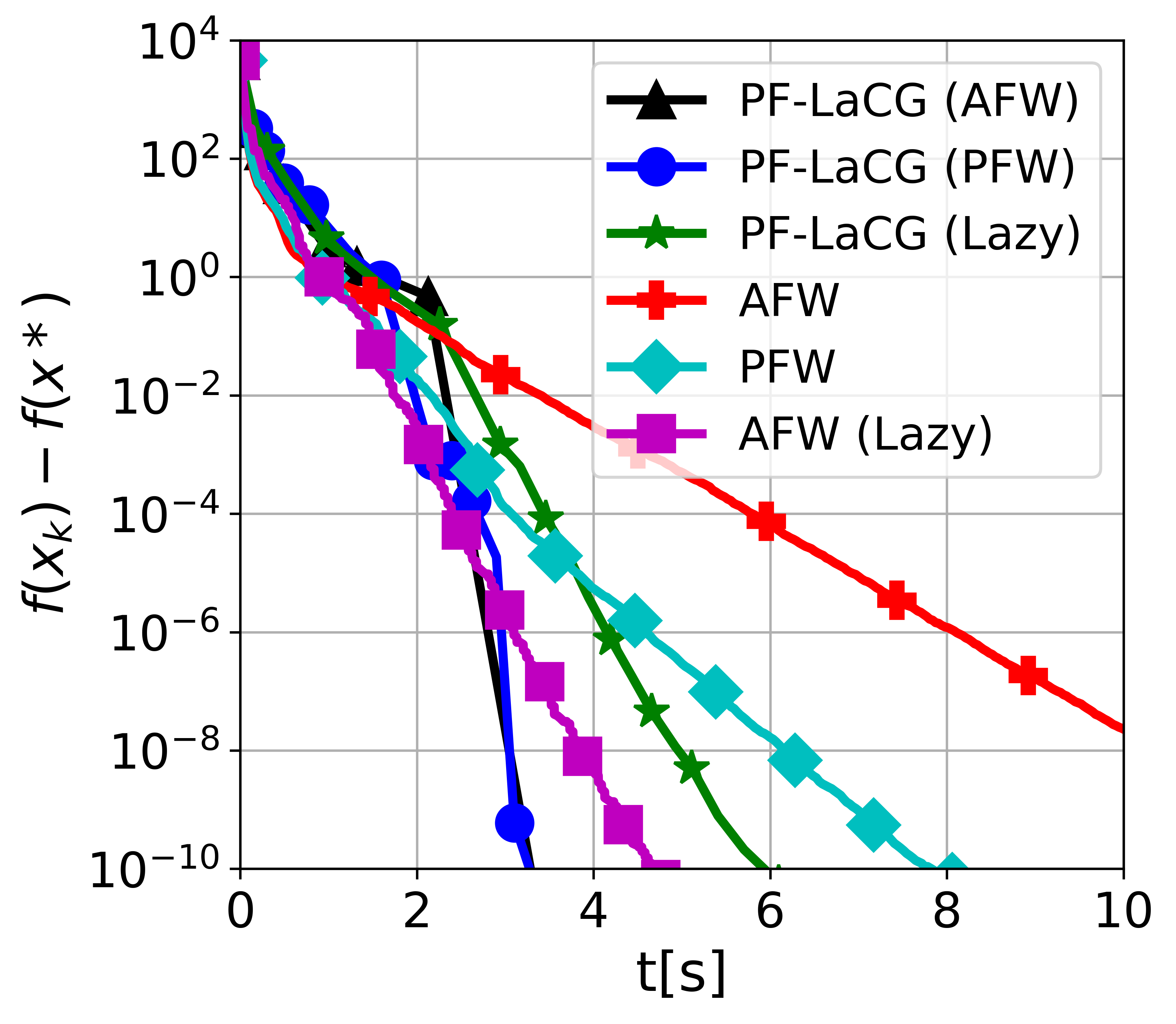} }\label{fig:matchingPGTime:Appx}}%
        \hspace{\fill}
        
        \bigskip 
        
        \hspace{\fill}
        \subfloat[$w(\vx_k, \cs_k)$ vs iteration count]{{\includegraphics[width=7.35cm]{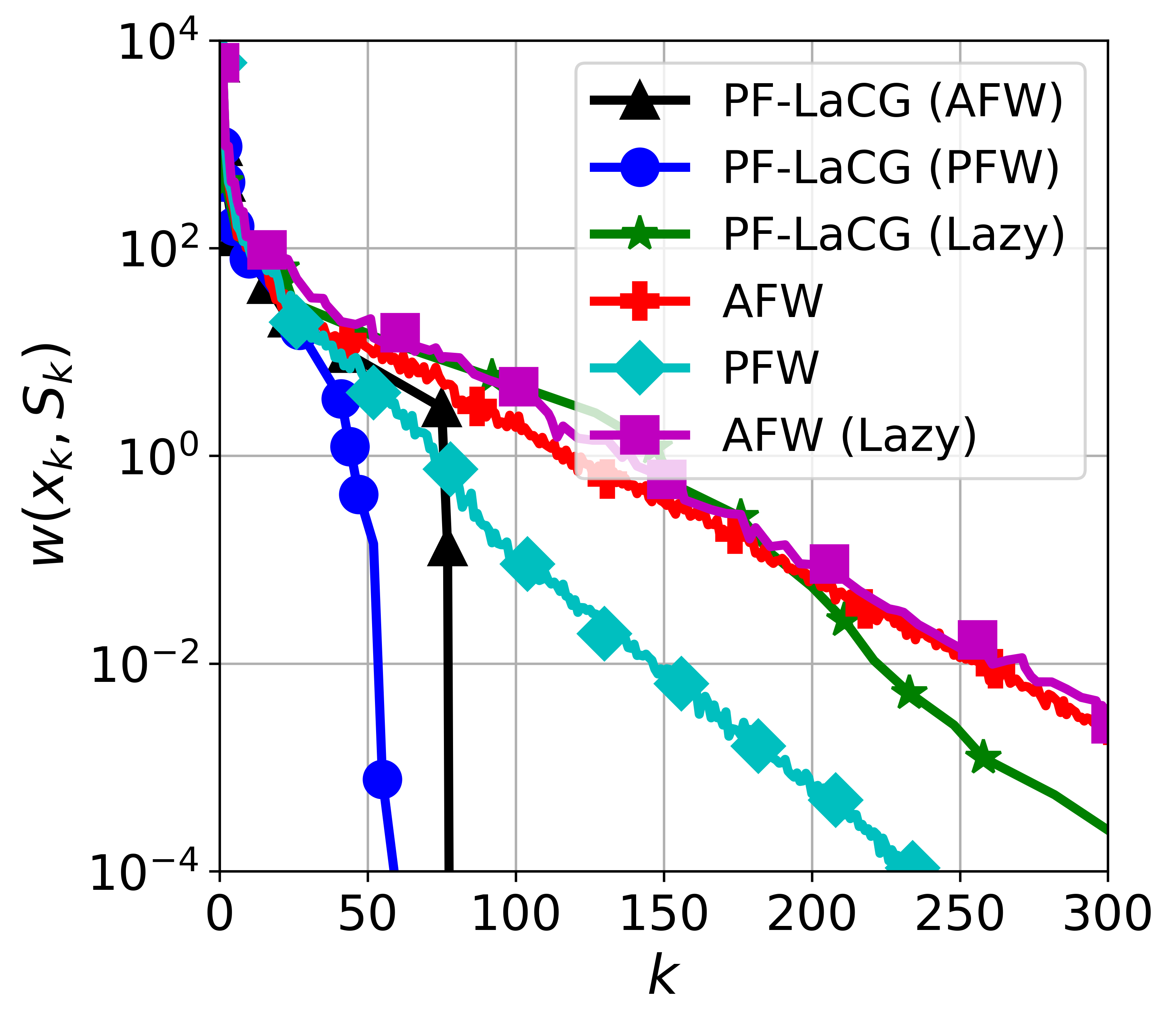} }\label{fig:matchingSWGIt:Appx}}%
        \hspace{\fill}
        \subfloat[$w(\vx_k, \cs_k)$ vs time (seconds)]{{\includegraphics[width=7.35cm]{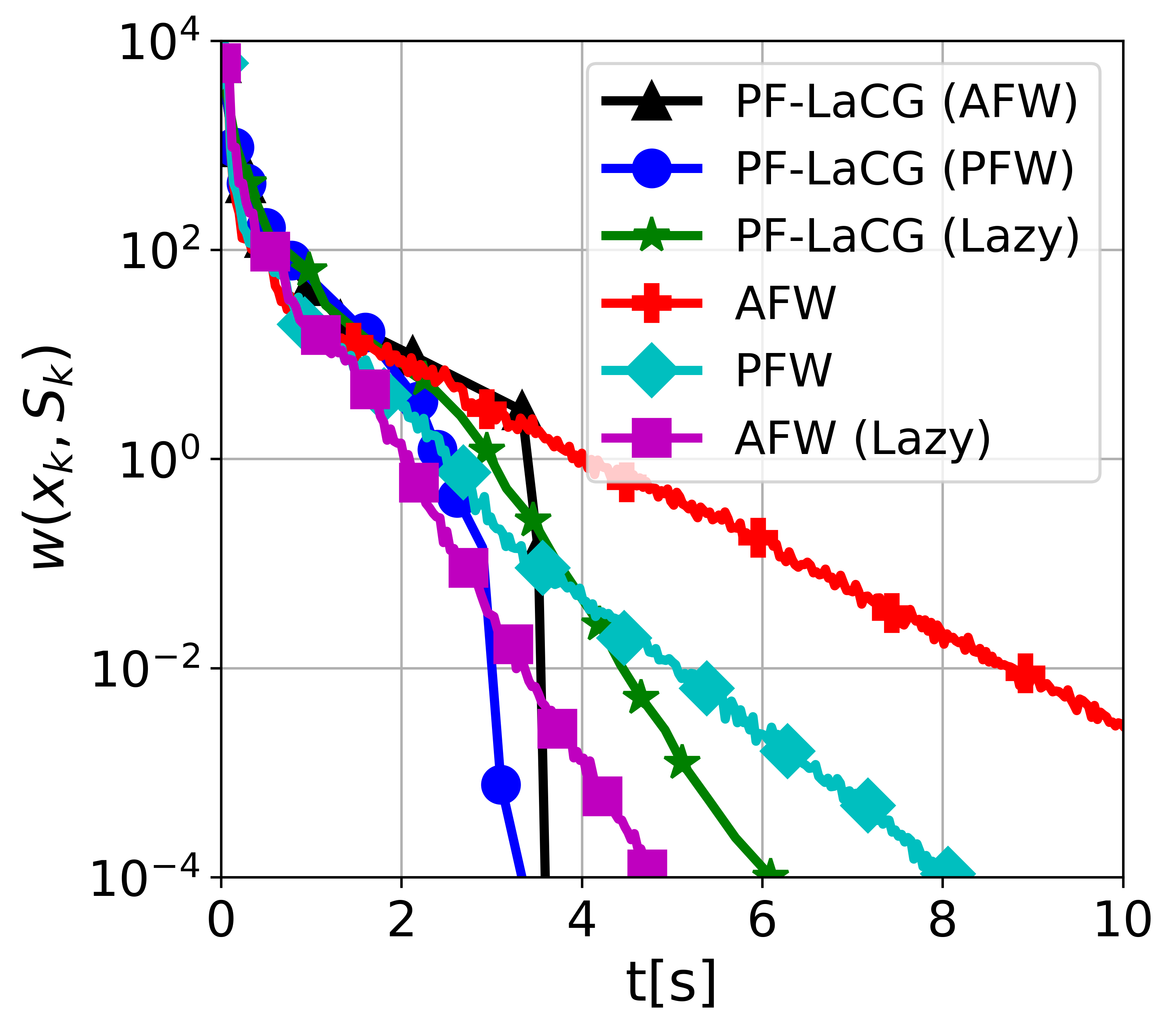} }\label{fig:matchingSWGTime:Appx}}%
        \hspace{\fill}
        \caption{\textbf{Strongly convex and smooth matching problem over a structured Birkhoff polytope: } Algorithm convergence in terms of $f(\vx_k) - f(\vx^*)$ and $w(\vx_k, \cs_k)$ versus iteration count $k$ and versus wall-clock time in seconds.}%
        \label{fig:matching:Appx}%
    \end{figure*}

 We minimize a quadratic cost function over a Birkhoff polytope of dimension $400$, that is with $n = 20$, where the objective function has the form $f(\vx) =  \vx^T \left( M^T M + \alpha \mathbf{1}_n \right)\vx/2 + \vb^T\vx$ with $\alpha = 1$, and where $M\in \rr^{n\times n}$ and $\vb\in \rr^{n}$ have entries sampled uniformly at random between $0$ and $1$. This results in an objective function with a condition number of $L/m = 100000$. In order to make the problem more challenging to solve, we impose an additional set of linear constraints on the learning problem. Otherwise if we were solving the problem over the Birkhoff polytope in $\rr^{n \times n}$ we could efficiently solve linear minimization problems over the aforementioned polytope using the Hungarian algorithm, with complexity $\mathcal{O}(n^3)$. The additional constraints that we impose represent either worker-task assignments that are not permitted, or capacity constraints that represent the maximum fractional matching that we can have between a given task and a worker. In order to generate these extra constraints we sample $80$ integers $i$ from $1\leq i \leq n^2$ without replacement, and we set $x_i = 0$ for the first $40$ integers (to represent that certain matchings are not possible), and $x_i \leq 0.5$ for the remaining $40$ integers to represent a maximum fractional matching. As in the previous example, we did not find the algorithm in \citet{bashiri2017decomposition} to have numerical performance comparable to the other algorithms tested, and so have not included the algorithm in the comparison. The results from the comparison can be seen in Figure~\ref{fig:matching:Appx}.

\subsection{Performance Comparison with 2 Cores} \label{Section:Appx:2coreperformance}

In Figure~\ref{fig:Appx_2core} we show a performance comparison for the different algorithms where we use 2 CPU cores for the CG-variants, and 2 cores for the PFLaCG algorithms (1 core for the AGD algorithm, and 1 core for the AFW algorithm). For the sparse regression problem over the structured LASSO feasible region and the matching problem over the structured Birkhoff polytope we did not observe a large increase in performance for the CG-variants by using 2 cores instead of 1. For example, for the sparse regression problem the PFW algorithm went from reaching a value of $w(\vx, \cs)$ below $10^{-5}$ in 188 seconds with 1 core, to reaching it in 166 seconds with 2 cores, which constitutes a 10$\%$ increase in performance with respect to wall-clock time. Similarly, the matching problem went from reaching a value of $w(\vx, \cs)$ below $10^{-4}$ in 4.75 seconds with 1 core, to reaching it in 4.63 seconds with 2 cores, which is a meager 2$\%$ increase in performance with respect to wall-clock time. On the other hand, for the probability simplex example we can see that the algorithm reaches a value of $w(\vx, \cs)$ below $10^{-5}$ in 190 seconds with 1 core, and it reaches that value in 96 seconds using 2 cores, which constitutes a 50$\%$ increase in performance with respect to wall-clock time.

    \begin{figure*}[th!]
        \centering
        \hspace{\fill}
        \subfloat[Probability Simplex]{{\includegraphics[width=4.85cm]{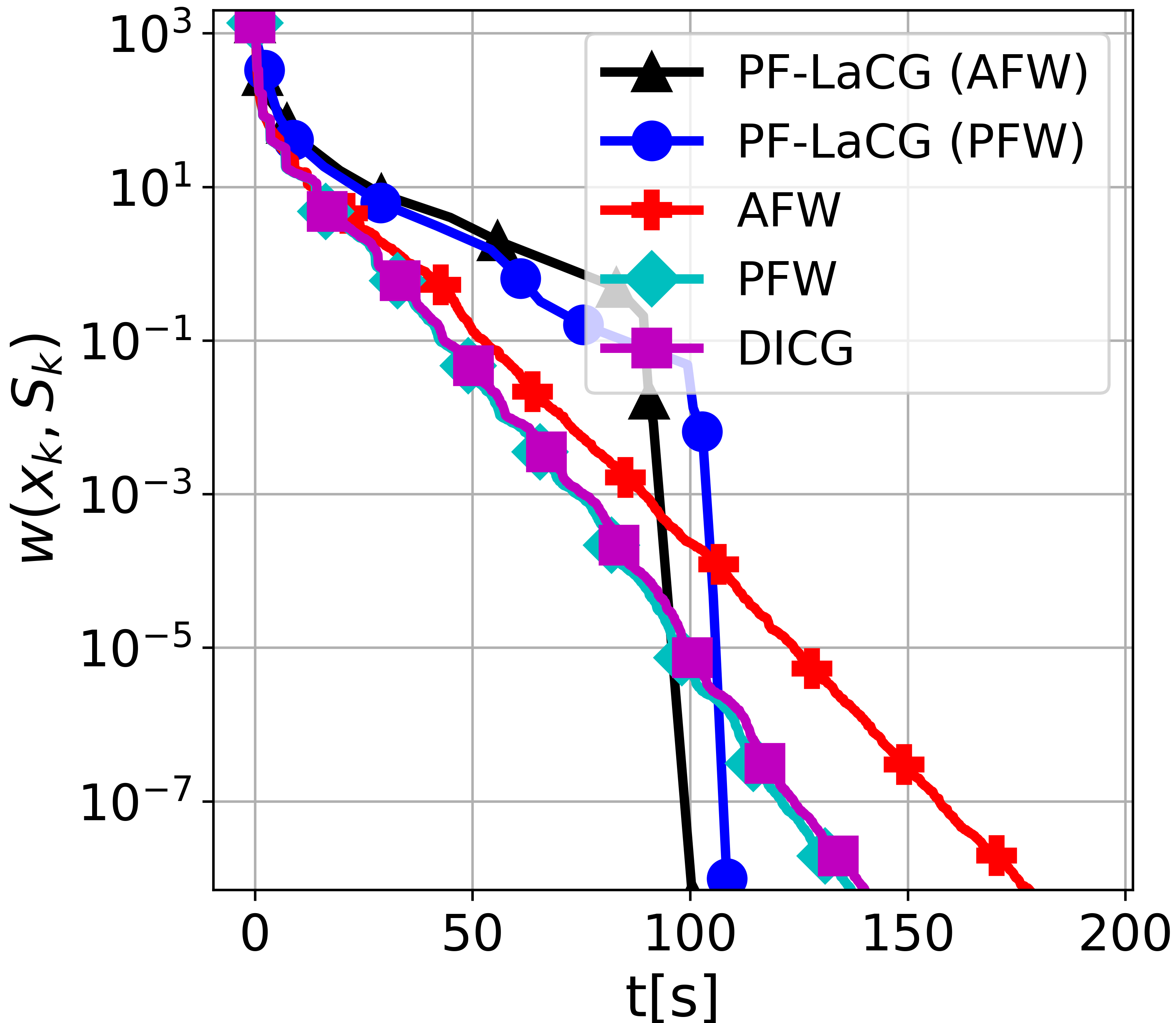} }\label{fig:SimplexSWGIt:Appx_2core}}%
        \hspace{\fill}
        \subfloat[Structured LASSO]{{\includegraphics[width=4.85cm]{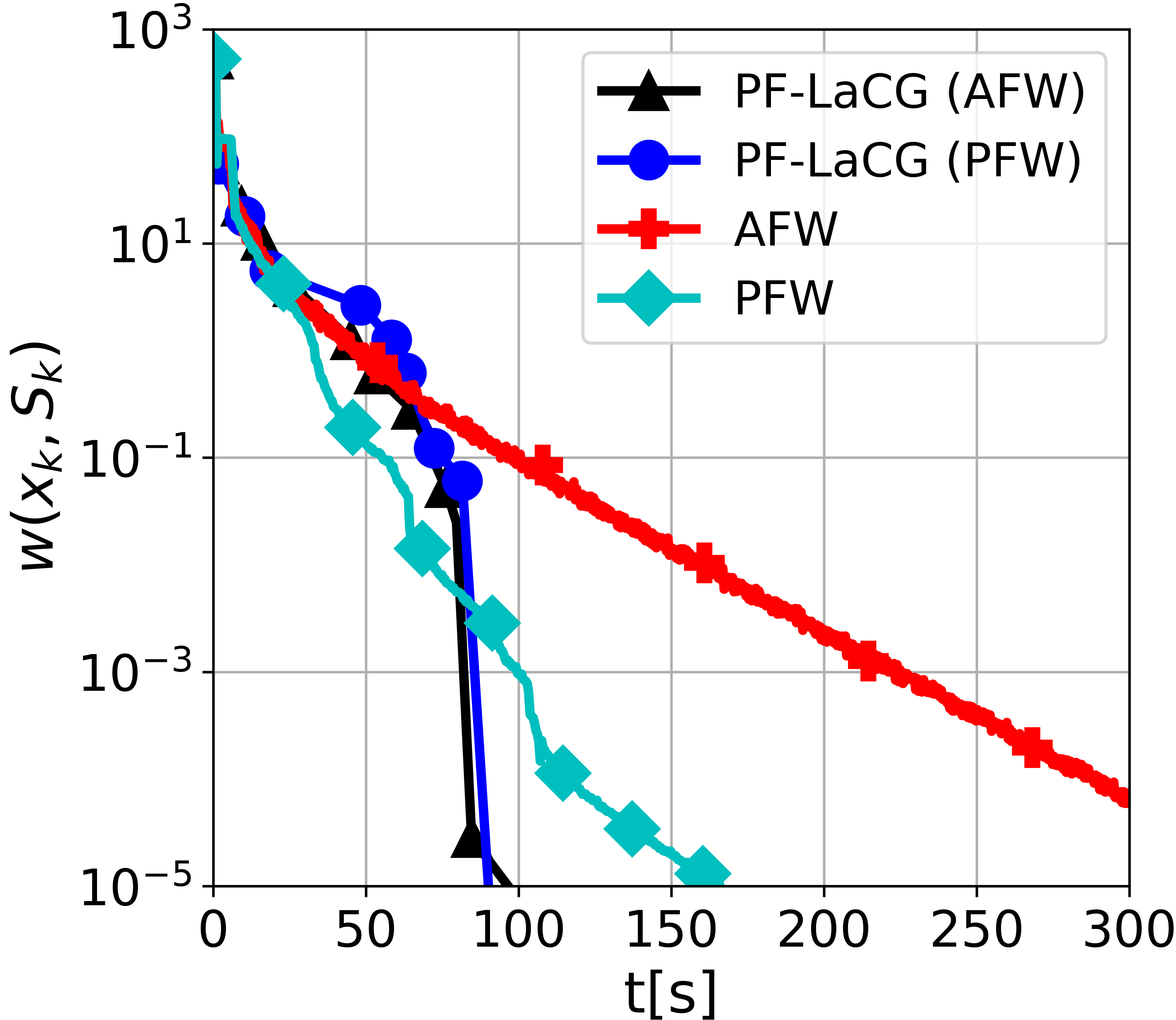} }\label{fig:LASSOSWGTime:Appx_2core}}%
        \hspace{\fill}
        \subfloat[Structured Matching]{{\includegraphics[width=4.85cm]{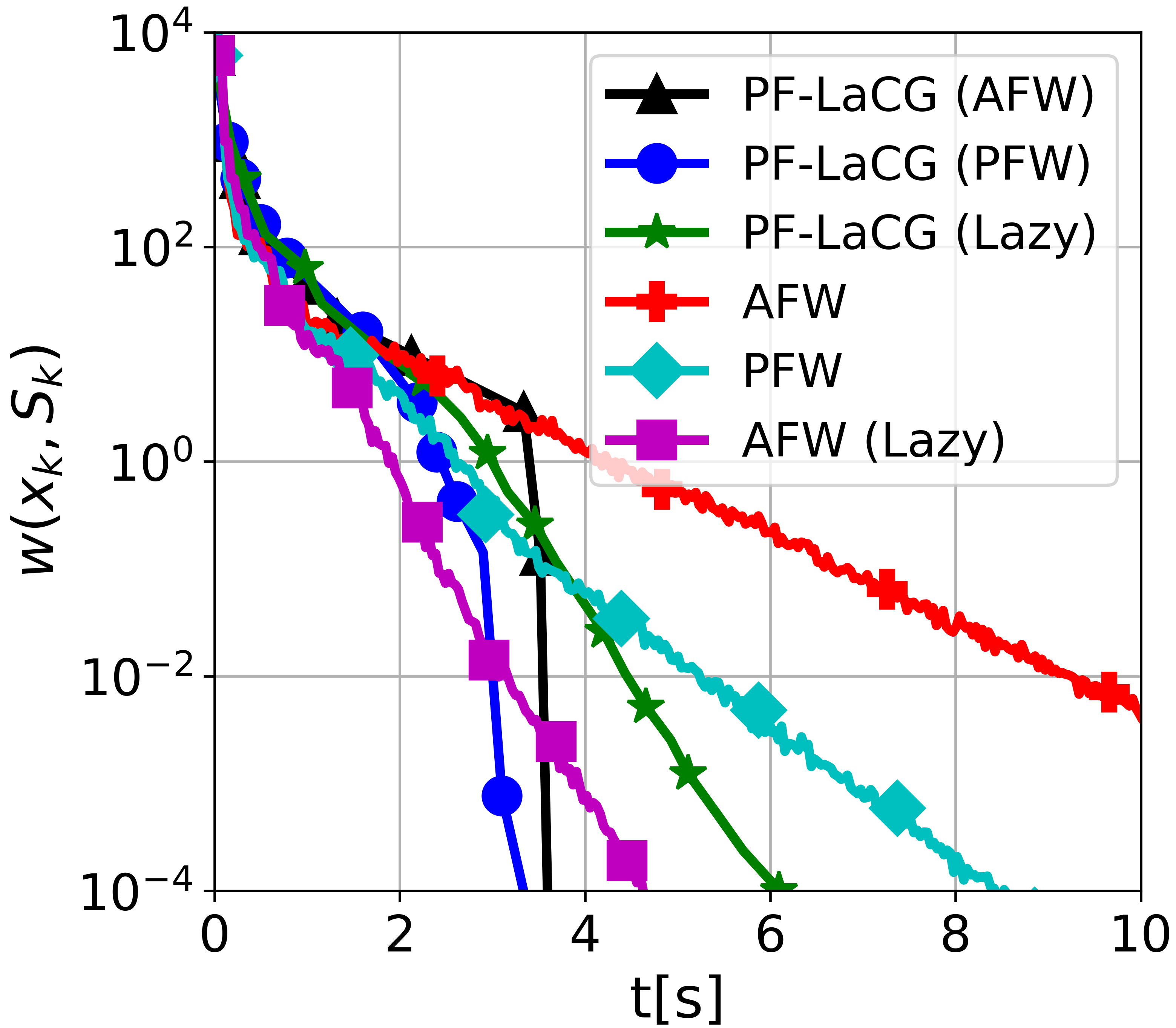} }\label{fig:BirkhoffSWGTime:Appx_2core}}%
        \hspace{\fill}
        \caption{\textbf{Performance comparison using 2-cores for CG-variants: } Comparison of $w(\vx_k, \cs_k)$  vs time (seconds) for the unit probability simplex experiment in Figure~\protect\subref{fig:SimplexSWGIt:Appx_2core}, for the structured LASSO problem in Figure~\protect\subref{fig:LASSOSWGTime:Appx_2core} and for the structured matching problem in Figure~\protect\subref{fig:BirkhoffSWGTime:Appx_2core}.}%
        \label{fig:Appx_2core}%
    \end{figure*}

Regardless of if we use 1 core or 2 cores for the CG variants in the comparison, we still obtain faster convergence in wall-clock time when using the PFLaCG algorithm. Finally, note that the convergence in terms of iteration count is independent on the number of cores used in the experiment.

\end{document}